\documentclass[10pt]{amsart}
\usepackage{latexsym,amsmath,amscd,amssymb,epsfig,verbatim}
\usepackage[margin=1in,letterpaper,portrait]{geometry}

\RequirePackage{color}
\definecolor{myred}{rgb}{0.75,0,0}
\definecolor{mygreen}{rgb}{0,0.5,0}
\definecolor{myblue}{rgb}{0,0,0.65}

\RequirePackage{ifpdf}
\ifpdf
  \IfFileExists{pdfsync.sty}{\RequirePackage{pdfsync}}{}
  \RequirePackage[pdftex,
   colorlinks = true,
   urlcolor = myblue, % \href{...}{...} external (URL)
   citecolor = mygreen, % \cite{...}
   linkcolor = myred, % \ref{...} and \pageref{...}
   pagebackref,
   bookmarksopen=true]{hyperref}
\else
  \RequirePackage[hypertex]{hyperref}
\fi

\theoremstyle{plain}
  \newtheorem{theorem}{Theorem}[section]
  \newtheorem{proposition}[theorem]{Proposition}
  \newtheorem{lemma}[theorem]{Lemma}
  \newtheorem{corollary}[theorem]{Corollary}
  \newtheorem{conjecture}[theorem]{Conjecture}
\theoremstyle{definition}
  \newtheorem{definition}[theorem]{Definition}
  \newtheorem{example}[theorem]{Example}

 \theoremstyle{remark}
  \newtheorem{remark}[theorem]{Remark}

\numberwithin{equation}{section}

\def\NN{{\mathbb N}}
\def\ZZ{{\mathbb Z}}
\def\QQ{{\mathbb Q}}

\def\FF{{\mathbb F}}

\def\one{\mathbf{1}}

\def\crossterms{\mathrm{c}}
\def\lcm{\mathrm{lcm}}

\def\sBC{\left\lfloor \frac{\ell(\lambda)}{2} \right\rfloor}
\def\Cat{\mathrm{Cat}}
\def\Nar{\mathrm{Nar}}
\def\Krew{\mathrm{Krew}}

\def\Sym{\mathrm{Sym}}

\def\rank{\mathrm{rank}}

\def\ar{\hat{A}(e)}
\newcommand{\aF}{{\bar{\mathbb{F}}}}

\newcommand{\dualp}{\lambda'}
\newcommand{\countp}{\mu}
\newcommand{\orbit}{\mathcal O}

\newcommand{\0}{\orbit}
\newcommand{\krew}{\Krew(\Phi,\0,m;q)}

\newcommand{\CM}{\mathbb C}

\newcommand{\BCDexponent}
   {\psi(n,m,\lambda)}
\def\symm{\mathfrak{S}}
\def\ggg{\mathfrak{g}}
\def\levi{\mathfrak{l}}

\def\BBB{{\mathcal{B}}}
\def\Par{{\mathcal{P}}}

\newcommand{\qbin}[3]{\left[ \begin{matrix} #1 \\ #2 \end{matrix} \right]_{#3}}
\newcommand{\binomial}[2]{\left( \begin{matrix} #1 \\ #2 \end{matrix} \right)}

\begin{document}

\title[Weyl group $q$-Kreweras numbers and cyclic sieving]
{Weyl group $q$-Kreweras numbers and cyclic sieving}

\author{Victor Reiner}
\address{School of Mathematics\\
University of Minnesota\\
Minneapolis, MN 55455}
\email{reiner@math.umn.edu}

\author{Eric Sommers}
\address{Dept. of Mathematics and Statistics\\
University of Massachusetts- Amherst\\
Amherst, MA 01003}
\email{esommers@math.umass.edu}

\thanks{First author supported by 
NSF grant DMS-1001933, 
second author supported by NSA grant H98230-11-1-0173 and 
by a National Science Foundation Independent Research and Development plan.}



\begin{abstract}
The paper concerns a definition for $q$-Kreweras numbers
for finite Weyl groups $W$, refining the $q$-Catalan numbers for $W$, and arising
from work of the second author.  We give explicit formulas
in all types for the $q$-Kreweras numbers.  In the classical types $A, B, C$, 
we also record formulas for the $q$-Narayana numbers and in the process show that the formulas 
depend only on the Weyl group (that is, they coincide in types $B$ and $C$).   
In addition we verify that in the classical types 
$A,B,C,D$ that the $q$-Kreweras numbers 
obey the expected cyclic sieving phenomena when evaluated at
appropriate roots of unity.
\end{abstract}

\maketitle

%\tableofcontents

%%%%%%%%%%%%%%%%%%%%%%%%%%%%%%%%%%%%%%%%%%%%%%%%%%%%%%%%%%%%%%%
\section{Introduction}
\label{intro-section}
%%%%%%%%%%%%%%%%%%%%%%%%%%%%%%%%%%%%%%%%%%%%%%%%%%%%%%%%%%%%%%%
This paper examines polynomials in $q$,
generalizing what are sometimes 
called {\it Kreweras} %and {\it Narayana} 
numbers, 
as refinements of the {\it Catalan numbers}.
The $q$-Kreweras numbers arose as
by-products of work of the second author \cite{Sommers2}
on nilpotent orbits in 
a simple Lie algebra $\ggg$ over $\CM$ under the action of the associated connected algebraic group $G$.
Their definition, using the Lusztig-Shoji algorithm in Springer theory,
is reviewed in Section~\ref{general-Kreweras-definition-section} below.

More specifically, one has polynomials in $q$ defined for 
certain positive integral parameters $m$ (see below) and 
for each nilpotent orbit $\0$ in $\ggg$ and 
each {\it local system} on $\0$ that arises in the Springer correspondence.
Let $\Phi$ be the root system of $\ggg$ relative to a Cartan subalgebra $\mathfrak h$.
What we call here the 
{\it $q$-Kreweras numbers} $\krew$ correspond
to the {\it trivial} local system on $\0$.
In this paper we show three new
results about the polynomials $\krew$, namely
%and $\Nar(\Phi,m,k;q)$, namely
Theorems~\ref{q-Kreweras-formula-theorem},
%~\ref{q-Narayana-formula-theorem},
~\ref{divisibility-and-nonnegativity-theorem_all}, 
~\ref{CSP-theorem} below, which will be proven in
Sections ~\ref{proof-of-Kreweras-formulas-section},
%~\ref{proof-of-Narayana-formulas-section},
~\ref{proof-of-divisibility-and-nonnegativity-section},
~\ref{proof-of-CSP-theorem-section}, respectively.
In types $A, B, C$, we also discuss a definition of 
$q$-Narayana numbers $\Nar(\Phi,m,k;q)$, which are sums of $q$-Kreweras numbers depending on a statistic on $\0$,
and establish Theorem \ref{q-Narayana-formula-theorem} in 
Section \ref{proof-of-Narayana-formulas-section}.
We will take up the $q$-Narayana numbers for other types in a sequel paper.

The further parameter $m$ in the definition 
of $\krew$ amd $\Nar(\Phi,m,k;q)$
 is a positive integer 
%satisfying the compatibility condition of being 
that is {\it very good for $\Phi$}: 
%in the terminology of \cite[\S 5]{Sommers2}; 
this amounts to $m$ being relatively prime to the {\it Coxeter number} $h$ in 
types $A,E,F,G$, and the (weaker) condition of $m$ being
{\it odd} in the classical types $B,C,D$.

Let $W$ be the Weyl group of $\Phi$.  
Since $\ggg$ is simple, $W$ acts irreducibly on $\mathfrak h$, which is called the 
reflection representation of $W$ and denoted by $V$.  Let $r = \dim V$, the rank of $\ggg$. 
%:= \mathfrak h$ of dimension $r$. % \cong \CM^r$.
Let $d_1 \leq \ldots \leq d_r$ be the degrees of any 
set of fundamental invariants for the action of $W$ on
the polynomials $S:=\Sym(V^*)$ on $V$.
The Coxeter number $h$ of $\Phi$ is equal to $d_r$.
Define 
\begin{equation}
\label{W-q-Catalan-definition}
\Cat(W,m;q):=\prod_{i=1}^r \frac{[m-1+d_i]_q}{[d_i]_q}
\end{equation}
where $[n]_q := 1+q+q^2+\cdots+q^{n-1}$.
This is known to be a polynomial in $\NN[q]$ for all very good $m$.
Results from \cite[\S 5.3]{Sommers2} imply a summation formula
\begin{equation}
\label{q-Kreweras-sum-to-q-Cat}
\Cat(W,m;q)=\sum_{\0} \krew
\end{equation}
as $\0$ runs through the nilpotent orbits in 
$\ggg$, which is a generalization of known results for the specialization at $q=1$, as we now recall.
%See Definition~\ref{q-Kreweras-definition} and
%\eqref{original-form-of-q-Kreweras-sums-to-q-Catalan}
%below.     ** revise in light of moving stuff to intro **

In type $A_{n-1}$, so that $m$ is
very good if $\gcd(m,n)=1$, this $\Cat(W,m;q)$ is the 
{\it rational $q$-Catalan number} 
considered, for example, 
by Armstrong, Rhoades and Williams \cite{ArmstrongRhoadesWilliams}.  
At the specializations $m = h+1 = n+1$ and $q=1$, these become 
the {\it Catalan numbers} 
\begin{equation}
\label{type-A-Catalan}
C_n :=\frac{1}{n+1}\binom{2n}{n}
\end{equation}
which have 
a plethora of combinatorial interpretations 
(see Stanley \cite{Stanley-Catalan-book} and
\cite[Exer. 6.19]{Stanley-EC2}),
some restricting to interpretations of the successive refinements
by the {\it Narayana} $N(n,k)$ and {\it Kreweras numbers} $\Krew(\lambda)$:
$$
\begin{array}{rcccccl}
C_n &=&\displaystyle \sum_{k=1}^n N(n,k)
    & \text{ where }
    & N(n,k) 
    &:= &\displaystyle \frac{1}{k}\binom{n-1}{k-1}\binom{n}{k-1} \\
N(n,k)&=&\displaystyle \sum_{\substack{\lambda \in \Par(n):\\ \ell(\lambda)=k}} \Krew(\lambda)
      & \text{ where }
      &\Krew(\lambda) 
      &:=& \displaystyle \frac{1}{n+1} \binom{n+1}{n-k,\mu_1(\lambda),\mu_2(\lambda), \ldots, \mu_n(\lambda)}
\end{array}
$$
Here $\Par(n)$ is the set of all partitions 
$
\lambda = (\lambda_1 \geq \lambda_2 \geq \ldots \geq \lambda_{\ell} > 0)
$
of the number $n$ and $\mu_j(\lambda)$ denotes the multiplicity of the number
$j$ among the parts $\{ \lambda_i \}$.
The number of parts $\ell=:\ell(\lambda)$ of $\lambda$ is called the {\it length} of
$\lambda$.

Kreweras \cite{Kreweras} originally interpreted
$C_n, N(n,k), \Krew(\lambda)$ in terms of the set $NC(n)$
of {\it noncrossing set partitions}
of $\{1,2,\ldots,n\}$ arranged circularly in the plane; that is, partitions 
for which the convex hull of the blocks are pairwise disjoint.
The Catalan number $C_n$ counts the whole set $NC(n)$,
while the Narayana number $N(n,k)$ counts those noncrossing set partitions
with exactly $k$ blocks, and the Kreweras number $\Krew(\lambda)$ counts
those for which $\lambda$ lists their block sizes.
The following table illustrates these interpretations for $n=4$.

\vskip.1in

\begin{center}
\begin{tabular}{|c|c|c|c|c|}\hline
$\substack{\text{noncrossing}\\\text{partitions}}$ & 
%$\substack{\text{nondecreasing}\\\text{parking functions}}$ &
$\lambda$ & $\Krew(\lambda)$ & $k$& $N(n,k)$ 
\\ \hline\hline
% &  &  &  &  & \\ 
$\substack{1234}$  &
% $ \substack{(1,1,1,1)} $ &
$(4)$ & $\frac{1}{5}\binom{5}{4,1}=1$ &
$1$ & $\frac{1}{1}\binom{3}{0}\binom{4}{0} =1$ \\ 
% &  &  &  &  & \\ 
\hline\hline
% &  &  &  &  & \\ 
$ \substack{123-4\\124-3\\134-2\\1-234} $ & 
%$ \substack{(1,1,1,2)\\(1,1,1,3)\\(1,1,1,4)\\(1,2,2,2)} $ &
$(3,1)$ & $\frac{1}{5}\binom{5}{3,1,1}=4$ & & \\ 

  &  &  & $2$ & $\frac{1}{2}\binom{3}{1}\binom{4}{1} =6$ \\ 

$ \substack{12-34\\14-23} $ & 
%$ \substack{(1,1,2,2)\\(1,1,3,3)} $ &
$(2,2)$ & $\frac{1}{5}\binom{5}{3,2}=2$ & & \\
% &  &  &  &  & \\ 
\hline\hline
% &  &  &  &  & \\ 
$ \substack{12-3-4\\13-2-4\\14-2-3\\1-23-4\\1-24-3\\1-2-34} $ & 
%$ \substack{(1,1,2,3)\\(1,1,2,4)\\(1,1,3,4)\\(1,2,2,3)\\(1,2,2,4)\\(1,2,3,3)} $ &
$(2,1,1)$ & $\frac{1}{5}\binom{5}{2,2,1}=6$ &
$3$ & $\frac{1}{3}\binom{3}{2}\binom{4}{2} =6$ \\
% &  &  &  &  & \\ 
\hline\hline
% &  &  &  &  & \\ 
$\substack{1-2-3-4}$ & 
%$ \substack{(1,2,3,4)} $ &
$(1,1,1,1)$ & 
$\frac{1}{5}\binom{5}{1,4}=1$ &
$4$ & $\frac{1}{4}\binom{3}{3}\binom{4}{3} =1$\\
%  &  &  &  &  & \\ 
\hline
\end{tabular}
\end{center}

%\vskip.1in
%
%\begin{tabular}{|c|c|c|c|c|c|}\hline
%$\substack{\text{noncrossing}\\\text{partitions}}$ & 
%$\substack{\text{nondecreasing}\\\text{parking functions}}$ &
%$\lambda$ & $\Krew(\lambda)$ & $k$& $N(n,k)$ 
%\\ \hline\hline
%% &  &  &  &  & \\ 
%$\substack{1234}$ & $ \substack{(1,1,1,1)} $ &
%$(4)$ & $\frac{1}{5}\binom{5}{4,1}=1$ &
%$1$ & $\frac{1}{1}\binom{3}{0}\binom{4}{0} =1$ \\ 
%% &  &  &  &  & \\ 
%\hline\hline
%% &  &  &  &  & \\ 
%$ \substack{123-4\\124-3\\134-2\\1-234} $ & 
%$ \substack{(1,1,1,2)\\(1,1,1,3)\\(1,1,1,4)\\(1,2,2,2)} $ &
%$(3,1)$ & $\frac{1}{5}\binom{5}{3,1,1}=4$ & & \\ 
%
%&  &  &  & $2$ & $\frac{1}{2}\binom{3}{1}\binom{4}{1} =6$ \\ 
%
%$ \substack{12-34\\14-23} $ & 
%$ \substack{(1,1,2,2)\\(1,1,3,3)} $ &
%$(2,2)$ & $\frac{1}{5}\binom{5}{3,2}=2$ & & \\
%% &  &  &  &  & \\ 
%\hline\hline
%% &  &  &  &  & \\ 
%$ \substack{12-3-4\\13-2-4\\14-2-3\\1-23-4\\1-24-3\\1-2-34} $ & 
%$ \substack{(1,1,2,3)\\(1,1,2,4)\\(1,1,3,4)\\(1,2,2,3)\\(1,2,2,4)\\(1,2,3,3)} $ &
%$(2,1,1)$ & $\frac{1}{5}\binom{5}{2,2,1}=6$ &
%$3$ & $\frac{1}{3}\binom{3}{2}\binom{4}{2} =6$ \\
%% &  &  &  &  & \\ 
%\hline\hline
%% &  &  &  &  & \\ 
%$\substack{1-2-3-4}$ & $ \substack{(1,2,3,4)} $ &
%$(1,1,1,1)$ & $\frac{1}{5}\binom{5}{1,4}=1$ &
%$4$ & $\frac{1}{4}\binom{3}{3}\binom{4}{3} =1$\\
%%  &  &  &  &  & \\ 
%\hline
%\end{tabular}

\vskip.1in

\subsection{Generalizations to other $W$ and other $m$}

\subsubsection{Generalizing noncrossing partitions to other $W$ and $m = sh+1$}

The set of noncrossing set partitions $NC(n)$ 
has a generalization to all Weyl groups and for any parameter $m$ of the form $sh+1$ where $s \in \mathbb N$.
The case of $m =h+1$ was introduced by Bessis \cite{Bessis1}.\footnote{In fact, Bessis's work in \cite{Bessis1} deals not just with Weyl groups, but
all finite real reflection groups, and his later work in 
\cite{Bessis2} deals more generally with the 
class of {\it well-generated complex reflection groups}.  See
work of Gordon and Griffeth \cite{GordonGriffeth} for definitions of
Catalan and $q$-Catalan numbers that apply to {\it all} 
complex reflection groups.}

\begin{definition}
Consider 
the Cayley graph for $W$ with respect to the 
generating set of {\it all} reflections in $W$.  Fix a {\it Coxeter element $c$} in $W$.
Then $NC(W)$ is defined to be the set of $w$ in $W$ that lie
along a shortest path between the identity 
element and $c$ in this Cayley graph.  We regard $NC(W)$ as a partially ordered set in 
which $x \leq y$ if there is a shortest path from the identity to $c$ in this
Cayley graph that passes first through $x$ and then through $y$.
\end{definition}

Bessis  \cite{Bessis1}  showed
that the cardinality of $NC(W)$ equals $\Cat(W,h+1,q=1)$, 
generalizing Kreweras's original interpretation of the Catalan numbers $C_n$ counting the number of  
noncrossing set partitions $NC(n)$ in type $A_{n-1}$.  

Armstrong \cite{Armstrong1} defined a generalization of $NC(W)$ for each positive integer $s$, 
inspired by Edelman's {\it $s$-divisible noncrossing partitions} \cite[\S 4]{Edelman}.
\begin{definition}
Let $NC^{(s)}(W)$ denote the set of all 
$s$-element multichains $w_1 \leq \cdots \leq w_s$
in $NC(W)$. 
\end{definition}
Armstrong showed that cardinality of
$NC^{(s)}(W)$ equals $\Cat(W,sh+1,q=1)$.
Note that when $s=1$ one recovers the set $NC(W)=NC^{(1)}(W)$.

\subsubsection{Generalizing Kreweras numbers to other $W$ and very good $m$}

The Kreweras numbers $\Krew(\lambda)$ in type $A_{n-1}$ have a generalization to any $W$ and any very good $m$.  

Let $X \subset V$ be the common fixed points of a set of reflections in $W$.  Then the pointwise-stabilizer 
subgroup $W_X$ of $X$ in $W$ is a parabolic subgroup of $W$.
The normalizer $N(W_X)$ of $W_X$ within $W$ is then the not-necessary-pointwise-stabilizer of
$X$ within $W$.  We are interested in $X$ and $W_X$ up to $W$-conjugacy, so we set  $[X]:=W \cdot X$ for the $W$-orbit of $X$.

Associated to the subspace $X$ is a hyperplane arrangement, 
obtained by considering the hyperplanes in $X$ of the form $V^{w} \cap X$ where $w$ is a reflection in $W$ and 
$V^{w}$ denotes the pointwise-stabilizer of $w$ on $V$.
The characteristic polynomial of a hyperplane arrangement
\cite{OrlikTerao} is an important invariant, denoted by $p_X(t)$ for 
the hyperplane arrangement we are considering in $X$.

Using work of Orlik-Solomon \cite{OrlikSolomon1}, we know 
from \cite{Sommers1} that
when $m$ is a very good for $W$ that
\begin{equation}  \label{ungraded_rep}
\sum_{[X]} \tfrac{p_X(m)}{[N(W_X):W_X]}  \cdot 1_{W_X}^W,
\end{equation}
%where $X$ is some element in $[X]$,
is a representation of $W$ whose character takes the value $m^{\dim V^w}$ at $w \in W$.
Moreover, by Shepard-Todd \cite{ShephardTodd}, proved uniformly by Solomon \cite{Solomon},
we know that the multiplicity of the trivial representation in \eqref{ungraded_rep}
is $\Cat(W, m, q=1)$, which implies that
\begin{equation}   \label{ungraded_sum_Cat}
\Cat(W, m, q=1)  = \sum_{[X]} \tfrac{p_X(m)}{[N(W_X):W_X]}.
\end{equation}

In fact, $p_X(t)$ takes a simple form.  It is a monic polynomial with 
positive roots 
\begin{equation}
\label{Orlik-Solomon-exponents}
m_1(X),m_2(X),\ldots,m_{\dim(X)}(X)
\end{equation}
called the {\it Orlik-Solomon exponents} for $X$.
The fact that $p_X(t)$ has this form is a consequence of this hyperplane arrangement being
{\it free} (see \cite{Terao}), proved by Orlik-Terao \cite{OrlikTerao} and then uniformly by Broer \cite{Broer} and Douglass \cite{Douglass}.

In type $A_{n-1}$, $W_X$ is, up to conjugacy, just a Young subgroup of $S_n$ of the form $S_{\lambda_1} \times S_{\lambda_2} \times \dots S_{\lambda_k}$, which allows us to associate the partition $\lambda = (\lambda_1, \lambda_2, \dots, \lambda_k)$ of $n$ to $W_X$.  
Then $p_X(t) = (t-1)(t-2) \cdots (t-k+1)$ and 
$\tfrac{p_X(n+1)}{[N(W_X):W_X]}$ coincides with $\Krew(\lambda)$.  
Based on this fact and 
\eqref{ungraded_sum_Cat}, Athanasiadis and Reiner \cite{AthanasiadisReiner} considered
$$\Krew(W,[X],m) := \tfrac{p_X(m)}{[N(W_X):W_X]} =   \tfrac{1}{[N(W_X):W_X]} \prod_{i=1}^{\dim(X)} \left( m-m_i(X) \right)$$
as the Kreweras numbers for arbitrary $W$ and very good $m$.  They then showed that 
$\Krew(W, [X], h+1)$ equals the cardinality 
of the set of $w \in NC(W)$ with $V^w \in [X]$ \cite[Theorem 6.3]{AthanasiadisReiner}.  In classical ($A,B,C,D$) and dihedral ($I$) types,
work of Rhoades \cite{Rhoades} showed more generally that 
$\Krew(W, [X], sh+1)$ counts the elements $w_1 \leq \cdots \leq w_s$
in $NC^{(s)}(W)$ with $V^{w_1}$ in $[X]$; this remains open
in the exceptional types $E,F, H$.

\subsection{Cyclic Sieving}

Armstrong  defined a natural action of the cyclic group $\ZZ/sh\ZZ$ on $NC^{(s)}(W)$.
In \cite{BessisR} it was shown that this gives an instance
of the {\it cyclic sieving phenomenon} introduced in \cite{RSW} on 
$NC^{(s)}(W)$.
To state it, for a positive integer $d$, let
$\omega_d:=e^{\frac{2\pi i}{d}}$, a primitive $d^{th}$ root-of-unity.

\begin{theorem}[\cite{BessisR}]  \label{BR_CSP}
For $m = sh+1$,
one has that
$\Cat(W,m;q=\omega_{d})$ counts
those 
$$w_1 \leq \cdots \leq w_s \in NC^{(s)}(W)$$
that are fixed under the action of
an element of order $d$ in the $\ZZ/sh\ZZ$-action.
\end{theorem}

In \cite[\S 6]{BessisR}, it was asked how to produce $q$-Kreweras numbers, 
$\Krew(\Phi,X,m;q)$, polynomials that would evaluate
to the Kreweras numbers $\Krew(W, [X], m)$ at $q=1$,
but more generally would have the following property:
$\Krew(\Phi, X,m; q=\omega_{d})$ counts
the elements $w_1 \leq \cdots \leq w_s \in NC^{(s)}(W)$ 
with $V^{w_1} \in [X]$  and which are additionally fixed under the action
an element of order $d$ in the $\ZZ/sh\ZZ$-action.
Such a result would generalize Theorem \ref{BR_CSP}.

\subsection{The $q$-Kreweras numbers}
In work of the second author \cite{Sommers2},
a polynomial in the variable $q$, denoted $f_{e,\phi}(m;q)$,
 is introduced for a nilpotent element $e \in \ggg$, an irreducible representation $\phi$ of 
the component group of $e$ arising in the Springer correspondence, 
and a very good $m$  (see \S \ref{general-Kreweras-definition-section}).  
The definition involves a graded version of the representation in \eqref{ungraded_rep} and only depends on the nilpotent orbit $\0$ containing $e$.  

Given $X$ as before,  the centralizer in $\mathfrak g$ of $X \subset \mathfrak h$ is a Levi subalgebra, denoted $\mathfrak l_X$, which contains $\mathfrak h$ and whose Weyl group identifies with $W_X$.  
We write $\0_X$ for the unique nilpotent orbit which contains elements that are principal nilpotent in $\mathfrak l_X$.
The definition of $\0_X$ only depends on $[X]$.  We say $\0_X$, and each of the elements it contains, is principal-in-a-Levi. 

Now when $e \in \0_X$, then $f_{e, \phi}(m; 1) = \Krew(W,[X],m)$.
Setting $\phi$ to be trivial, we also have 
$$\Cat(W,m;q) = \sum f_{e, 1}(m;q),$$  
where the sum is over a set of representatives $e$ from each nilpotent orbit.
These two results led us to the following definition of $q$-Kreweras numbers for each nilpotent orbit $\0$
$$\krew := f_{e,1}(m ;q) \text{ where } e \in \0$$
and to conjecture
\begin{conjecture}
\label{CSP-conjecture}
For $m = sh+1$,
and for each $W$-orbit $[X]$,
$\Krew(\Phi,\0_X,m;q=\omega_{d})$ counts
those 
$$w_1 \leq \cdots \leq w_s \in NC^{(s)}(W)$$
which are fixed under the action an element of order $d$ in the $\ZZ/sh\ZZ$-action
and have $V^{w_1} \in [X]$.
\end{conjecture}

%\footnote{point out that B/C the Krew are different, but values at these roots coincide.  maybe show examples??}

\subsection{Results}

\subsubsection{Formulas for the $q$-Kreweras numbers}
The formulas for $f_{e, \phi}$ for general $\phi$ in the classical groups are given in Propositions~\ref{BC_most_general} and \ref{D_most_general}.   The formulas for $f_{e, \phi}$ in the exceptional groups are tabulated in Section~\ref{Exceptional_calcs}.
%The $f_{e,\phi}$ are computed in all cases in Section ~\ref{proof-of-Kreweras-formulas-section}.
Theorem~\ref{q-Kreweras-formula-theorem} below 
summarizes the formulas for the $q$-Kreweras numbers (that is, $\phi$ trivial)
in the classical types.

Recall that the nilpotent orbits in the classical Lie algebras 
can be parametrized by number partitions $\lambda$ 
obeying certain restrictions by considering the defining representation of $\ggg$ 
and taking the Jordan form for an element in the orbit.
For such a partition $\lambda$, we write $\0_\lambda$ for the corresponding orbit in $\ggg$.
%$\lambda=(\lambda_1,\ldots,\lambda_\ell)$ 
%with $\lambda_1 \geq \cdots \geq \lambda_\ell > 0$ obeying certain restrictions.
%Let $\mu(\lambda):=(\mu_1(\lambda),\mu_2(\lambda),\ldots)$ 
%where $\mu_j$ is the multiplicity of
%the part $j$ in $\lambda$, so that $\ell=:\ell(\lambda)=\sum_j \mu_j(\lambda)$.

\begin{itemize}
\item In type $A_{n-1}$, that is, $\ggg=\mathfrak{sl}_n(\CM)$, 
nilpotent orbits are parametrized by 
all partitions $\lambda$ of $n$; 
as before, denote this set of partitions
$\Par(n)$.

\item In type $B_{n}$, that is, $\ggg=\mathfrak{so}_{2n+1}(\CM)$, 
nilpotent orbits are parametrized by 
partitions $\lambda$ of $2n+1$ 
having $\mu_j(\lambda)$ even for $j$ even;
denote this set of partitions 
$\Par_B(2n+1)$. 

\item In type $C_{n}$, that is, $\ggg=\mathfrak{sp}_{2n}(\CM)$, 
nilpotent orbits are parametrized by 
all partitions $\lambda$ of $2n$ 
having $\mu_j(\lambda)$ even for $j$ odd;
denote this set of partitions
$\Par_C(2n)$.

\item In type $D_{n}$, that is, $\ggg=\mathfrak{so}_{2n}(\CM)$, 
nilpotent orbits under the orthogonal group $\mbox{O}_{2n}(\CM)$ are
parametrized by partitions $\lambda$ of $2n$ having
$\mu_j(\lambda)$ even for $j$ even;
denote this set of partitions
$\Par_D(2n)$.   
We denote by $\0_{\lambda}$ the orbit under $\mbox{O}_{2n}(\CM)$.
Then $\0_{\lambda}$ is a single $\mbox{SO}_{2n}(\CM)$-orbit unless
$\lambda$ has only even parts, in which case 
$\0_{\lambda}$ splits into two orbits under $\mbox{SO}_{2n}(\CM)$.
Both of these orbits have the same 
$q$-Kreweras numbers, given by $1/2$ times the formula shown in
Theorem~\ref{q-Kreweras-formula-theorem}(Type $D_n$).
\end{itemize}

For the classical groups of types $A,B,C,D$ 
the polynomial $\krew$ is expressed using {\it $q$-multinomials} which are defined as
follows:  for $\nu=(\nu_1,\ldots,\nu_t) \in \NN^{t}$ with 
$|\nu|:=\sum_i \nu_i \leq n$, let
$$
\qbin{n}{\nu}{q} :=\qbin{n}{\nu,n-|\nu|}{q} :=
\frac{[n]!_q}{[\nu_1]!_q \cdots [\nu_t]!_q [n-|\nu|]!_q },
$$
where $[n]!_q := [n]_q [n-1]_q \cdots [1]_q$,
and define the left side to be zero whenever $|\nu| > n$.

Letting $\lambda^\prime$ denote the {\it conjugate} or {\it transpose} partition of $\lambda$,
define 
$$
c(\lambda) :=\sum_{j} \lambda^\prime_j \lambda^\prime_{j+1}.
$$

\begin{theorem}
\label{q-Kreweras-formula-theorem}
(Type $A_{n-1}$)
For $\lambda \in \Par(n)$ and 
%for  very good $m$, that is, $\gcd(m,n)=1$, one has
for $\gcd(m,n)=1$, one has
$$
\Krew(A_{n-1}, \0_\lambda,m;q)
=q^{m(n-\ell(\lambda))-c(\lambda)} 
\frac{1}{[m]_q}\qbin{m}{\mu(\lambda)}{q}.
$$
\end{theorem}

\noindent
In types $B, C, D$, the formulas are similar, replacing various
parameters by roughly half their values.  Introduce the notation
$\hat{N}:=\lfloor N/2 \rfloor$ for $N \in \NN$, and
for %a sequence
$\nu=(\nu_1,\nu_2,\ldots)$, set
%$\hat{\mu}:=(\lfloor \mu_1/2 \rfloor, \lfloor \mu_2/2 \rfloor, \ldots)$.
$\hat{\nu}:=(\hat{\nu}_1, \hat{\nu}_2, \ldots)$.
%Let $\epsilon \in \{ 0 ,1 \}$. 
Using ``$a \equiv b$'' to abbreviate ``$a = b \bmod{2}$'',
define the following quantities
$$
\begin{aligned}
L(\lambda) &:= \# \{j \in \mathbb N : \mu_j(\lambda) \text{ odd}\}, \\
%\quad \text{ and } \quad 
\BCDexponent&
 :=m(n-\hat{\ell}(\lambda)) - \frac{c(\lambda)}{2}-\frac{L(\lambda)}{4},\\
\tau_{\epsilon}(\lambda) 
 &: =\frac{1}{2} \sum_{\substack{j \equiv \epsilon \\ \mu_j \equiv 0}}   \mu_j(\lambda) 
\text{ where }
\epsilon:=
   \begin{cases}0 & \text{ in type }C,\\
                1 & \text{ in types }B\text{ and }D.
   \end{cases}
\end{aligned}
$$

\vskip.1in
\noindent
{\bf Theorem \ref{q-Kreweras-formula-theorem}}
{\it
(Type $B_n$) 
For $\lambda \in \Par_B(2n+1)$ and 
for $m$ odd, one has
$$
\label{type-B-Kreweras-formula}
%\Krew(\0_\lambda,m;q) =
\Krew(B_n, \0_\lambda,m;q) = 
q^{
\BCDexponent
%m(n-\hat{\ell}(\lambda)) - \frac{c(\lambda)}{2}-\frac{L(\lambda)}{4} 
+ \tau_{1}(\lambda)+ \frac{1}{4}}
\prod_{i=1}^{\hat{L}(\lambda)} (q^{m-2i+1}-1)
\qbin{\hat{m}-\hat{L}(\lambda)}{\hat{\mu}(\lambda)}{q^2}
$$
}

For the type $C$ formula, additionally define
$$
\delta(\lambda):=
\begin{cases}
\frac{1}{4}-\frac{\ell(\lambda)}{2} &\text{ for }\ell(\lambda) \text{ odd,}\\
0 &\text{ for }\ell(\lambda) \text{ even.}
\end{cases}
$$

\vskip.1in
\noindent
{\bf Theorem \ref{q-Kreweras-formula-theorem}}
{\it
(Type $C_n$) 
For $\lambda \in \Par_C(2n)$ and 
for $m$ odd, one has 
$$
\label{type-C-Kreweras-formula}
\Krew(C_n, \0_\lambda,m;q) =
%\Krew(\0_\lambda,m;q) =
q^{\BCDexponent
% m(n-\hat{\ell}(\lambda)) - \frac{c(\lambda)}{2}-\frac{L(\lambda)}{4} 
   + \tau_{0}(\lambda) + \delta(\lambda)}
\prod_{i=1}^{\hat{L}(\lambda)} (q^{m-2i+1}-1)
\qbin{\hat{m}-\hat{L}(\lambda)}{\hat{\mu}(\lambda)}{q^2}.
$$
%where
%$\delta(\lambda):=\frac{1}{4}-\frac{\ell(\lambda)}{2}$
%for $\ell(\lambda)$ odd, and zero for $\ell(\lambda)$  even.
%$$
%\delta(\lambda)
%:=
%\begin{cases}
%0 & \text{ for }\ell(\lambda)\text{ even}, \\
%\frac{1}{4}-\frac{\ell(\lambda)}{2} 
%& \text{ for }\ell(\lambda)\text{ odd}.
%\end{cases}
%$$
}
\noindent
\vskip.1in

In type $D_n$, the multiplicity
$\mu_1(\lambda)$ of the part $1$ in $\lambda$ plays a special role,
and we also define 
$$\mu_{\geq 2}(\lambda):=(\mu_2(\lambda),\mu_3(\lambda),\ldots).$$

\vskip.1in
\noindent
{\bf Theorem \ref{q-Kreweras-formula-theorem}}
{\it
(Type $D_n$) 
For $\lambda \in \Par_D(2n)$
and $m$ odd,
% one has that
$\Krew(D_n, \0_\lambda,m;q)$ is 
$
q^{
\BCDexponent
%m(n-\hat{\ell}(\lambda)) -\frac{c(\lambda)}{2}-\frac{L(\lambda)}{4} 
   + \tau_{1}(\lambda)}
$ 
times this:
%\footnote{correcting earlier version: $\hat{L}(\lambda)$ was missing in second case in second multinomial-- Vic should double-check my proof}
$$ 
\begin{cases}
\displaystyle 
%q^{m-\frac{\ell(\lambda)}{2}+1}
q^{m-\hat{\ell}(\lambda)+1}
\prod_{i=1}^{\hat{L}(\lambda)-1} (q^{m-2i+1}-1) \cdot
\qbin{\hat{m}+1 -\hat{L}(\lambda)}{\hat{\mu}(\lambda)}{q^2} 
& \text{if }\mu_1(\lambda) \text{ is odd},\\
& \\
\displaystyle 
q^{\hat{\ell}(\lambda)-\mu_1(\lambda)}
\prod_{i=1}^{\hat{L}(\lambda)} (q^{m-2i+1}-1) \cdot 
\qbin{\hat{m}-\hat{L}(\lambda)}{\hat{\mu}_{\geq 2}(\lambda)}{q^2} 
\qbin{\hat{m}+1 - \hat{L}(\lambda)-|\hat{\mu}_{\geq 2}(\lambda)|}
     {\hat{\mu}_1(\lambda)} {q^2} 
%&\mu_1(\lambda)\text{ even, but some }\mu_j(\lambda)\text{ odd},\\
&\text{if } \mu_1(\lambda)\text{ even and } \hat{L}(\lambda) \geq 1,\\ 
& \\
q^{\hat{\ell}(\lambda)-\tau_1(\lambda)}
  \qbin{\hat{m}}{\hat{\mu}(\lambda)}{q^2}  + 
q^{\hat{\ell}(\lambda) -\mu_1(\lambda)}
  \qbin{\hat{m}}{\hat{\mu}_{\geq 2}(\lambda)}{q^2} 
\qbin{\hat{m}+1-|\hat{\mu}_{\geq 2}(\lambda)|}
     {\hat{\mu}_1(\lambda)} {q^2} 
%& \text{if } \mu_j(\lambda) \text{ all even}. \\
& \text{if } \hat{L}(\lambda) = 0 . \\
\end{cases}
$$
}

%The formulas for $f_{e, \phi}$ when $\phi$ is nontrivial in the classical groups, are given in Propositions~\ref{BC_most_general} and \ref{D_most_general}.   The formulas for $f_{e, \phi}$ in the exceptional groups are tabulated in Section~\ref{Exceptional_calcs}.

\subsubsection{Divisibility and positivity properties of the $q$-Kreweras numbers}

Using our explicit formulas for the $f_{e,\phi}$, we gather 
some of their properties in the following 
Theorem~\ref{divisibility-and-nonnegativity-theorem_all}.  Its statement 
will be slightly less precise for a fairly short list of 
ill-behaved nilpotent  orbits occurring inside the 
exceptional types $F_4, E_6, E_7,$ and $E_8$,
given here by their Bala-Carter notation \cite{Carter}:
%(see Section~\ref{Exceptional_calcs}):
\begin{equation}
\label{ill-behaved-orbits}
F_4(a_3), \,
E_6(a_3), \,
E_6(a_3)\!+\!A_1, \, 
E_7(a_5), \,
E_7(a_3), \,
E_8(a_7), \,
E_8(a_6), \,
E_8(b_5), \, 
E_8(a_4), \,
E_8(a_3).
\end{equation}
Let $R = \rank(Z_G(e))$.
Recall that $e$ is principal-in-a-Levi if $e \in \0_X$ for some $X$.
We will denote by $H^*(\BBB_e)$ the cohomology of the Springer fiber 
for $e \in \0$, regarded as a $W$-representation,
which will play a central role in the definition of the $q$-Kreweras
numbers in \S \ref{general-Kreweras-definition-section}.

\begin{theorem}
\label{divisibility-and-nonnegativity-theorem_all}
Let $e$ be a nilpotent element {\bf not} among the 
ill-behaved orbits from \eqref{ill-behaved-orbits},
and assume that $f_{e,\phi}$ is not identically zero.
Then there exists $L, c \in \NN$, independent of $\phi$, such that
$$f_{e,\phi}(m ;q) = \prod^{L}_{j=1} (q^{m+1-2j} -1) \cdot q^{cm} \cdot g_{\phi}(m; q),$$
where $g_{\phi}(m;q)$ is the sum of at most two products of the form $q^{-z} \displaystyle \prod^{R}_{i=1} \tfrac{[m-a_i]_q}{[b_i]_q}$ for some $a_i, b_i, z \in \NN$.
Moreover,
\begin{enumerate}
\item[(i)] For each very good $m$, the polynomial $q^{cm} \cdot g_{\phi}(m; q)$ 
lies in $\NN[q]$.
\item[(ii)] The rank $r$ of $\ggg$ equals $L + c + R$.
\item[(iii)] The multiplicity of $V$ in the $W$-representation $H^*(\BBB_e)$ 
is $r-c$.
\item[(iv)] If $e$ is principal-in-a-Levi, then $L=0$. 
In particular, $f_{e,\phi}(m;q) \in \NN[q]$ for each very good $m$. 
\item[(v)] If $e$ is not principal-in-a-Levi, then $L \geq 1$.  In the exceptional types it always happens that $L=1$.
\end{enumerate}

Even when $e$ is one of the ill-behaved orbits from
\eqref{ill-behaved-orbits}, if one further specializes
to the case $\phi=1$, then the polynomial $f_{e,1}(m;q)$ 
is always nonzero,
and still has properties (i),(ii),(iv),(v) listed above\footnote{
Examination of the tables in Section~\ref{Exceptional_calcs} 
shows that for $e$ in an ill-behaved nilpotent orbit,
for certain $\phi \neq 1$ one still has a factorization of
$f_{e,\phi}(m;q)$ as in the theorem, 
but with $-g_\phi(m;q)$ in $\NN[q]$.  Also, for such $e$,
property (iii) fails even if $\phi=1$.  Instead, the value $r-c$ is the multiplicity of
$W$ in the $A(e)$-invariants in $H^*(\BBB_e)$.}. 
\end{theorem}

%%%%  The old Theorem 1.7 ...
\begin{comment}
A slightly modified statement holds for the remaining orbits.  For simplicity we state a version of the theorem that holds for any $e$ when $\phi=1$.
\begin{theorem}
\label{divisibility-and-nonnegativity-theorem}
The $f_{e,1}$ are always nonzero and 
$$f_{e,1}(m ;q) = \prod^{L}_{j=1} (q^{m+1-2j} -1) \cdot q^{cm} \cdot g(m; q)$$
for some $L, c \in \NN$ and where $g(m;q)$ is the sum of at most two products of the form $q^{-z} \displaystyle \prod^{R}_{i=1} \tfrac{[m-a_i]_q}{[b_i]_q}$ for some $a_i, b_i, z \in \NN$.
Moreover,
\begin{itemize}
\item For each very good $m$, the polynomial $q^{cm} \cdot g(m; q)$ lies in $\NN[q]$.
\item The rank $r$ of $\ggg$ equals $L + c + R$.
\item If $e$ is principal-in-a-Levi, then $L=0$. %$f_{e,\phi}(m ;q)  = q^{cm} \cdot  g(m; q)$.
In particular, $f_{e,1}(m;q) \in \NN[q]$ for each very good $m$. 
\item If $e$ is not principal-in-a-Levi, then $L \geq 1$.  
In particular, $f_{e,1}$, as polynomial in $q$, is divisible by $q^{m-1}-1$ for each very good $m$.  In the exceptional types it turns out that $L$ is always equal to $1$.
\end{itemize}
\end{theorem}
\end{comment}
%%%%%%%%%%%%%%%%%%%%

It follows that the $q$-Kreweras numbers $f_{e,1}(m;q)$
are never identically zero, that they have nonnegative coefficients 
as polynomials in $q$ if $e$ is principal-in-a-Levi, and are divisible by $q^{m-1}-1$ otherwise.   
These facts are used in establishing the cyclic sieving property.
The proof of Theorem~\ref{divisibility-and-nonnegativity-theorem_all} 
is given in Section \ref{proof-of-divisibility-and-nonnegativity-section}.

\subsubsection{Cyclic sieving}

We also show that the cyclic sieving property holds in the classical types in Section \ref{proof-of-CSP-theorem-section}.

\begin{theorem}
\label{CSP-theorem}
Conjecture~\ref{CSP-conjecture}
holds in all of the classical types $A,B,C,D$.
\end{theorem}

%%%%%%%%%%%%%%%%%%%%%%%%%%%%
%Narayana part
%%%%%%%%%%%%%%%%%%%%%%%%%

\subsubsection{$q$-Narayana numbers}

We are able to establish a $q$-version of the Narayana numbers in types $A,B,C$.  We take up the question of the $q$-Narayana numbers for other types in a sequel paper.
From Theorem \ref{divisibility-and-nonnegativity-theorem_all} we have that the lowest degree $q$-monomial in $\krew$ equals $q^{cm-z}$ for some $c,z \in \NN$.  

\begin{definition}
Define $d(\0) = r - c$.
\end{definition}

As long as $e$ does not belong to one of the orbits in \eqref{ill-behaved-orbits}, 
then by Theorem \ref{divisibility-and-nonnegativity-theorem_all} (iii), 
$d(\0)$ equals the multiplicity of $V$  in $H^*(\BBB_e)$.
%of the Springer fiber for $e \in \0$ (see \S \ref{general-Kreweras-definition-section}).
In particular, this holds whenever $\ggg$ is of classical type or $e$ is principal-in-a-Levi.  
On the other hand, when $e$ belongs to one of the orbits in \eqref{ill-behaved-orbits}, then 
$d(\0)$ equals the multiplicity of $V$  in $A(e)$-invariants of $H^*(\BBB_e)$.
Using the parameter $d(\0)$, we obtain a good definition for a $q$-version of the Narayana numbers in types $A,B,C$.

\begin{definition}  (Types $A, B, C$)
The {\it $q$-Narayana number} for $k=0,1,2,\ldots,r$ and very good $m$ is given by
\begin{equation}
\label{q-Narayana-definition}
\Nar(\Phi,m,k;q):=
 \sum_{\substack{\text{nilpotent orbits } \0:\\ d(\0)=k}} \Krew(\Phi, \0 ,m;q).
\end{equation}
\end{definition}

\begin{theorem}
\label{q-Narayana-formula-theorem}
The $q$-Narayana polynomials have the following formulas in types $A,B,C$:
\begin{itemize}
\item 
For type $A_{n-1}$, when $\gcd(m,n)=1$ one has for $0 \leq k \leq n-1$ that
$$
\Nar(A_{n-1},m,k;q)=
q^{(n-1-k)(m-1-k)}
\frac{1}{[k+1]_q} \qbin{n-1}{k}{q} \qbin{m-1}{k}{q}.
$$
\item 
For either of type $B_n, C_n$, when $m$ is odd one has 
for $0 \leq k \leq n$ that
$$
\Nar(B_n,m,k;q)=\Nar(C_n,m,k;q)=
(q^{2})^{(n-k)(\hat{m}-k)} \qbin{n}{k}{q^2} \qbin{\hat{m}}{k}{q^2}.
$$
\end{itemize}
\end{theorem}

\noindent
When $m=h+1=n+1$ in type $A_{n-1}$, the polynomial $\Nar(A_{n-1},h+1,k;q)$ 
equals a $q$-Narayana number considered by Furlinger and Hofbauer \cite{FurlingerHofbauer} and Br\"anden in \cite{Braenden}.

Theorem~\ref{q-Narayana-formula-theorem}
shows that, in 
the one instance where two root systems $\Phi=B_n,C_n$ are associated
with the same Weyl group $W$, it turns out that
$\Nar(\Phi,m,k;q)$ depends only on $W$, and not on $\Phi$,
even though the $\krew$ for the two root systems are not 
the same (they are not even indexed by the same set).
It is also interesting to note that, although the polynomials $\krew$
can have negative integral coefficients in types $B$ and $C$, the 
formulas above for $\Nar(\Phi,m,k;q)$ exhibit them as polynomials 
in $q$ with {\it nonnegative} coefficients, that is, lying in $\NN[q]$.

Theorem \ref{q-Narayana-formula-theorem} is proved in Section~\ref{proof-of-Narayana-formulas-section}.

%%%%%%%%%%%%%%%%%%%%%%%%%%%%%%%%%%%%%%%%%%%%%%%%%%%%%%%%%%%%%%%
\section{Reviewing $f_{e,\phi}$ and 
the $q$-Kreweras numbers $f_{e,\one}$}
\label{general-Kreweras-definition-section}
%%%%%%%%%%%%%%%%%%%%%%%%%%%%%%%%%%%%%%%%%%%%%%%%%%%%%%%%%%%%%%%

In this section we detail some results from \cite{Sommers2}. 
Recall that 
$%$
S = \Sym(V^*) %\cong \RR[x_1,\ldots,x_r]
$ is the graded ring of polynomials on the reflection representation $V$ of $W$
and $r = \dim V$.
When $m$ is very good for $\Phi$, it is known \cite{Gordon, BerestEtingofGinzburg}
that $S$ contains a 
{\it homogeneous system of parameters} $\theta^{(m)}=(\theta^{(m)}_1,\ldots,\theta^{(m)}_r)$
whose 
%$\RR$-span 
span is $W$-stable and carries a representation isomorphic to $V$.
%It is also known \cite{BerestEtingofGinzburg, BessisR, Gordon}
This implies by \cite{Solomon} that the $W$-invariant subspace of the finite-dimensional quotient
ring $S/(\theta^{(m)})$ is a graded vector space whose Hilbert series is
$\Cat(W,m;q)$ and in particular shows that $\Cat(W,m;q)$ is polynomial.
%define a graded virtual representation of $W$ by
%$$\H = \sum^{l}_{j=0}j (-q^m)^j S \otimes \wedge^j V,$$ 
%where $V*$ sits in degree $1$ in $S$.  ($V$* versus $V$; $q$ versus degree).
%When $m$ is restricted to 
%this is known to be an actual graded representation of $W$ since it coincides with a quotient of $S$ generated by 
%an ideal generated by  single copy of $V$ in $S$ in degree $m$.
%\begin{definition}
%For a graded $\RR$-vector space $U = \sum_{d \geq 0} U_d$ with each
%$u_d$ a $W$-representation, its {\it Frobenius series} is this
%polynomial or power series in $q$, having coefficients in the Grothendieck
%group of $\RR W$-modules:
%$$
%\sum_d [U_d] q^d.
%$$ 
%Here $[U_d]$ is the class of $U_d$ in the Grothendieck group.
%\end{definition}
A main result of \cite{Sommers2} is that $S/(\theta^{(m)})$, 
as graded $W$-module, is an integral combination of certain
finite-dimensional graded $W$-representations (and their graded shifts),
related to the Green functions that arise in the representation theory of Chevalley groups.

To be more precise,
let $G$ be the connected simple algebraic group of adjoint type over an algebraically closed field 
$\mathbf{k}$ of good characteristic $p$ attached to the root system $\Phi$.  Let $\ggg$ be its Lie algebra.
For a nilpotent element $e \in \ggg$, let $\BBB_e$ be the variety of Borel subalgebras containing $e$, known as a Springer fiber.
Let $Z_G(e)$ be the centralizer of $e$ in $G$ and let $A(e) := Z_G(e)/ Z^{\circ}_G(e)$ be the component group of $e$.
Then the $l$-adic cohomology $H^*(\BBB_e)$ carries a representation of $W \times A(e)$ \cite{Lus1.5, Hotta}, originally defined by Springer \cite{SpringerRep}.
Denote the irreducible $\overline{{\mathbb Q}}_l$-representations of $A(e)$ by $\ar$.  
For $\phi \in \ar$ 
define the finite-dimensional, graded representations $Q_{e, \phi}$ so that 
$$H^*(\BBB_e) =  \sum_{\phi \in \ar} Q_{e, \phi} \otimes \phi,$$ 
as graded representations of $W \times A(e)$.
The cohomology of $\BBB_e$ vanishes in odd degrees and 
we grade $Q_{e, \phi}$ by putting $q$ in cohomological degree two.
Then (as Frobenius series)

\begin{equation}
\label{parking-space-as-Kreweras-times-Springer-fibers}
S/(\theta^{(m)}) = \sum_{e,\phi} f_{e,\phi}(m; q) Q_{e, \phi}
\end{equation}

\noindent
where $f_{e,\phi}(m; q) \in \ZZ[q]$ and the sum is over a set of representatives $e$ of the nilpotent orbits in $\ggg$
and those $\phi \in \ar$ that appear in the Springer correspondence.

There is an explicit formula \cite[Equation 18]{Sommers2} for $f_{e,\phi}(m; q)$ that involves (1) the cardinality of the rational nilpotent orbits in $\ggg$ for a finite subfield $\FF_q$ of $\mathbf k$;
and (2) the Frobenius series of the reflection representation $V$ in $Q_{e,\phi}$.

%Define $\{(e_1, \pi_1),  (e_2, \pi_2), \dots, (e_{\kappa}, \pi_{\kappa})  \},$ where $e_j \in \ZZ_{\geq 0}$  and $\pi_j \in \ar$, by 
Define $\{(m_1, \pi_1),  (m_2, \pi_2), \dots, (m_{\kappa}, \pi_{\kappa})  \},$ where $m_j \in \ZZ_{\geq 0}$  and $\pi_j \in \ar$, by 
\begin{equation}
\label{nilpotent-orbit-exponents-definition}
\sum_{ j \geq 0} q^{j} \langle H^{2j}(\BBB_e), V \rangle_{W}  = q^{m_1} \pi_1 + q^{m_2} \pi_2 + \cdots + q^{m_{\kappa}} \pi_{\kappa},
\end{equation}
where the pairing is the usual inner product for $W$ and the result is viewed as a Frobenius series for $A(e)$. 
%In other words, the $(m_j, \pi_j)$'s measure the occurrences of $V$ in the total Springer representation $H^*(\BBB_e)$
%and ${\kappa} = \kappa(e)$ is the multiplicity of such occurrences.
It turns out that at most one of the $\pi_j$'s is non-trivial and we set 
$\pi_{\kappa}$ to be this non-trivial representation of $A(e)$ when it occurs.  %See table ??.
When $e$ belongs to $\0_X$, it is known  \cite{Spalt, LeSh}  that the $m_i$'s coincide with the Orlik-Solomon exponents from \ref{Orlik-Solomon-exponents}.

Let $G(\FF_q) \subset G$ be the $\FF_q$-points of $G$ with respect to a split Frobenius endomorphism $F$.
Let $c$ denote a conjugacy class in $A(e)$ and $e_c$ a representative from the rational 
$G(\FF_q)$-orbit in $\ggg$ over $\FF_q$ corresponding to the pair $(e,c)$.   
%Set $d = \dim(\pi_k)$ and $M = m_1 + \dots + m_{\kappa-1} + dm_{\kappa}$.   
Set $d_\kappa=\dim(\pi_\kappa)$ and
$M = m_1 + \dots + m_{\kappa-1} + d_\kappa \cdot m_{\kappa}$.   
%$M = e_1 + \dots + e_{\kappa-1} + de_{\kappa}$.   
%and let $\tilde{\kappa} = \kappa -1+d$ 
%and $M = \sum_{j=1}^{\kappa-1} e_j + de_{\kappa}$.  

%% evaluate at c=1, q=1 and do the same with the derivative of the Frobenius polynomial above.

Then
\begin{equation}  \label{formula:f}
%f_{e, \phi}(m; q) = q^{m(r - \tilde{\kappa})+  M}   \prod_{j=1}^{\kappa-1} (q^{m-e_j} - 1)  \cdot 
%f_{{\mathbf e}, \phi}(m; q) = q^{m(r - \kappa  - d + 1)+  M}   \prod_{j=1}^{\kappa-1} (q^{m-e_j} - 1)  \cdot 
f_{e, \phi}(m; q) = q^{m(r - \kappa  - d_\kappa + 1)+  M}   \prod_{j=1}^{\kappa-1} (q^{m-m_j} - 1)  \cdot 
%f_{{\mathbf e}, \phi}(m; q) = q^{m(r - \kappa  - d + 1)+  \sum_{j=1}^{\kappa-1} e_j + de_{\kappa}}   \prod_{j=1}^{\kappa-1} (q^{m-e_j} - 1)  \cdot 
% \left( \sum_{i=0}^d (-1)^{d-i} q^{i(m - e_{\kappa})}  \sum_c  \frac{ \!\wedge^{d-i}  \pi_{\kappa}(c) \phi(c)} {|Z_{G^F}(e_c)|} \right) \nonumber
 \left( \sum_{i=0}^{d_\kappa} (-1)^{d_\kappa-i} q^{i(m - m_{\kappa})}  \sum_c  \frac{ \!\wedge^{d_\kappa-i}  \pi_{\kappa}(c) \phi(c)} {|Z_{G^F}(e_c)|} \right) %\nonumber
\end{equation}

When $\pi_{\kappa}$ is trivial, the above expression simplifies to
\begin{equation}  \label{formula:f_easy}
f_{e, \phi}(m; q) = q^{m(r - \kappa)+  \sum m_j} \prod_{j=1}^{\kappa} \big (q^{m-m_j} - 1 \big)
\left(  \sum_c  \frac{ \phi(c)} {|Z_{G^F}(e_c)|} \right)
\end{equation}

In the present paper we are primarily concerned with the Frobenius series of 
$S/(\theta^{(m)})$ at the trivial representation of $W$.  Since the trivial representation of $W$ only
occurs in $Q_{e,\phi}$ for the trivial local system $\phi$ and then only once in degree zero \cite{Borho-Mac_partial},
%\footnote{Eric, I added boldface because these are very important facts, and I'd like to add some justification.   Can we find them somewhere specific in your paper \cite{Sommers2}, or do we need to say more about deducing them from the definitions of the $Q_{e,\phi}$??}
we obtain the identity 
\begin{equation}
\label{original-form-of-q-Kreweras-sums-to-q-Catalan}
\Cat(W,m;q) = \sum_e f_{e,1}(m; q).
\end{equation}
By \cite[Theorem 15]{Sommers2} 
\begin{equation}%{Kr2}
\label{principal-in-Levi-Kreweras-at-q=1}
f_{e,\phi}(m; q=1) = \Krew(W,m,[X])
\end{equation}
when $e$ is conjugate to a principal nilpotent element in $\mathfrak l_X$.
When $e$ is not of that form, on the other hand,
$f_{e,\phi}(m; q=1)=0.$
In light of these results, it is reasonable to think of $f_{e,1}(m; q)$ as $q$-Kreweras numbers, where now there is one such polynomial for each nilpotent orbit in $\ggg$.  Thus,

\begin{definition}
\label{q-Kreweras-definition}
The $q$-Kreweras numbers for $\Phi$ are defined to be
$$\krew =   f_{e,1}(m; q)$$
for $e \in \0$.
\end{definition}

%In types $A,B,C$, the representation $\pi_k$ is trivial so the formula for the $f_{e,1}$'s is particularly easy.  Namely,
%\begin{equation}  \label{formula:f_easiest}
%f_{e,1} = q^{m(r - k)+  M} \frac{ |\mathcal{O}^F_e| }{|G^F|}   \prod_{j=1}^{k} \big (q^{m-e_j} - 1 \big)
%\end{equation}
%where 
%$|\mathcal{O}^F_e| $ is the cardinality of the $F$-rational points of the orbit $\mathcal{O}_e$.  

We will write down the formulas for $f_{e,\phi}(m; q)$, and hence $\krew$ in Section \ref{A_calculations} for type $A_{n-1}$, Section \ref{BCD_calcs} for the other classical types, and Section \ref{Exceptional_calcs} for the exceptional types.

%%%%%%%%%%%%%%%%%%%%%%%%%%%%%%%%%%%%%%%%%%%%%%%%%%%%%%%%%%%%%%%
\section{Computing $f_{e,\phi}$ and the Proof of Theorem~\ref{q-Kreweras-formula-theorem}}
\label{proof-of-Kreweras-formulas-section}
%%%%%%%%%%%%%%%%%%%%%%%%%%%%%%%%%%%%%%%%%%%%%%%%%%%%%%%%%%%%%%%

We use the following notation in this section.   For a partition $\lambda$, let 
$\mu_j := \mu_j(\lambda)$, which recall is the number of parts of $\lambda$ of size $j$.
For a nilpotent element $e \in \ggg$, 
let $d = \dim Z_G(e)$ and $d^u$ denote the dimension of a maximal unipotent subgroup of $Z_G(e)$.

\subsection{Type $A$}  
\label{A_calculations}

For simplicity we will work with $G = GL_n(\bar{\FF}_q)$ and adjust our results for the case where $G$ is adjoint.  Recall from the introduction that the nilpotent $G$-orbits in $\ggg$ are parametrized by $\Par(n)$ of partitions of $n$, with $\lambda \in \Par(n)$
corresponding to the sizes of the Jordan blocks of any element in the nilpotent orbit $\0_{\lambda}$ 
indexed by $\lambda$. 
Let the Frobenius map $F$ consist of raising each matrix element to the $q$-th power, giving the standard
 split structure on $G$, so that 
 $G^F = GL_n(\FF_q)$ and $\ggg^F = {\mathfrak gl}_n(\FF_q)$.
 Then it is known, say by using rational canonical form, that nilpotent $G^F$-orbits on
 $\ggg^F$ are indexed by $\Par(n)$;  in other words, the rational points of $\0_{\lambda}$ remain a single orbit under $G^F$.
 
We wish to compute $$\Krew(A_{n-1}, \0_\lambda,m;q) := f_{e, 1}(m;q)$$
where $e:=e_{\lambda} \in \0_{\lambda}$ is a rational element. % in $\0_{\lambda}$ where $\lambda \in \Par(n)$.
Since all $A(e)$ are trivial in $GL_n(\bar{\FF}_q)$,  
the computation of $f_{e,1}$ in \eqref{formula:f_easy}
reduces to calculating $m_1, \dots, m_{\kappa}$ 
and the value of $|Z_{G^F}(e)|$.  
First, according to  \cite{LeSh} we have 
\begin{equation}\label{type-A-kappa-formula}
\kappa = \ell(\lambda) - 1 \text{ and }  m_j = j,
\end{equation}
Thus \eqref{formula:f_easy}, with $r = n-1$,  becomes
\begin{equation} \label{formula:typeA}
f_{e,1} = 
q^{m(n- \ell(\lambda))+ \binom{\ell(\lambda)}{2}} \cdot
\frac{q-1}{|Z_{G^F}(e)|}   \prod_{j=1}^{\ell(\lambda)  -1} \big (q^{m-j} - 1 \big)
\end{equation}
where the extra factor of $q-1$ accounts for the center of $GL_n(\FF_q)$ since the formulas in 
\eqref{formula:f_easy} are presented relative to an adjoint group.

Next, we need to compute $|Z_{G^F}(e)|$.
As a variety $Z_G(e)$ is isomorphic to the product of 
an affine space (its unipotent radical) and a maximal reductive part of the centralizer of $Z_G(e)$.
The reductive part is isomorphic to
$$Z_{red} := \prod_j GL_{\mu_j} (\bar{\FF}_q).$$
Up to isomorphism, the affine space and each factor in the reductive part carry the standard action of $F$, 
so $Z_{G^F}(e)$ is isomorphic over $\FF_q$ to 
$$\FF_q^{d_1^u} \times \prod_j GL_{\mu_j} (\FF_q).$$
%, and carries the split $F$-structure.
Since 
$$|GL_{r} (\FF_q)| = q^{\binom{r}{2}} (q-1)^r [r]!_q,$$
it follows that
\begin{equation}\label{cent_A}
%|Z_{G^F}(e)| =   q^{ n + \sum_i \lambda'_i (\lambda'_i - 1)- \sum_j \binom{\mu_j+1}{2} }  (q-1)^{\ell} \prod_{j} [\mu_j]!_q
%|Z_{G^F}(e)| =   q^{d - \sum \binom{\mu_j\!+\!1}{2}} (q-1)^{\ell(\lambda)} \displaystyle\prod_{j} [\mu_j]!_q
|Z_{G^F}(e)| =   q^{d^u} (q-1)^{\ell(\lambda)} \displaystyle\prod_{j} [\mu_j]!_q
\end{equation}
The Borel subgroup of $Z_{red}$ has dimension $\sum \binom{\mu_j\!+\!1}{2}$,
so $d^u = d - \sum \binom{\mu_j\!+\!1}{2}$.  
Now $d = \sum (\lambda'_i)^2$ (see \cite{Carter}) and so a calculation gives
\begin{equation}  \label{small_calc}
d^u = \sum (\lambda'_i)^2 -  \sum \binom{\mu_j+1}{2} =
\binom{\ell(\lambda)}{2}  + \crossterms(\lambda)
\end{equation}
where
$$
\crossterms(\lambda):=
\sum_{j \geq 1} \lambda'_j \lambda'_{j+1}.
$$

Plugging  (\ref{cent_A}) and (\ref{small_calc}) into (\ref{formula:typeA})
into the above %and substituting $k = \ell(\lambda) -1$ gives
yields
%$$Krew(A_{n-1},e_\lambda,m;q) = 
$$
\begin{aligned}
f_{e,1} 
&= 
q^{m(n- \ell(\lambda)) -  \crossterms(\lambda)}
\frac{1}{\prod_{j} [\mu_j]!_q} \cdot \frac{[m-1]!_q}{[m-\ell(\lambda)]!_q} \\
&=q^{m(n-\ell(\lambda))-c(\lambda)} 
\frac{1}{[m]_q}\qbin{m}{\mu(\lambda)}{q}.
\end{aligned}
$$
as asserted in Theorem~\ref{q-Kreweras-formula-theorem}(Type $A$).

\begin{remark}
The formula for $|Z_{G^F}(e)|$ could also be obtained by looking up the value $|\mathcal{O}^F_e| = \frac{|G^F|}{|Z_{G^F}(e)|}$ in, for example, \cite{Crabb}.  
\end{remark}

\subsection{Preparation for types $B, C, D$}

From (\ref{formula:f}), in order to compute $f_{e, \phi}$,
we need to evaluate the sum
$$\sum_c  \frac{ \phi(c)} {|Z_{G^F}(e_c)|}$$
as $c$ runs over representatives of the conjugacy classes in $A(e)$.
The component group $A(e)$ is elementary abelian, hence we will have occasion to use the following two lemmas.

First, we introduce some notation for working with elementary abelian groups and their characters.
Let $v, w \in (\FF_2)^{s}$.  Write $v = (v_i), w=(w_i)$ relative to the standard basis 
and denote the usual dot product $\langle v, w \rangle := \sum_{i=1}^s v_iw_i$.
For each $w \in \FF_2^s$, 
define the character $\phi_w \in \widehat{\FF_2^{s}}$ by 
$$\phi_w(v) = (-1)^{\langle v, w \rangle} \in \QQ.$$
Every character of $\FF_2^{s}$ is of the form $\phi_w$ for a unique $w$.

Let $x_1, x_2, \dots, x_s$ be a set of $s$ variables.  
For $a \in \FF_2$ and a variable $y$, we evaluate $y^a$ to $1$ or $y$ according to 
whether $a =0$ or $a=1$, respectively, in $\FF_2$.

\begin{lemma}  \label{full_lemma}
%Let $x_1, x_2, \dots, x_m$ be a set of $m$ variables.  
Let $\phi$ be a character of $\FF_2^s$.  Write $\phi=\phi_w$ for the unique such $w \in \FF_2^s$.
Let $t$ be a nonnegative integer with $0 \leq t \leq s$.  Then
\begin{eqnarray*} % \label{identity}
\sum_{v \in \FF_2^s}  \phi(v) \prod_{i=1}^{t} (x_i + (-1)^{v_i})   = 
\begin{cases}
%2^{s} (\prod_{i=1}^t x_i^{w_i + 1}) & \text{if }t<s \text{ and }w_j=0 \text{ for all } j>t \\  % t not s
2^{s} \displaystyle \prod_{i=1}^t x_i^{w_i + 1} & \text{if } w_j=0 \text{ for all } j>t \\  % t not s
0 \text{ otherwise}
\end{cases}
\end{eqnarray*}
\end{lemma}
The identity also hold in the degenerate case where $s=0$.
We omit the proof since it is essentially the same (but simpler) as the proof of the next lemma.
Let $K$ denote the subgroup of $\FF_2^{s}$ consisting of those $v \in \FF_2^{s}$ with $\sum v_i = 0$.
Any character $\phi \in \widehat{K}$ is now equal to the restriction of $\phi_w$ 
for two values of $w\in \FF_2^{s}$, call them $w', w''$, where 
$w' + w'' = (1, 1, \dots, 1)$.
% with $w_m=0$.
%%Write $\epsilon_i := (-1)^{v_i}$.  %Let $\delta = (1,1,\dots, 1) \in (\FF_2)^{m}$.

\begin{lemma}  \label{even_lemma}
%Let $x_1, x_2, \dots, x_m$ be a set of $m$ variables.  
Let $\phi$ be a character of $K$ and $s >0$.  Let $w \in \FF_2^{s}$ be the unique choice such that $\phi=\phi_w$ and $w_s=0$.
Let $t$ be a nonnegative integer with $0 \leq t \leq s$.  Then
\begin{eqnarray*}  %\label{identity}
\sum_{v \in K}  \phi_w(v) \prod_{i=1}^{t} (x_i + (-1)^{v_i})   = 
\begin{cases}
2^{s-1} \left(\prod_{i=1}^s x_i^{w_i} + \prod_{i=1}^s x_i^{w_i + 1} \right) & \text{if } t=s\\
2^{s-1} \prod_{i=1}^t x_i^{w_i + 1} & \text{if }t<s \text{ and }w_j=0 \text{ for all } j>t \\
0 \text{ otherwise}
\end{cases}
\end{eqnarray*}

\end{lemma}

\begin{proof}
Embed $\FF_{\tiny 2}^t$ in $\FF_2^s$ via the first $t$ coordinates.  
%For $u =(u_1, \dots, u_r, 0, \dots, 0) \in \FF_{\tiny 2}^m$, 
%let $\bar{u}$ be the vector $(u_1+1, \dots, u_r+1, 0, \dots, 0)$. 
Then
\begin{eqnarray*}
\prod_{i=1}^{t} (x_i + (-1)^{v_i})  = \sum_{u \in \FF_2^{t}} \prod_{i=1}^{t} (x_i^{u_i+1} ((-1)^{v_i})^{ u_i}) =
\sum_{u \in \FF_2^{t}} (\prod_{i=1}^{t} x_i^{u_i+1} ) (-1)^{\sum v_i u_i}  & \\
 = \sum_{u \in \FF_2^{t}} \prod_{i=1}^{t} x_i^{u_i+1}  \cdot (-1)^{\langle v, u \rangle} = 
  \sum_{u \in \FF_2^{t}} (\prod_{i=1}^{t} x_i^{u_i+1}) \phi_{u} (v)
\end{eqnarray*}

\noindent Writing $\mathbf{x}^{u+\mathbf{1}}$ for $\prod_{i=1}^{t} x_i^{u_i+1}$ and using the identity above
and switching the order of summation gives
\begin{eqnarray*}
\sum_{v \in K}  \phi_w(v) \prod_{i=1}^{t} (x_i + (-1)^{v_i})  &  
%\sum \limits_{u \in \FF_2^{r}}  \mathbf{x}^u  \sum_{v \in K} \phi(v)  \prod_{i=1}^{m} (-1)^{(v_i)(u_i+1)}  \\
= & \sum \limits_{u \in \FF_2^{t}}   \mathbf{x}^{u+\mathbf{1}} \sum_{v \in K} \phi_w(v) \phi_{u}(v).
\end{eqnarray*}
%Interpreting $\phi_w$ and $\phi_{\bar{u}}$ as characters of the subgroup $K$,  
Character theory for $K$ implies that
the inner sum equals zero unless $\phi_w = \phi_{u}$ on $K$, in which case it equals $|K| = 2^{s-1}$.  
Now the equality $\phi_w = \phi_{u}$ holds on all of $K$ if and only 
$$w_1+w_2 = u_1+u_2, w_2+w_3 = u_2+u_3, \dots, w_{s-1}+w_s = u_{s-1}+u_s.$$
If $t=s$, this happens if and only if $u = w$ or $u = w + (1,1, \dots, 1)$, giving the first part of the result.

Next consider the case where $t<s$.  Since $u_{t+1}=u_{t+2}=...=u_{s}=0$, a necessary condition for $\phi_w = \phi_u$
is that $w_{t+1}+w_{t+2} = 0$, $w_{t+2}+w_{t+3} = 0, \dots, w_{s-1}+w_{s}=0$.
Since $w_s=0$ by hypothesis, this means that $w_{t+1} = \dots = w_s = 0$ for $\phi_w = \phi_u$ to hold, giving
the third part of the result. %If this fails to hold, the expression is zero.
Moreover, if $\phi_w = \phi_u$, then also 
$u_t =w_t$ since both $w_{t+1}=u_{t+1}=0$.  Continuing in this fashion
$u_{t-1} = w_{t-1}, \dots, u_{1}= w_{1}$.  So the unique solution for $u$ is $u=w$, which is the second part of the result.
\end{proof}

\subsection{Computing the summation.}\label{compute-summation}
Here again we abbreviate $a \equiv b \bmod{2}$ as "$a \equiv b$".
For $\epsilon \in \{ 0, 1\}$, let 
\begin{eqnarray*} 
S^+_{\epsilon}  &:=  & \{ j \in \NN \ | \ j  \equiv \epsilon, \countp_j \equiv 0, \countp_j \neq 0  \} 
%S^+_{\epsilon} & := \{ j \in \NN \ | \ j  \equiv \epsilon, \countp_j \text{ even and nonzero} \} 
\text{ and } \\
S^-_ {\epsilon}  &:=  & \{ j \in \NN \ | \ j \equiv \epsilon,  \countp_j \equiv 1 \} \text{ and } \\
%S^-_ {\epsilon} & := \{ j \in \NN \ | \ j \equiv \epsilon,  \countp_j \text{ odd} \}.
S_{\epsilon} &:=  & S^-_{\epsilon}   \cup S^+_{\epsilon} =  \{ j \in \NN \ | \ j \equiv \epsilon,  \countp_j > 0 \}.
\end{eqnarray*}
%Let $\epsilon =0$ in type $C$ and let $\epsilon= 1$ in types $B$ and $D$.  
For types $B,C,D$, we set $q$ to be a power of an odd prime.

\medskip
{\bf Type} $C$
\medskip

For $\lambda \in \Par_C(2n)$, pick $e \in \0_{\lambda}$.
%We continue to set $\countp_j:=\mu_j(\lambda)$.
%, so that $\ell(\lambda)=|\mu|$ is the number of parts of $\lambda$.
%Let 
%$S^-_0 = \{ j \in \NN \ | \ j \text{ even}, \countp_j \text{ even and nonzero} \}$ 
%and
%$S^-_1 = \{ j \in \NN \ | \ j \text{ even}, \countp_j \text{ odd} \}.$
%Let $q$ be an odd prime.
Then working in $G = \mbox{Sp}_{2n}(\aF_q)$,
a maximal reductive part of the centralizer $Z_G(e)$ is isomorphic to
%$$\prod_{j \equiv 0} \mbox{O}_{\countp_j} (\aF_q) \times \prod_{j \equiv 1} \mbox{Sp}_{\countp_j} (\aF_q).$$
\begin{equation} \label{red_decomp1}
%\prod_{j \equiv 1} \mbox{Sp}_{\countp_j} (\aF_q) \times \prod_{j \in S^-_0 \cup S^-_1} \mbox{O}_{\countp_j} (\aF_q).
\prod_{j \equiv 1} \mbox{Sp}_{\countp_j} (\aF_q) \times \prod_{j \in S_0} \mbox{O}_{\countp_j} (\aF_q).
\end{equation}

Let $A$ be the elementary abelian $2$-group with basis 
%$S^-_0 \cup S^-_1$.  
$S_0 = S^-_0 \cup S^+_0$.  
For $c \in A$, we write
$$c=(c_j)_{j \in S_0} \text{ with } c_j \in \FF_2.$$
We identify $A$ with the component group $A(e)$, where 
$(c_j) \in A$ corresponds to taking an element of determinant $(-1)^{c_j}$ in the orthogonal group in (\ref{red_decomp1}) indexed by $j$, 
for each $j \in S_0$.%$j \in S^-_0 \cup S^-_1$.

Now we assume that $e \in \ggg^F$ and has split centralizer.  For $c \in A(e)$, 
we twist $e$ by $c$ to get another 
rational element $e_c \in \orbit_e$.  In this way we obtain representatives from all 
the $G^F$-orbits on $\0^F_e$.
Under our identification of $A$ and $A(e)$, 
the group of rational points in a maximal reductive subgroup of $Z_G(e_c)$ is isomorphic to
$$\prod_{j \equiv 1} \mbox{Sp}_{\countp_j} (\FF_q) \times \prod_{j \in S^-_0} \mbox{O}_{\countp_j} (\FF_q) \times 
\prod_{j \in S^+_0} \mbox{O}^{c_j}_{\countp_j} (\FF_q),$$
where the groups in the last product are either split or twisted orthogonal groups of type $D$ %$D_{\frac{\mu_j}{2}}$ 
depending on whether $c_j$ is equal to $0$ or $1$, respectively;  see Shoji \cite[\S 1]{Shoji}.

At this stage another $q$-analogue notation is helpful: 
for a nonnegative integer $n$, let
$$
\begin{aligned}
\eta(N)&:= (q^2-1)(q^4-1) \cdots (q^{2\lfloor\frac{N}{2}\rfloor}-1) \\
&=\begin{cases}
(q^2-1)(q^4-1) \cdots (q^N-1) &\text{ if }N\text{ is even},\\
(q^2-1)(q^4-1) \cdots (q^{N-1}-1) &\text{ if }N\text{ is odd}.\\
\end{cases}
\end{aligned}
$$
or in other words,
$$
\eta(2m+1) = \eta(2m) = \prod_{i=1}^m (q^{2i} - 1).
$$
The cardinality of $Z_{G^F}(e_c)$ is therefore (see, e.g., Carter \cite[\S 2.9, p. 75]{Carter})

\begin{equation} \label{points_C}
q^{d^u} \cdot |A(e)|
\prod_{j \equiv 1} \eta(\countp_j)
\prod_{j \in S^-_0} \eta(\countp_j)
\prod_{j \in S^+_0} (q^{\frac{\countp_j}{2}} - (-1)^{c_j}) \cdot \eta(\countp_j -2). 
\end{equation}
%$$\prod_{\stackrel{j \equiv 0}{ \countp_j \equiv 1} } [\countp_j -1]_2! \prod_{j \equiv 1} [\countp_j ]_2!
%\prod_{\stackrel{j \equiv 0}{\countp_j \equiv 0}} [\countp_j -2]_2! (q^{\frac{\countp_j}{2}} - c_i)$$ equals
%Here, $c_j = (-1)^{x_j}$.
Getting a common denominator over all the conjugacy classes in $A(e)$, we obtain
\begin{eqnarray}  \label{c_sum1}
\sum_c \frac{ \phi(c)}{|Z_{G^F}(e_c)|} =  \frac{\sum_c \phi(c) \prod_{j \in S^+_0} (q^{ \frac{\countp_j}{2}} + (-1)^{c_j})}
{ q^{d^u} |A(e)| \prod_{j} \eta(\countp_j)}. % \cdot \frac{1}{}.
\end{eqnarray}
%now d_u instead of d_e,where $d_e$ equals the dimension of the unipotent radical of a Borel subgroup of $Z_G(e)$.
To evaluate this sum we use Lemma \ref{full_lemma} with $x_i = q^{\frac{\countp_i}{2}}$, $s = |S_0|$, 
and $t = |S^+_0|$.
Choose $w \in A$ to be the unique element so that $\phi = \phi_w$.
Then by the lemma the expression in \eqref{c_sum1}
%\begin{eqnarray}  \label{c_sum2}
%\sum_c \frac{ \phi(c)}{|Z_{G^F}(e_c)|}
%\end{eqnarray}
equals $0$ if $w_j \neq 0$ for any $j \in S^-_0$,
and 
\begin{eqnarray}  \label{c_sum3}
\sum_c \frac{ \phi(c)}{|Z_{G^F}(e_c)|} =  \frac{q^{ -d^u +  \sum_{j \in S^+_0, w_j = 0} \frac{\countp_j}{2}} }
{ \prod_{j} \eta(\countp_j)} % \cdot \frac{1}{ q^{d^u}}
\end{eqnarray}
if $w_j = 0$ for all $j \in S^-_0$. 
The value of $d^u$ is computed in Section~\ref{value-of-d-section}.

\medskip
{\bf Type $B$}
\medskip

An important feature is that $S^-_1$ is always non-empty, since $|\lambda|$ is odd.
In type $B_n$ we work with $G = \mbox{SO}_{2n+1}(\aF_q)$ and thus
a maximal reductive subgroup of $Z_G(e)$ is isomorphic to
the determinant one elements in 
\begin{equation} \label{red_decompB}
\prod_{j \equiv 0} \mbox{Sp}_{\countp_j} (\aF_q) \times \prod_{j \in S_1} \mbox{O}_{\countp_j} (\aF_q).
\end{equation}

Here, we define $A$ to be the elementary abelian $2$-group with basis %$S^+_0 \cup S^+_1$, with
$S_1$, with
elements written as $$c=(c_j)_{j \in S_1},$$ where $c_j \in \FF_2$.
Let $K$ be the subgroup of $A$ consisting of elements $(c_j)$ with $\sum c_j = 0$.
We identify $K$ with the component group $A(e)$, where 
$(c_j) \in K$ corresponds to taking an element of determinant $(-1)^{c_j}$ in the orthogonal group in (\ref{red_decompB}) indexed by $j$ for each 
%$j \in S^+_0 \cup S^+_1$.
$j \in S_1$.

Keeping the same notation as in type $C$, 
%We fix a split $e$ and $c \in A(e)$ and twist $e$ by $c$ to get another rational representative $e_c$ of the orbit $\orbit_e$.
%Under our identification of $K$ and $A(e)$, $c$ corresponds to $(c_j) \in K$.
the group of rational points in a maximal reductive subgroup of $Z_G(e_c)$ is isomorphic to
$$\prod_{j \equiv 0} \mbox{Sp}_{\countp_j} (\FF_q) \times \prod_{j \in S^-_1} \mbox{O}_{\countp_j} (\FF_q) \times 
\prod_{j \in S^+_1} \mbox{O}^{c_j}_{\countp_j} (\FF_q).$$
%where the last groups in the product are either split or twisted orthogonal groups of type $D$ depending on 
%whether $c_j$ is equal to $0$ or $1$.
The cardinality of $Z_{G^F}(e_c)$ is therefore
%The number of points in a maximal reductive part of $Z_{G^F}(e_c)$ is therefore
\begin{equation} \label{points_B}
q^{d^u} \cdot |A(e)|
\prod_{j \equiv 0} \eta(\countp_j)
\prod_{j \in S^-_1} \eta(\countp_j)
\prod_{j \in S^+_1} (q^{\frac{\countp_j}{2}} - (-1)^{c_j}) \cdot \eta(\countp_j -2) .
\end{equation}
%$$\prod_{\stackrel{j \equiv 0}{ \countp_j \equiv 1} } [\countp_j -1]_2! \prod_{j \equiv 1} [\countp_j ]_2!
%\prod_{\stackrel{j \equiv 0}{\countp_j \equiv 0}} [\countp_j -2]_2! (q^{\frac{\countp_j}{2}} - c_i)$$ equals
%%multiplied by an appropriate power of $q$.  
%Here, $c_j = (-1)^{x_j}$.
Getting a common denominator over all the conjugacy classes in $A(e)$, we obtain
\begin{eqnarray}  \label{b_sum1}
\sum_c \frac{ \phi(c)}{|Z_{G^F}(e_c)|} =  \frac{\sum_c \phi(c) \prod_{j \in S^+_1} (q^{ \frac{\countp_j}{2}} + (-1)^{c_j})}
{q^{d^u} |A(e)| \prod_{j} \eta(\countp_j)}.% \cdot \frac{1}{ }.
\end{eqnarray}

To evaluate the sum we use Lemma \ref{even_lemma} with $x_i = q^{\frac{\countp_i}{2}}$, $s = |S_1|$,
and $t = |S^+_1|$.  Note that $t<s$ since $S^-_1$ is non-empty, so the first scenario in the lemma never occurs.
Write $\phi$  as $\phi_w$ for $w \in A$ with $w_i=0$ for some $i \in S^-_1$;  such a $w$ is unique.
Now if $w_j \neq 0$ for any $j \in S^-_1$, then the third scenario in the lemma applies and the expression in \eqref{b_sum1} equals zero.
%\begin{eqnarray}  \label{b_sum2}
%\sum_c \frac{ \phi(c)}{|Z_{G^F}(e_c)|} =  0.
%\end{eqnarray}
On the other hand, if $w_j = 0$ for all $j \in S^-_1$, then since $t<s$, the second scenario of the lemma gives
\begin{eqnarray}  \label{b_sum3}
\sum_c \frac{ \phi(c)}{|Z_{G^F}(e_c)|} =  \frac{q^{ -d^u + \sum_{j \in S^+_1, w_j = 0} \frac{\countp_j}{2}} }
{ \prod_{j} \eta(\countp_j)}.% \cdot \frac{1}{ q^{d^u}}
\end{eqnarray}
The value of $d^u$ is computed in Section~\ref{value-of-d-section}.

\medskip
{\bf Type $D$}
\medskip

%*****s=0 is the very even case***
We proceed as in type $B$ and write $\phi = \phi_w$ for the unique 
%$$w \in \displaystyle \bigoplus_{j \in S^+_0} \FF_2 \oplus \displaystyle  \bigoplus_{j \in S^+_1} \FF_2$$ 
$w \in A$ with $w_i = 0$ for some $i \in S^-_1$ when $|S^-_1| > 0$.
%We retain the notation introduced in the type $B$ case.  
Now, $|\lambda|$ is even and so $|S^-_1|$ is even, but it could be zero. Therefore, 
in evaluating the sum $\sum_c \frac{ \phi(c)}{|Z_{G^F}(e_c)|}$ all three scenarios in Lemma \ref{even_lemma} can occur.  When $s >0$, we 
therefore have
\begin{align}  \label{d_sum3}
\sum_c \frac{ \phi(c)}{|Z_{G^F}(e_c)|} =  \frac{q^{-d^u}}{ \prod_{j} \eta(\countp_j) } \cdot
\begin{cases}
q^{ \sum_{j \in S^+_1, w_j = 1} \frac{\countp_j}{2}} + q^{ \sum_{j \in S^+_1, w_j = 0} \frac{\countp_j}{2} } & \text{if } |S^-_1|=0 \\
q^{ \sum_{j \in S^+_1, w_j = 0} \frac{\countp_j}{2}} & \text{if } |S^-_1| > 0, w_j = 0 \text{ for all } j \in S^-_1\\
0    & \text{otherwise}
\end{cases}
%\number
\end{align}

\begin{remark} \label{even_orbs}
The case of $s=0$ is equivalent to the partition $\lambda$ having only even parts.   In such cases, there are two $\mbox{SO}_{2n}(\aF_q)$-orbits corresponding to the same 
$\lambda$.  In each of these cases, $A(e)$ is trivial and the formula for $|S^-_1|=0$ above is correct if we interpret it as the sum over two elements, $e_1$ and $e_2$, one in each of the two  $\mbox{SO}_{2n}(\aF_q)$-orbits:  $\frac{1}{|Z_{G^F}(e_1)|} + \frac{1}{|Z_{G^F}(e_2)|}$.  
In other words, the value, when multiplied by $|G|$, gives the number of points in the $\mbox{O}_{2n}(\aF_q)$-orbit through either element.
\end{remark}

\subsection{Value of $d^u$}
\label{value-of-d-section}

Recall that  $d^u$ is the dimension of a maximal unipotent subgroup of $Z_G(e)$.
Let $d^u_1$ be the dimension of the unipotent radical of $Z_G(e)$ and $d^{u}_2$ 
the dimension of a maximal unipotent subgroup of the reductive part $Z_{red}$ of $Z_G(e)$.  
Then $d^u = d_1^u + d_2^u$.  
Since $d^u_2$ is the number of positive roots for $Z_{red}$, we can compute its value 
from the known type of $Z_{red}$ given previously.  The value of $d^{u}_1$ can be found in \cite[pp. 398-9]{Carter}.
Recall that $\dualp$ is the dual partition for $\lambda$
and that $c(\lambda):=\sum_j \lambda'_j \lambda'_{j+1}$.

\medskip
\noindent  {\bf Type} $C$.
\medskip

We have
$$d^{u}_1  = \frac{1}{2} \left( \sum (\dualp_j)^2 - \sum \countp_j^2 + \sum_{j \equiv 0} \countp_j \right) \text{ and }$$
\begin{eqnarray*}
%d^{u}_e  =   \sum_{j \equiv 1} (\frac{\mu_j}{2})^2 + \sum_{\stackrel{j \equiv 0}{ \countp_j \equiv 1}}  (\frac{\mu_j-1}{2})^2 +
%\sum_{\stackrel{j \equiv 0}{\countp_j \equiv 0}}  ((\frac{\mu_j}{2})^2  - \frac{\mu_j}{2} )
% & = & \\
%\frac{1}{4} \sum \mu_j^2  - \frac{1}{2} \sum_{j \equiv 0} \mu_j+  \frac{ L(\lambda) }{4},  \label{pos_roots_C}
d^{u}_2 =   
%\sum_{j \equiv 1} (\frac{\mu_j}{2})^2 + \sum_{\stackrel{j \equiv 0}{ \countp_j \equiv 1}}  (\frac{\mu_j-1}{2})^2 +
%\sum_{\stackrel{j \equiv 0}{\countp_j \equiv 0}}  ((\frac{\mu_j}{2})^2  - \frac{\mu_j}{2} )
\sum_{j \equiv 1} \frac{\mu^2_j}{4} + \sum_{\stackrel{j \equiv 0}{ \countp_j \equiv 1}}  \frac{(\mu_j-1)^2}{4} +
\sum_{\stackrel{j \equiv 0}{\countp_j \equiv 0}}  \left(\frac{\mu^2_j}{4}  - \frac{\mu_j}{2} \right)
  = \frac{1}{4} \sum \mu_j^2  - \frac{1}{2} \sum_{j \equiv 0} \mu_j+  \frac{ L(\lambda) }{4},  \label{pos_roots_C}
\end{eqnarray*}
where $L(\lambda)$ is the number of $\countp_j$ that are odd. 
Hence, since $\lambda'_1 = \ell(\lambda)$, 
\begin{eqnarray}
\nonumber
%d_e = d^{u}_e + d^{ra}_e & =  & 
d^u = d^{u}_1 + d^{u}_2 & =  & 
{ \frac{1}{2} \sum (\dualp_j)^2 - \frac{1}{4} \sum \mu_j^2 + \frac{ L(\lambda) }{4}} \\ \nonumber
& = &  { \frac{1}{2} \sum (\dualp_j)^2 - \frac{1}{4} \sum (\dualp_j  -\dualp_{j+1})^2 + \frac{ L(\lambda) }{4}} \\
& = &  \frac{\ell(\lambda)^2}{4} + \frac{c(\lambda)}{2} + \frac{ L(\lambda) }{4}.
\end{eqnarray}

\medskip
\noindent {\bf Type} $B$ and $D$.
\medskip

We have
$$d^{u}_1  =\frac{1}{2} \left( \sum (\dualp_j)^2 - \sum \countp_j^2 - \sum_{j \equiv 0} \countp_j \right) \text{  and }
d^{u}_2  = \frac{1}{4} \sum \countp_j^2 - \frac{1}{2} \sum_{j \equiv 1} \countp_j + \frac{ L(\lambda) }{4}.$$
 Hence, since $\sum  \countp_j = \ell(\lambda)$,
 \begin{eqnarray}
\nonumber
d^u = d^{u}_1 + d^{u}_2  & =  & 
 \frac{1}{2} \sum (\dualp_j)^2 - \frac{1}{4} \sum \mu_j^2 + \frac{ L(\lambda) }{4} - \frac{\ell(\lambda)}{2} \\
 %\nonumber
  & =  &  \frac{\ell(\lambda)^2}{4} - \frac{\ell(\lambda)}{2}  + \frac{c(\lambda)}{2} + \frac{ L(\lambda) }{4}.   \label{BD_d_u}
  %& = &  \binom{\ell(\lambda)}{2} + \frac{c(\lambda)}{2} + \frac{ L(\lambda) }{4}
 \end{eqnarray}
%where we have used the identity $\sum  \countp_j = \ell(\lambda)$ and the same simplication as in type $C$.

%%%%%%%%%%%%%%%%%%%%%%%%%%%%%%%%%%%%%%%%%%%%%%%%%%%%%%%%%%%%%%%%%%%
\subsection{Finishing the $f_{e, \phi}$ calculation in types $B, C, D$}   
\label{BCD_calcs}
%%%%%%%%%%%%%%%%%%%%%%%%%%%%%%%%%%%%%%%%%%%%%%%%%%%%%%%%%%%%%%%%%%%

We first handle types $B_n$ and $C_n$.
%Set $\epsilon = 1$ in type $B$ and $\epsilon =0$ in type $C$.
%As before,  $e \in \0_{\lambda}$, with
%$\lambda$ in $\Par_B(2n+1)$ or $\Par_C(2n)$ according to whether $\ggg$ is of type $B_n$ or type $C_n$, respectively.
By \cite{LeSh} and \cite{Spalt}, we have
\begin{equation}
\label{type-BC-kappa-formula}
\kappa = \left\lfloor \frac{\ell(\lambda)}{2} \right\rfloor \text{ and }
\end{equation}
%and
$$
(m_1, \dots, m_{\kappa})=(1, 3, \dots, 2\kappa-1).
$$
%then there are
%\begin{equation}
%\label{type-BC-kappa-formula}
%\kappa = \left\lfloor \frac{\ell(\lambda)}{2} \right\rfloor
%\end{equation}
%occurrences of the representation $V$ in the cohomology
%of the Springer fiber $\BBB_{e_\lambda}$, and
%$$
%(m_1, \dots, m_{\kappa})=(1, 3, \dots, 2\kappa-1).
%$$
Moreover, $\pi_{\kappa}$ is always trivial (see \cite{Sommers2}), so $f_{e, \phi}$ is computed by \eqref{formula:f_easy}, which becomes
\begin{equation} \label{BC_almost}
f_{e, \phi} = 
%q^{m(n - \kappa)+  \kappa^2} \prod_{j=1}^{\kappa} \big (q^{m-(2j-1)} - 1 \big)
q^{m(n - \sBC)+   \sBC^2} \prod_{j=1}^{\sBC} \big (q^{m-(2j-1)} - 1 \big)
\big(  \sum_c  \frac{ \phi(c)} {|Z_{G^F}(e_c)|} \big).
\end{equation}

%We introduce some notation.  
%Let $\epsilon =0$ in type $C$ and let $\epsilon= 1$ in types $B$ and $D$.  
We introduce the following notation 
%for the sum in the exponent
%of $q$ that appears in the previous formulas,
\begin{equation}  \label{beta}
\beta_{\epsilon}(\lambda, w) :=  \sum_{\stackrel{ j \equiv \epsilon, \countp_j \equiv 0}{w_j = 0}} \frac{\countp_j}{2}
\end{equation}
%and %define 
%$\delta(\lambda) := \sBC^2 - \frac{\ell(\lambda)^2}{4}$, which equals
%$0$ when $\ell(\lambda)$ is even and $\frac{1}{4}-\frac{\ell(\lambda)}{2}$ when $\ell(\lambda)$ is odd.
and recall from the introduction that 
$$
\delta(\lambda) := \sBC^2 - \frac{\ell(\lambda)^2}{4} =
\begin{cases}
0 & \text{ for }\ell(\lambda)\text{ even}, \\
\frac{1}{4}-\frac{\ell(\lambda)}{2} & \text{ for }\ell(\lambda)\text{ odd}.
\end{cases}
$$
%$$\beta_0(\lambda, w; \epsilon) = \sum_{\stackrel{ j \equiv \epsilon, \countp_j \equiv 0}{w_j = 0}} \frac{\countp_j}{2}
%\,\, \text{  and    } \,\,
%\beta_1(\lambda, w; \epsilon) = \sum_{\stackrel{ j \equiv \epsilon, \countp_j \equiv 0}{w_j = 1}} \frac{\countp_j}{2}.$$
and $\hat{N}:=\lfloor N/2 \rfloor$, and
for %a sequence
$\nu=(\nu_1,\nu_2,\ldots)$ that
$\hat{\nu}:=(\hat{\nu}_1, \hat{\nu}_2, \ldots)$.

\begin{proposition}  \label{BC_most_general}
Write $\phi=\phi_w$ as in Section~\ref{compute-summation}.  
For $\ggg$ of type $B_n$ or $C_n$,  then $f_{e,\phi}$ equals zero unless $w_j = 0$ for all $j \in S^{-}_{\epsilon}$, where $\epsilon = 0$ for type $C$ and 
$\epsilon = 1$ for type $B$, 
in which case  $f_{e,\phi}$ equals
$$
%q^{m\left( n-\sBC \right) 
q^{m\left( n-\hat{\ell}(\lambda) \right) + \frac{1}{4} -  \frac{c(\lambda)}{2} - \frac{ L(\lambda) }{4}  + \beta_{1}(\lambda, w)} \cdot \prod_{i=1}^{\hat{L}(\lambda)} (q^{m-2i+1}-1)
\qbin{\hat{m}-\hat{L}(\lambda)}{\hat{\mu}(\lambda)}{q^2}   \text{ (Type $B_n$)  }   %\text{ if } w_j = 0 \text{ for all } j \in S^{-}_{1}\\
$$
$$
 q^{m\left( n-\hat{\ell}(\lambda) \right) + \delta(\lambda) -  \frac{c(\lambda)}{2} - \frac{ L(\lambda) }{4}  + \beta_{0}(\lambda, w)} \cdot \prod_{i=1}^{\hat{L}(\lambda)} (q^{m-2i+1}-1)
\qbin{\hat{m}-\hat{L}(\lambda)}{\hat{\mu}(\lambda)}{q^2}  \mbox{ (Type $C_n$)  }
$$
\end{proposition}

\begin{proof}
Using \eqref{c_sum3} and \eqref{b_sum3}
and the fact that  
$$\prod_{j=1}^{\hat{\ell}(\lambda) } \big (q^{m-(2j-1)} - 1 \big) =  \frac{\eta(m)}
%      {\eta \left(2\hat{m}-2 \sBC \right) 
      {\eta \left(m- \ell(\lambda) \right)},$$
the expression in \eqref{BC_almost}, when nonzero, equals
\begin{equation} 
f_{e, \phi} = q^{m(n -  \hat{\ell}(\lambda) )+   \hat{\ell}(\lambda) ^2 - d^u + \beta_{\epsilon}(\lambda, w)} \cdot \frac{\eta(m)}
%      {\eta \left(2\hat{m}-2 \sBC \right) 
      {\eta \left(m- \ell(\lambda) \right) 
%          \cdot \eta(\mu_1(\lambda) ) \cdot \eta(\mu_2(\lambda)) \cdots }
          \cdot \eta(\mu_1 ) \cdot \eta(\mu_2) \cdots }.
\end{equation}
Next, we have
%$$\prod_{j=1}^{\sBC} \big (q^{m-(2j-1)} - 1 \big) =  \frac{\eta(m)}
%%      {\eta \left(2\hat{m}-2 \sBC \right) 
%      {\eta \left(m- \ell(\lambda) \right)} $$
%   and so 
\begin{equation} \nonumber
\label{type-BC-multinomial}
%f_{\lambda,m}(q):=\frac{\eta(2\hat{m})}
\frac{\eta(m)}
%      {\eta \left(2\hat{m}-2 \sBC \right) 
      {\eta \left(m- \ell(\lambda) \right) 
%          \cdot \eta(\mu_1(\lambda) ) \cdot \eta(\mu_2(\lambda)) \cdots }
          \cdot \eta(\mu_1 ) \cdot \eta(\mu_2) \cdots }
= \prod_{i=1}^{\hat{L}(\lambda)} (q^{m-2i+1}-1)
\qbin{\hat{m}-\hat{L}(\lambda)}{\hat{\mu}(\lambda)}{q^2},
\end{equation}
since $\hat{\ell}(\lambda) = \hat{L}(\lambda)  +  |\hat{\mu}(\lambda)|$.
The results follow after substituting in the appropriate value of $d^u$ from \S \ref{value-of-d-section}.

%By \eqref{c_sum3} and \eqref{b_sum3}, \eqref{BC_almost} becomes
%\begin{equation} 
%f_{e, \phi} = q^{m(n -  \sBC)+   \sBC^2 - d^u + \beta_{\epsilon}(\lambda, w)} \cdot h_{\lambda,m}(q)
%\end{equation}
%
%The power of $q$ in \eqref{BC_almost} equals 

\end{proof}

Recall from the introduction that 
$$\tau_{\epsilon}(\lambda) = %\frac{1}{2}
\sum_{\substack{j \equiv \epsilon, \countp_j \equiv 0}} \frac{\countp_j}{2},%\mu_j.
$$
which is the value of is $\beta_{\epsilon}(\lambda, 0)$ when $w=0$, which corresponds to the trivial character of $A(e)$.
Thus the results in the Proposition 
simplify to those in Theorem~\ref{q-Kreweras-formula-theorem} (Types $B_n$ and $C_n$).

\begin{comment}

\begin{corollary}
%For $\lambda \in \Par_B(2n+1)$ that
\begin{equation}
\label{type-B-q-Kreweras}
\Krew(B_n,\0_\lambda,m;q) = q^{m\left( n-\hat{\ell}(\lambda)  \right)  + 
\frac{1}{4} -  \frac{c(\lambda)}{2} - \frac{ L(\lambda) }{4} + \tau_1(\lambda)} \cdot \prod_{i=1}^{\hat{L}(\lambda)} (q^{m-2i+1}-1)
\qbin{\hat{m}-\hat{L}(\lambda)}{\hat{\mu}(\lambda)}{q^2} ,
\end{equation}
and %for $\lambda \in \Par_C(2n)$ that
\begin{equation}
\label{type-C-q-Kreweras}
\Krew(C_n,\0_\lambda,m;q) =q^{m\left( n-\hat{\ell}(\lambda)  \right)  + \delta(\lambda) - \frac{c(\lambda)}{2} - \frac{ L(\lambda) }{4}  + \tau_0(\lambda)} \cdot \prod_{i=1}^{\hat{L}(\lambda)} (q^{m-2i+1}-1)
\qbin{\hat{m}-\hat{L}(\lambda)}{\hat{\mu}(\lambda)}{q^2}.
\end{equation}
\end{corollary}

\end{comment}

\medskip

\noindent {\bf Type $D_n$}.

\medskip

\bigskip

We now turn to type $D_n$.  Here the values of $m_1, m_2, \dots, m_{\kappa}$ depend on both $\ell(\lambda)$ and the parity of $\mu_1$ \cite{Spalt}:
\begin{itemize}
\item
When $\mu_1$ is odd, $\kappa =  \frac{\ell(\lambda)}{2} -1$ and 
$$(m_1,m_2,\ldots,m_{\kappa})=(1, 3, \dots, 2\kappa -1).$$
\item
When $\mu_1$ is even, $\kappa =  \frac{\ell(\lambda)}{2}$ and 
$$
(m_1, m_2, \dots, m_{\kappa})=
\left(1, 3, \dots, 2\kappa -3,
\ell(\lambda)  - \frac{\mu_1}{2} -1\right).
$$
\end{itemize}

What complicates the type $D_n$ picture is that $\pi_{\kappa}$ may be non-trivial when $\mu_1$ is even.  
It is always trivial when $\mu_1$ is odd.   
Let us now describe when this happens and what $\pi_{\kappa}$ is.   
To that end, we define a subgroup $H \subset A(e)$.   Suppose that $\mu_1$ is even and nonzero.  Then
$e$ lies in a proper Levi subalgebra $\levi$ of $\ggg$ of type $D$ of semisimple rank $n - \frac{\mu_1}{2}$.   
Now if $\lambda$ contains an odd part different from $1$, then $A(e)$ will be nontrivial and moreover 
the component group of $e$ relative to $\levi$ defines an index two subgroup of $A(e)$, which we denote $H$.
We can now recall
%If $r$ is even and $k \neq r$, then $\lambda$ contains $k-r$ parts equal to $1$ and 
%$e$ lies in a proper Levi subalgebra $\levi$ of $\ggg$ of type $D_{l}$ with $l = n - \frac{k-r}{2}$.   
%If $\lambda$ contains an odd part different than $1$, then $A(e)$ will be nontrivial and moreover 
%the component group of $e$ in $\levi$ defines an index two subgroup of $A(e)$, which we denote $H$.
\begin{proposition}{\cite[Proposition 10]{Sommers2}} \label{me:pij}  
%For $j \in \{ 1, 2, \dots, s-1 \}$, $\pi_j$ is trivial.
If $\mu_1$ is odd or $\mu_1=0$ or $\mu_j = 0$ for all odd $j >1$ , then $\pi_{\kappa}$ is trivial.
%% could say lambda has no odd part different from one *or* lambda has only one odd part
Otherwise, $\pi_{\kappa}$ is the 
nontrivial representation of $A(e)$ which is trivial on the subgroup $H$.
%If $\mu_1$ is even and $\lambda$ has , then $\pi_s$ is trivial unless $k \neq r$, and $\lambda$ has an odd part different from $1$.
%If not, then $\pi_s$ is the 
%nontrivial representation of $A(e)$ which takes value $1$ on $H$.
\end{proposition}

%%don't need this with new format
\begin{comment}
Let 
$$g_{\lambda, m}:= \frac{ \displaystyle \prod_{j=1}^{  \hat{\ell}(\lambda) - 1} (q^{m-2j+1} - 1) }
{  \displaystyle \prod_{j} \eta(\countp_j)   } = 
\begin{cases}
\displaystyle \prod_{j=1}^{\hat{L}(\lambda)-1} (q^{m-2j+1}-1) \cdot
\qbin{\hat{m}+1 - \hat{L}(\lambda)}{\hat{\mu}(\lambda)}{q^2} & \text{ if } \hat{L}(\lambda) \geq 1  \\
%\frac{1}{q^{m+1}-1}\cdot
\tfrac{1}{[\hat{m} +1]_{q^2} }\cdot
\qbin{\hat{m} +1}{\hat{\mu}(\lambda)}{q^2} & \text{ if } \hat{L}(\lambda) = 0.
\end{cases}
$$
\end{comment}

We can now give the formula for $f_{e, \phi}$.

\begin{proposition}  
\label{D_most_general}
Write $\phi = \phi_w$ as in the Type D subsection of Section~\ref{compute-summation}.
%. for the unique 
%$$w \in \displaystyle \bigoplus_{j \in S^+_0} \FF_2 \oplus \displaystyle  \bigoplus_{j \in S^+_1} \FF_2$$ 
%with $w_j = 0$ for at least one $j \in S^+_1$ when $|S^+_1| > 0$.
Then $f_{e, \phi} = 0$ unless $w_j = 0$ for all $j \in S^-_1$.
Otherwise, 

\begin{equation*} \nonumber
f_{e, \phi} = q^{
\BCDexponent
%m(n - \hat{\ell}(\lambda)) 
%- \frac{c(\lambda)}{2} - \frac{ L(\lambda) }{4}
+ \beta_1(\lambda, w)}  % \cdot g_{\lambda, m}
\text{ multiplied by } \\
\end{equation*} 

\begin{eqnarray*} \label{Factor}
\begin{cases}
q^{
m  + 1 - \hat{\ell}(\lambda)
} \cdot \displaystyle \prod_{j=1}^{\hat{L}(\lambda)-1} (q^{m-2j+1}-1) \cdot
\qbin{\hat{m}+1 - \hat{L}(\lambda)}{\hat{\mu}(\lambda)}{q^2} 
& \text{ if } \mu_1 \text{ is odd}  \\

%%% Next

q^{\hat{\ell}(\lambda) -\mu_1}
\displaystyle \prod_{j=1}^{\hat{L}(\lambda)} (q^{m-2j+1}-1) \cdot 
\qbin{\hat{m}-\hat{L}(\lambda)}{\hat{\mu}_{\geq 2}(\lambda)}{q^2} 
%%%\footnote{correction here: Vic should double-check}
\qbin{\hat{m}+ 1- \hat{L}(\lambda) -|\hat{\mu}_{\geq 2}(\lambda)|}
     {\hat{\mu}_1(\lambda)} {q^2} 
& \text{ if } \mu_1 \text{ is even, } w_1 = 0, \text{ and } \hat{L}(\lambda) \geq 1 \\  

%%%% Next

q^{\hat{\ell}(\lambda) } 
\displaystyle \prod_{j=1}^{\hat{L}(\lambda)} (q^{m-2j+1}-1)  \cdot
\qbin{\hat{m} - \hat{L}(\lambda)}{\hat{\mu} (\lambda)}{q^2}  
& \text{ if } \mu_1 \text{ is even, } w_1 = 1, \text{ and } \hat{L}(\lambda) \geq 1 \\  

%%%%% Next
q^{\hat{\ell}(\lambda) + \tau_1(\lambda) - 2\beta_1(\lambda,w)}
  \qbin{\hat{m}}{\hat{\mu}(\lambda)}{q^2} 
  +
q^{\hat{\ell}(\lambda) -  \mu_1}
  \qbin{\hat{m}}{\hat{\mu}_{\geq 2}(\lambda)}{q^2} 
\qbin{\hat{m}+1-|\hat{\mu}_{\geq 2}(\lambda)|}
     {\hat{\mu}_1(\lambda)} {q^2} 
%NOTE: $\tau(\lambda) \geq \mu_1$;
& \text{ if } w_1=0 \text{ and } \hat{L}(\lambda) = 0    \\

%%%%% Next
\nonumber q^{\hat{\ell}(\lambda) } 
\qbin{\hat{m} }{\hat{\mu}(\lambda)}{q^2}
+
q^{\hat{\ell}(\lambda) - \mu_1 + \tau_1(\lambda) -2 \beta_1(\lambda,w)}
\qbin{\hat{m} }{\hat{\mu}_{\geq 2} (\lambda)}{q^2}  
 \qbin{\hat{m}+1 -  |\hat{\mu}_{\geq 2}(\lambda)|}{\hat{\mu}_{1}(\lambda)}{q^2}
& \text{ if } w_1=1 \text{ and } \hat{L}(\lambda) = 0    \\
\end{cases}
\end{eqnarray*} 
In fact, since in the last case $w$ can always be taken to have $w_1=0$, we only need the first formula.
\end{proposition}

\begin{proof}

When $\mu_1$ is odd, then $\pi_{\kappa}$ is trivial, so formula \eqref{formula:f} becomes  
\begin{eqnarray}  \label{formula:d, odd}
\nonumber 
f_{e, \phi} & =  & q^{m(n -  \frac{\ell(\lambda)}{2} + 1)+ ( \frac{\ell(\lambda)}{2} - 1)^2} \prod_{j=1}^{\hat{\ell}(\lambda) - 1} (q^{m-(2j-1)} - 1)
 \sum_c  \frac{ \phi(c)} {|Z_{G^F}(e_c)|}    \\
\nonumber   & = &  q^{m(n -  \frac{\ell(\lambda)}{2} + 1)+ ( \frac{\ell(\lambda)}{2} - 1)^2 }  
	\frac{\eta(m)}  {\eta \left(m+1- \ell(\lambda) \right) } \cdot
     \frac{ q^{- d^u + \beta_1(\lambda,w)}  }{\eta(\mu_1 ) \cdot \eta(\mu_2) \cdots }
 \end{eqnarray}
using the second part of \eqref{d_sum3} since $|S^-_1| > 0$.  
The result follows from the formula for $d^u$  in \eqref{BD_d_u} % and the definition of $\beta_1(\lambda,w)$ in \eqref{beta}.
and the identity
\begin{equation}   \label{D_version_multinomial}
\frac{\eta(m)}
      {\eta \left(m+1- \ell(\lambda) \right) 
          \cdot \eta(\mu_1 ) \cdot \eta(\mu_2) \cdots }
= 
\begin{cases}
\displaystyle \prod_{j=1}^{\hat{L}(\lambda)-1} (q^{m-2j+1}-1) \cdot
\qbin{\hat{m} - ( \hat{L}(\lambda) -1)}{\hat{\mu}(\lambda)}{q^2} & \text{ if } \hat{L}(\lambda) \geq 1  \\
\tfrac{1}{[\hat{m} +1]_{q^2} }\cdot
\qbin{\hat{m} +1}{\hat{\mu}(\lambda)}{q^2} & \text{ if } \hat{L}(\lambda) = 0,
\end{cases}
\end{equation}
which uses the equality $\hat{\ell}(\lambda) = \hat{L}(\lambda)  +  |\hat{\mu}(\lambda)|$.

Next when $\mu_1$ is even, formula \eqref{formula:f} becomes 
\begin{equation} \label{formula:d, even}
f_{e, \phi} = 
q^{m(n -  \frac{\ell(\lambda)}{2})+  \frac{\ell(\lambda)^2}{4} - \frac{\mu_1}{2}}
 \prod_{j=1}^{\hat{\ell}(\lambda) - 1} (q^{m-(2j-1)} - 1)
 \cdot \left(  q^{m- (\ell(\lambda) - \frac{\mu_1}{2} - 1)}  \sum_c  \frac{ \phi(c)} {|Z_{G^F}(e_c)|}  -  \sum_c   \frac{ \pi_{\kappa}(c) \phi(c)} {|Z_{G^F}(e_c)|}  \right),
\end{equation}
since $\pi_{\kappa}$ is always one-dimensional.

When $\mu_1$ is even and $\hat{L}(\lambda) \geq 1$, 
since $|S^-_1| > 0$, the second part of \eqref{d_sum3} is used to evaluate each sum in \eqref{formula:d, even}.
In this case, $\pi_k$ is nontrivial if and only if $\mu_1$ is nonzero.  
For $\phi = \phi_w$, the vector corresponding to $\pi_k \phi_w$ is the same as $w$ but with the parity of $w_1$ changed.   
Thus as in the previous case, \eqref{formula:d, even} becomes
\begin{eqnarray}   
\label{first_step1}
&\\
%\nonumber f_{e, \phi} =  q^{m(n -  \frac{\ell(\lambda)}{2})+  \frac{\ell(\lambda)^2}{4} - \frac{\mu_1}{2} -d^u +\beta_1(\lambda,w)}  
\nonumber
q^{m(n -  \frac{\ell(\lambda)}{2})+  \frac{\ell(\lambda)^2}{4} - \frac{\mu_1}{2}}  
& \displaystyle \prod_{j=1}^{\hat{L}(\lambda)-1} (q^{m-2j+1}-1)  \cdot  
 \qbin{\hat{m}+1 - \hat{L}(\lambda)}{\hat{\mu}(\lambda)}{q^2}     q^{-d^u +\beta_1(\lambda,w)} 
\begin{cases}
 q^{m +1 - \ell(\lambda) + \frac{\mu_1}{2} }  - q^{ - \frac{\mu_1}{2} }     & \text{ if } w_1 = 0 \\
 q^{m +1 - \ell(\lambda) + \frac{\mu_1}{2} }  - q^{ \frac{\mu_1}{2} }  & \text{ if } w_1 = 1.   \\
\end{cases}
\end{eqnarray}
%(the formula remains valid even if $\mu_1 = 0$.)
The following two identities are easy to verify, for $A \in \NN$ and $\nu$ a partition with $|\nu| \leq A$:
$$
\qbin{A}{\nu}{q} = \qbin{A}{\nu_{\geq 2}}{q} \qbin{A - |\nu_{\geq 2}|}{\nu_{1}}{q} 
\text{ and } 
%[A+1 - |\nu| ] \qbin{A+1}{\nu}{q}  = [A+1] \qbin{A}{\nu}{q}.
(q^{A+1 - |\nu|}-1) \qbin{A+1}{\nu}{q}  = (q^{A+1}-1) \qbin{A}{\nu}{q}.
$$
Using these identities and the identity
$\hat{L}(\lambda) +  |\hat{\mu}_{\geq 2}(\lambda)| = \hat{\ell}(\lambda) - \hat{\mu}_1(\lambda)$,
we have 
\begin{eqnarray}  \label{w1=0}
\nonumber
\qbin{\hat{m}+1 - \hat{L}(\lambda)}{\hat{\mu}(\lambda)}{q^2}
%\cdot
\left( q^{m +1 - \ell(\lambda) + \frac{\mu_1}{2} }  - q^{ - \frac{\mu_1}{2} }   \right) = 
q^{- \hat{\mu}_1 } 
%\displaystyle \prod_{j=1}^{\hat{L}(\lambda)-1} (q^{m-2j+1}-1) \cdot
%[\hat{m} + 1 - \hat{\ell}(\lambda) + \hat{\mu}_1]_{q^2} 
((q^2)^{\hat{m} + 1 - \hat{\ell}(\lambda) + \hat{\mu}_1}-1) 
\qbin{\hat{m}+1 - \hat{L}(\lambda)}{\hat{\mu}(\lambda)}{q^2}  =   \\
%\qbin{\hat{m}+1 - \hat{L}(\lambda)}{\hat{\mu}(\lambda)}{q^2}  =   \\
\nonumber 
q^{-  \hat{\mu}_1 } 
%\displaystyle \prod_{j=1}^{\hat{L}(\lambda)-1} (q^{m-2j+1}-1) 
%[\hat{m} + 1 - \hat{L}(\lambda) -  |\hat{\mu}_{\geq 2}(\lambda)| ]_{q^2}  
((q^2)^{\hat{m} + 1 - \hat{\ell}(\lambda) + \hat{\mu}_1}-1) 
\qbin{\hat{m}+1 - \hat{L}(\lambda)}{\hat{\mu}_{\geq 2} (\lambda)}{q^2}  
 \qbin{\hat{m}+1 - \hat{L}(\lambda) -  |\hat{\mu}_{\geq 2}(\lambda)|}{\hat{\mu}_{1}(\lambda)}{q^2} = \\
 \nonumber q^{-  \hat{\mu}_1 } 
%\displaystyle \prod_{j=1}^{\hat{L}(\lambda)-1} (q^{m-2j+1}-1) 
%[\hat{m} + 1 - \hat{L}(\lambda) ]_{q^2}  
((q^2)^{\hat{m} + 1 - \hat{L}(\lambda)}-1)
\qbin{\hat{m} - \hat{L}(\lambda)}{\hat{\mu}_{\geq 2} (\lambda)}{q^2}  
 \qbin{\hat{m}+1 - \hat{L}(\lambda) -  |\hat{\mu}_{\geq 2}(\lambda)|}{\hat{\mu}_{1}(\lambda)}{q^2} 
 % = \\
%  \nonumber q^{\hat{\ell}(\lambda) - \mu_1 } 
%\displaystyle \prod_{j=1}^{\hat{L}(\lambda)} (q^{m-2j+1}-1) 
%\qbin{\hat{m} - \hat{L}(\lambda)}{\hat{\mu}_{\geq 2} (\lambda)}{q^2}  
% \qbin{\hat{m}+1 - \hat{L}(\lambda) -  |\hat{\mu}_{\geq 2}(\lambda)|}{\hat{\mu}_{1}(\lambda)}{q^2}.
\end{eqnarray} 
and  %in the $w_1 = 1$ case:
\begin{eqnarray}  \label{w1=1}
\qbin{\hat{m}+1 - \hat{L}(\lambda)}{\hat{\mu}(\lambda)}{q^2}
%\cdot
\left( q^{m +1 - \ell(\lambda) + \frac{\mu_1}{2} }  - q^{ \frac{\mu_1}{2} }   \right)   = 
\nonumber q^{ \hat{\mu}_1 } 
%\displaystyle \prod_{j=1}^{\hat{L}(\lambda)-1} (q^{m-2j+1}-1) \cdot
%[\hat{m} + 1 - \hat{\ell}(\lambda) ]_{q^2} 
((q^2)^{\hat{m} + 1 - \hat{\ell}(\lambda)}-1)
\qbin{\hat{m}+1 - \hat{L}(\lambda)}{\hat{\mu}(\lambda)}{q^2}  =   \\
%\qbin{\hat{m}+1 - \hat{L}(\lambda)}{\hat{\mu}(\lambda)}{q^2}  =   \\
\nonumber q^{ \hat{\mu}_1 } 
%\displaystyle \prod_{j=1}^{\hat{L}(\lambda)-1} (q^{m-2j+1}-1) 
%[\hat{m} + 1 - \hat{L}(\lambda) -  |\hat{\mu}(\lambda)|]_{q^2}  
((q^2)^{\hat{m} + 1 - \hat{L}(\lambda) -  |\hat{\mu}(\lambda)|}-1)  
\qbin{\hat{m}+1 - \hat{L}(\lambda)}{\hat{\mu}(\lambda)}{q^2}  =   
% \nonumber q^{\hat{\ell}(\lambda)} 
%\displaystyle \prod_{j=1}^{\hat{L}(\lambda)-1} (q^{m-2j+1}-1) 
% [\hat{m} + 1 - \hat{L}(\lambda) ]_{q^2}  
%\qbin{\hat{m} - \hat{L}(\lambda)}{\hat{\mu}_{\geq 2} (\lambda)}{q^2}  
% \qbin{\hat{m}+1 - \hat{L}(\lambda) -  |\hat{\mu}_{\geq 2}(\lambda)|}{\hat{\mu}_{1}(\lambda)}{q^2}  = \\
  \nonumber q^{ \hat{\mu}_1  } 
%\displaystyle \prod_{j=1}^{\hat{L}(\lambda)} (q^{m-2j+1}-1)
((q^2)^{\hat{m} + 1 - \hat{L}(\lambda)}-1) 
\qbin{\hat{m} - \hat{L}(\lambda)}{\hat{\mu} (\lambda)}{q^2}  
% \qbin{\hat{m}+1 - \hat{L}(\lambda) -  |\hat{\mu}_{\geq 2}(\lambda)|}{\hat{\mu}_{1}(\lambda)}{q^2}.
\end{eqnarray}  
and the result follows for this case by inserting these values into \eqref{first_step1}.

Finally when $\mu_1$ is even and $\hat{L}(\lambda) = 0$, we have $S^-_1 = \emptyset$.  Thus there are two possible choices for $w$:  $w$ and $w + (1,1, \dots, 1)$ will both give the same $\phi_w$.  The formulas will not depend on the choice. 
The two summations in \eqref{formula:d, even} will both make use of the first part of  \eqref{d_sum3}.
The calculation is similar to the previous case after noting that 
$\tau_1(\lambda) - \beta_1(\lambda,w) = \sum_{j \in S^+_1, w_j = 1} \frac{\countp_j}{2}.\qedhere$ 

\end{proof}

\noindent
At $w=0$, Proposition~\ref{D_most_general} is 
Theorem~\ref{q-Kreweras-formula-theorem} (Type $D_n$),
since $\beta_1(\lambda,w) = \tau_1(\lambda)$ and 
$\tau_1(\lambda) - 2\beta_1(\lambda,w) = -\tau_1(\lambda)$.

\begin{remark}
When $\lambda$ has only even parts, then in particular $\mu_1 = 0$ and $\hat{L}(\lambda) =0$ and $\tau_1(\lambda)=0$.
We are in the last case.  Then  the expression in the Corollary simplifies to
$
2 q^{
\BCDexponent
%m(n-\hat{\ell}(\lambda))-\frac{c(\lambda)}{2}-\frac{L(\lambda)}{4} 
+ \hat{\ell}(\lambda)} \qbin{\hat{m}}{\hat{\mu}(\lambda)}{q^2}.
$
This value is actually twice the value of $f_{e,1}$ for $e$ in either nilpotent orbit associated to $\lambda$.  See Remark \ref{even_orbs}.
% , relative to the connected adjoint group, associated to $\lambda$.
\end{remark}

%%%%%%%%%%%%%%%%%%%%%%%%%%%%%%%%%%%%%%%%%%%%%%%%%%%%%%%%%%%%%%%%%
\subsection{The $f_{e, \phi}$ for the exceptional groups}  
\label{Exceptional_calcs}
%%%%%%%%%%%%%%%%%%%%%%%%%%%%%%%%%%%%%%%%%%%%%%%%%%%%%%%%%%%%%%%%%

The polynomials $f_{e, \phi}$ are listed in the third column of the following tables.  The first column is the Bala-Carter notation
for the nilpotent orbit $\0_e$ together with $\phi$, if non-trivial.  All $A(e)$ are symmetric groups and 
we denote $\phi$ by the corresponding partition for an irreducible representation of a symmetric group, where $[1^k]$ is the sign representation.
Recall that an orbit is principal in a Levi subalgebra when there are no parentheses in the Bala-Carter notation.  The letters in the notation denote the semisimple part of the Levi subalgebra.  In the second column are the representation exponents 
$m_i$.  
Exponents are listed according to the value of $\pi_i$, so that if $V$ occurs in the $\phi$-isotypic component, it is listed in the row of $(e, \phi)$.  We abbreviate $[a]_q$ by $[a]$ in the last column of the tables.

\begin{center}
\begin{tabular}{|c|c|c|} \hline
\multicolumn{3}{|c|}{$G_{2}$} \\ \hline
\multicolumn{1}{|c|}{$(e, \phi)$}
%&  \multicolumn{1}{|c|}{$e_i$}
&  \multicolumn{1}{|c|}{$m_i$}
& \multicolumn{1}{|c|}{$f_{e, \phi}$}
\\ \hline   \hline

$0$  & $1,5$ &  $\frac{[m-1] [m-5]}{ [2] [6] } $  \\\hline
$A_1$ & $1$ &  $q^{m - 5}\frac{ [m-1]}{ [2] } $  \\ \hline
$\Tilde{A}_{1}$  & $1$ & $q^{m - 3}\frac{ [m-1] }{ [2] } $    \\ \hline
$G_{2}(a_{1})$  &  $1$ &  $(q^{m-1}-1) \cdot q^{m - 3}$  \\ \hline
$G_{2}(a_{1}), [2,1]$  &  $-$ &  $0$  \\ \hline
$G_{2}(a_{1}), [1^3]$  &  $-$ &  $0$  \\ \hline
$G_{2}$   & $-$ &  $q^{2m - 2}$  \\ \hline
\end{tabular}

\vskip.2in

\begin{tabular}{|c|c|c|} \hline
\multicolumn{3}{|c|}{$F_{4}$} \\ \hline
\multicolumn{1}{|c|}{$(e, \phi)$}
&  \multicolumn{1}{|c|}{$m_i$}
& \multicolumn{1}{|c|}{$f_{e, \phi}$}
\\ \hline   \hline

$0$  & $1,5,7,11$ &$\frac{[m-1][m-5][m-7][m-11]}{ [2][6][8][12]} $  \\\hline
$A_1$ & $1,5,7$ & $q^{m - 11}\frac{ [m-1][m-5][m-7]}{ [2][4][6]} $  \\ \hline
$\Tilde{A}_{1}$ & $1,5,7$  & $q^{m - 5}\frac{ [m-1][m-5][m-7]}{ [2][4][6]} $    \\ \hline
$\Tilde{A}_{1}$, $[1^2]$ &  $-$ & $q^{m - 8}\frac{[m-1][m-5][m-7]}{[2][4][6]} $    \\ \hline
$A_{1} + \Tilde{A_{1}}$  & $1,5$  &  $q^{2m - 14}\frac{ [m-1][m-5]}{[2][2]} $   \\ \hline  
$A_2$  & $1,5$   &$q^{2m -8}\frac{ [m-1][m-5]}{ [2][6]} $   \\ \hline  
$A_2$, $[1^2]$ &  $-$   &$q^{2m -11}\frac{ [m-1][m-5]}{[2][6]} $   \\ \hline  
$\Tilde{A}_{2}$  & $1,5$  & $q^{2m -8 }\frac{ [m-1][m-5]}{[2][6]} $   \\ \hline  
$A_{2} + \Tilde{A}_{1}$  & $1$ & $q^{3m - 15} \frac{[m-1]}{[2]} $  \\ \hline
$B_2$   & $1,3$  & $q^{2m -6 }\frac{ [m-1][m-3]}{[2][4]} $   \\ \hline
$B_2$, $[1^2]$  &  $-$  & $q^{2m -8 }\frac{ [m-1][m-3]}{[2][4]} $   \\ \hline
$\Tilde{A}_{2} + A_{1}$ & $1$  &   $q^{3m - 13} \frac{[m-1]}{[2]} $  \\ \hline
$C_{3}(a_{1})$  & $1,3$ & $(q^{m-1} - 1) \cdot q^{2m - 8}\frac{[m-3]}{[2]} $  \\ \hline
$C_{3}(a_{1})$, $[1^2]$ &   $-$  & $0$  \\ \hline
%$F_{4}(a_{3})$ &     q^{m(4-3)+7 - 12} (q^{m-1} - 1) ( 1*0    - q^{m-3}*0   + q^{2(m-3)})   ( q^{2(m-3)}) \\ \hline
$F_{4}(a_{3})$ & $1$   &     $(q^{m-1} - 1) \cdot q^{3m- 11}$ \\ \hline
$F_{4}(a_{3}), [3,1]$ &  $-$  &     0 \\ \hline
$F_{4}(a_{3}), [2,2]$ &  $3$  &      $(q^{m-1} - 1) \cdot(-q^{2m- 8})$ \\ \hline
$F_{4}(a_{3}), [2,1^2]$ &  $-$ &     0 \\ \hline
%$F_{4}(a_{3}), [1^4]$ &  &     \text{not in Springer correspondence}  \\ \hline
$B_3$ & $1$ &  $q^{3m -7} \frac{[m-1]}{[2]}$ \\ \hline
$C_3$ & $1$ &  $q^{3m-7} \frac{[m-1]}{[2]}$\\ \hline
$F_{4}(a_{2})$ &  $1$ & $(q^{m-1} -1) \cdot q^{3m - 7}$ \\ \hline
$F_{4}(a_{2})$, $[1^2]$ &  $-$ & $0$ \\ \hline
$F_{4}(a_{1})$ & $1$ & $(q^{m-1} -1) \cdot q^{3m - 5}$ \\ \hline
$F_{4}(a_{1})$, $[1^2]$   &  $-$  & $0$ \\ \hline
$F_{4}$ &  $-$ & $q^{4m-4}$ \\ \hline

\end{tabular}
\end{center}

%\begin{table}[htdp]
%\caption{}
\begin{center}
\begin{tabular}{|c|c|c|} \hline
\multicolumn{3}{|c|}{$E_{6}$} \\ \hline
\multicolumn{1}{|c|}{$(e, \phi)$}
&  \multicolumn{1}{|c|}{$m_i$}
& \multicolumn{1}{|c|}{$f_{e, \phi}$}
\\ \hline \hline
$0$  & $1,4,5,7,8,11$ &  $\frac{ [m-1] [m-4][m-5][m-7][m-8][m-11]}
{ [2][5][6][8] [9][12] } $  \\\hline
$A_1$ & $1,4,5,7,8$ &  $q^{m-11}\frac{ [m-1] [m-4][m-5][m-7][m-8]}
{ [2][3][4][5] [6]} $  \\ \hline
$2A_{1}$  & $1,4,5,7$ &  $q^{2m - 16}\frac{ [m-1] [m-4][m-5][m-7]}
{ [1][2][4][6]} $   \\ \hline  
$3A_{1}$  & $1,4,5$ &  $q^{3m - 21}\frac{ [m-1] [m-4][m-5]}
{ [2][2][3]} $   \\ \hline  
%$A_2$  & $1,4,5$ &$q^{2m -9}\frac{ [m-1] [m-4][m-5]([m-2]+q[m-8])}
$A_2$  & $1,4,5$ &$q^{2m -9}\frac{ [m-1] [m-4][m-5]([m-7]+q[m-3])}
                                      { [2][3][4][6]}$ \\ \hline 
%\cdot (q^{m-2} \!+\! q^{m-7}\! -\!q\!-\!1) $   \\ \hline  
$A_2$, $[1^2]$   & $5$ &$q^{2m -11}\frac{ [m-1] [m-4][m-5]([m-3]+q^5[m-7])}{ [2][3][4][6] }$ \\ \hline 
%\cdot (q^{m-2} \!+ \!q^{m-3} \!-\! q^{5}\!-\! 1) $   \\ \hline  
$A_{2} + A_{1}$ & $1,4,5$& $q^{3m -16}\frac{ [m-1] [m-4][m-5]}{ [1][2][3]}  $  \\ \hline
$2A_{2}$& $1,5$ & $q^{4m -16}\frac{ [m-1][m-5]}{ [2][6]}  $  \\ \hline
$A_{2} + 2A_{1}$ & $1,4$ & $q^{4m-20}\frac{ [m-1] [m-4]}{[1][2]}  $  \\ \hline
$A_{3}$ & $1,3,4$& $q^{3m-11}\frac{ [m-1] [m-3][m-4]}{[1][2][4]}  $  \\ \hline
$2A_{2} + A_{1}$ & $1$ & $q^{5m-21}\frac{[m-1]}{[2]}  $  \\ \hline

$A_{3}+A_1$ & $1,3$& $q^{4m-15}\frac{ [m-1] [m-3]}{[1][2]}  $  \\ \hline
$D_{4}(a_1)$ & $1$ &  $(q^{m-1}-1) \cdot q^{3m-11 }\frac{[m-1][m-2]}{[2][3]}  $  \\ \hline
 $D_{4}(a_1), [2,1]$ &  $3$ &   $(q^{m-1}-1) \cdot q^{3m-10 }\frac{[m-2][m-4]}{[1][3]}  $   \\ \hline
 $D_{4}(a_1), [1^3]$ & $-$  &    $(q^{m-1}-1) \cdot q^{3m-8}\frac{[m-4][m-5]}{[2][3]}  $   \\ \hline 

$A_4$  & $1,3$&  $q^{4m-11} \frac{[m-1][m-3]}{[1][2]}$\\ \hline
$D_4$ & $1,2$ &  $q^{4m-10} \frac{[m-1][m-2]}{[2][3]}$\\ \hline
$A_4+A_1$  & $1$ &  $q^{5m-14} \frac{[m-1]}{[1]}$\\ \hline

$A_5$ & $1$ &  $q^{5m-11} \frac{[m-1]}{[2]}$\\ \hline
$D_5(a_1)$ & $1,2$ &  $(q^{m-1} - 1) \cdot q^{4m-10} \frac{[m-2]}{[1]}$\\ \hline
$E_{6}(a_{3})$  & $1$& $(q^{m-1} -1) \cdot q^{5m - 11}$ \\ \hline
$E_{6}(a_{3})$, $[1^2]$ &  $2$  & $(q^{m-1} -1) \cdot (-q^{4m - 9})$ \\ \hline
$D_{5}$  & $1$& $q^{5m - 8} \frac{[m-1]}{[1]}$ \\ \hline
$E_{6}(a_{1})$ & $1$ & $(q^{m-1} -1) \cdot q^{5m - 7}$ \\ \hline
$E_{6}$ & $-$ & $q^{6m-6}$ \\ \hline

\hline

\end{tabular}
\end{center}
%\label{}
%\end{table}

%\begin{table}[htdp]
%\caption{}
\begin{center}
\begin{tabular}{|c|c|c|} \hline
\multicolumn{3}{|c|}{$E_{7}$} \\ \hline
\multicolumn{1}{|c|}{$(e, \phi)$}
&  \multicolumn{1}{|c|}{$m_i$}
& \multicolumn{1}{|c|}{$f_{e, \phi}$}
\\ \hline \hline

$0$  & $1,5,7,9,11,13,17$ &  
%$\frac{ [m-1] [m-5][m-7][m-9][m-11][m-13][m-17]}
%{ (q^{2} - 1)[6][8][10] [12][14][18] } $  \\\hline
$\prod_{i=1}^7 \frac{[m-m_i]}{[m_i+1]}$  \\\hline
%\frac{q^{m-m_i} - 1}{q^{m_i+1} - 1}$  \\\hline
%{ (q^{2} - 1)[6][8][10] [12][14][18] } $  \\\hline
%{ (q^{2} - 1)[6][8][10] [12][14][18] } $ 

$A_1$ & $1,5,7,9,11,13$  &  $q^{m-17}\frac{ [m-1] [m-5][m-7][m-9][m-11][m-13]}
{ [2][4][6][6] [8][10]} $  \\ \hline
$2A_{1}$   & $1,5,7,9,11$ &  $q^{2m-26  }\frac{ [m-1] [m-5][m-7][m-9][m-11]}
{ [2][2][4][6] [8]} $  \\ \hline

$(3A_{1})''$  & $1,5,7,11$ & $q^{3m-27  }\frac{ [m-1] [m-5][m-7][m-11]}
{ [2][6][8][12]} $  \\ \hline
$(3A_{1})'$   & $1,5,7,9$ & $q^{3m-33  }\frac{ [m-1] [m-5][m-7][m-9]}
{ [2][2][4][6]} $  \\ \hline
%$A_2$   &$q^{2m }\frac{ [m-1] [m-5][m-7][m-9]}{ (q^2-1)[4][6][6][10]} 
%\cdot (q^{m-8}*[(q^3+1)(q^5+1) + (q^3-1)(q^5-1)]  -1*[(q^3+1)(q^5+1) - (q^3-1)(q^5-1)]) $   \\ \hline  
%
%$A_2$   &$q^{2m }\frac{ [m-1] [m-5][m-7][m-9]}{ (q^2-1)[4][6][6][10]} 
% \cdot (q^{m-8}*[q^8+1]  -1*[q^3+q^5]) $   \\ \hline  
%
$A_2$   & $1,5,7,9$ & $q^{2m - 14}\frac{ [m-1] [m-5][m-7][m-9]([m-3]+q^2[m-13])}{ [2][4][6][6][10] }$ \\ \hline
% \cdot (q^{m-3} \!+ \!q^{m-11} \! -\!q^2 \!- \!1) $   \\ \hline  
%$A_2$, $[1^2]$   &$q^{2m - 17}\frac{ [m-1] [m-5][m-7][m-9]}{ (q^2-1)[4][6][6][10]} 
%\cdot (q^{m-8}*[(q^3+1)(q^5+1) - (q^3-1)(q^5-1)]  -1*[(q^3+1)(q^5+1) + (q^3-1)(q^5-1)]) $   \\ \hline  
%$A_2$, $[1^2]$   &$q^{2m - 17}\frac{ [m-1] [m-5][m-7][m-9]}{ (q^2-1)[4][6][6][10]} 
%\cdot (q^{m-8}*[q^3+q^5]  -1*[q^8+1]) $   \\ \hline  
$A_2$, $[1^2]$   & $8$ &$q^{2m - 17}\frac{ [m-1] [m-5][m-7][m-9] ([m-3]+q^8[m-13]) }{ [2][4][6][6][10] }$ \\ \hline 
% \cdot (q^{m-3}\!+ \!q^{m-5} \! -\! q^8 \!-\! 1) $   \\ \hline  
$4A_{1}$   & $1,5,7$ & $q^{4m-  38}\frac{ [m-1] [m-5][m-7]}{ [2][4][6]} $  \\ \hline
%$A_{2} + A_{1}$ & 
%$q^{3m + 21 -  47}\frac{ [m-1] [m-5][m-7]}{ (q^2-1)(q^2-1)[4][6]} 
%\cdot (q^{m-8}*[(q^3+1)(q+1) + (q^3-1)(q-1)]  -1*[(q^3+1)(q+1) - (q^3-1)(q-1)]) $   \\ \hline  
$A_{2} + A_{1}$ & $1,5,7$ & 
$q^{3m - 25 }\frac{ [m-1] [m-5][m-7](q^2[m-7]+ [m-9])}{ [2][2][4][6] }$ \\ \hline
% \cdot (q^{m-5} + q^{m-9} - q^2 - 1) $   \\ \hline  
$A_{2} + A_{1}$, $[1^2]$ & $8$  & 
$q^{3m - 26 }\frac{ [m-1] [m-5][m-7]([m-7]+ q^4[m-9])}{ [2][2][4][6] }$ \\ \hline 
% \cdot (q^{m-5} + q^{m-7} - q^4 - 1) $   \\ \hline  
$A_{2} + 2A_{1}$ & $1,5,7$ & $q^{4m-32} \frac{ [m-1] [m-5][m-7]}{ [2]^3}  $  \\ \hline%%#49
$2A_{2}$ & $1,5,7$ & $q^{4m-26} \frac{ [m-1] [m-5][m-7]}{ [2]^2[6]}  $  \\ \hline%%#48
$A_{2} + 3A_{1}$ & $1,5$  & $q^{5m-35}\frac{ [m-1] [m-5]}{ [2][6]}  $  \\ \hline%%#47
$A_{3}$ & $1,5,5,7$ &  $q^{3m-17}\frac{ [m-1] [m-5][m-5][m-7]}{ [2]^2[4][6]}  $  \\ \hline%%#46
$(A_{3}+A_1)''$ & $1,5,7$ & $q^{4m-22}\frac{ [m-1] [m-5][m-7]}{ [2][4][6]}  $  \\ \hline%%#45
$2A_{2} + A_{1}$ & $1,5$ & $q^{5m-33}\frac{ [m-1] [m-5]}{ [2]^2}  $  \\ \hline%%#44
$(A_{3}+A_1)'$ & $1,5,5$ & $q^{4m-24}\frac{ [m-1] [m-5][m-5]}{ [2]^3}  $  \\ \hline%%#43
$D_{4}(a_1)$ &  $1,5$ &
%$q^{3m+16- 33}(q^{m-1}-1)[m-5]\frac{ (q^{2(m-5)}
%$ replace q by q^2 in the counting of points
%(1/6(q+1)(q^2+q+1) + 1/3((q^2-1)(q-1) + 1/2(q^3-1) = q^3)
 %-q^{m-3}(1/6*2*(q+1)(q^2+q+1) + -1/3(q^2-1)(q-1) +0(q^3-1) = q^2+q)     +1*
  %(1/6*1*(q+1)(q^2+q+1) + 1/3(q^2-1)(q-1) -1/2(q^3-1) = 1)       )}{(q^3-1)(q^2-1)}  $  \\ \hline

%$q^{3m-17 }(q^{m-1}-1)[m-5]\frac{ (q^{2(m-5)}(q^6)
 %-q^{m-5}(q^4+q^2)     +1*(1)       )}{(q^3-1)(q^2-1)}  $  \\ \hline
$(q^{m-1}-1) \cdot q^{3m-17 }\frac{ [m-1][m-3][m-5]}{[2][4][6]}  $  \\ \hline
 $D_{4}(a_1)$, $[2,1]$ & $5^{}$ &  $(q^{m-1}-1) \cdot q^{3m-15 }\frac{[m-3][m-5][m-7]}{[2]^2[6]}  $  \\ \hline
 $D_{4}(a_1)$, $[1^3]$  & $-$ &  $(q^{m-1}-1) \cdot q^{3m-11 }\frac{[m-5][m-7][m-9]}{[2][4][6]}  $  \\ \hline
$A_{3}+2A_1$ & $1,5$ & $q^{5m-29}\frac{ [m-1] [m-5]}{ [2]^2}  $  \\ \hline  %%#39
% $D_{4}(a_1)+A_1$   &  $q^{3m+11-33} \frac{[m-1][m-5]}{[2][4]}(q^{m-5}(q^2) -(1)) \\ \hline %% #37
$D_{4}(a_1)+A_1$  & $1,5$ &  $(q^{m-1}-1) \cdot q^{4m-22} \frac{[m-3][m-5]}{[2][4]}$ \\ \hline %% #37
$D_{4}(a_1)+A_1$, $[1^2]$  & $5$  &  $(q^{m-1}-1) \cdot q^{4m-20} \frac{[m-5][m-7]}{[2][4]}$ \\ \hline %% #37
$D_4$  & $1,3,5$  &  $q^{4m-16} \frac{[m-1][m-3][m-5]}{[2][4][6]}$\\ \hline  %%#36
%$A_3+A_2$  &  $q^{5m + 6 - 32} \frac{[m-1][m-5]}{[2]^2)}
%(q)$\\ \hline   %% #34,35
$A_3+A_2$  & $1,5$  &  $q^{5m -25} \frac{[m-1][m-5]}{[2]^2}$\\ \hline   %% #34,35
$A_3+A_2$, $[1^2]$  & $-$ &  $q^{5m -26} \frac{[m-1][m-5]}{[2]^2}$\\ \hline   %% #34,35
%$A_4$  &  $q^{4m + 10 - 27} \frac{[m-1][m-5]}{[2]^2 [6]}
%(q^{m-4}( (q+1)(q^3+1) + (q-1)(q^3-1)) - 1((q+1)(q^3+1) - (q-1)(q^3-1))$\\ \hline   %% #32,33
%$A_4$  &  $q^{4m + 10 - 27} \frac{[m-1][m-5]}{[2]^2 [6]}
%(q^{m-4}(q^4 +1)- q-q^3)$\\ \hline   %% #32,33
$A_4$  &  $1,5$ &  $q^{4m -16} \frac{[m-1][m-5](q^2 [m-3]+[m-5])}{[2]^2 [6]}$ \\ \hline
% (q^{m-1}+q^{m-5}- q^2-1)$\\ \hline   %% #32,33
%$A_4$, $[1^2]$  &  $q^{4m -17} \frac{[m-1][m-5]}{[2]^2 [6]}
%(q^{m-4}(q^3 +q)- 1-q^4)$\\ \hline   %% #32,33
$A_4$, $[1^2]$  & $4$ &  $q^{4m -17} \frac{[m-1][m-5]([m-3]+q^4[m-5])}{[2]^2 [6]}$ \\ \hline
%(q^{m-1}+q^{m-3}-q^4-1)$\\ \hline   %% #32,33
$A_3+A_2+A_1$ &  $1$ &  $q^{6m-30}  \frac{[m-1]}{[2]} $\\ \hline  %% #31
$A_5''$  & $1,5$ &  $q^{5m-17}  \frac{[m-1][m-5]}{[2][6]} $\\ \hline  %% #30
$D_4+A_1$  & $1,3$ &  $q^{5m-21}  \frac{[m-1][m-3]}{[2][4]} $\\ \hline

\hline
\end{tabular}
\end{center}
%\label{}
%\end{table}

%\begin{table}
\begin{center}
\begin{tabular}{|c|c|c|} \hline
\multicolumn{3}{|c|}{$E_{7}$, part 2} \\ \hline
\multicolumn{1}{|c|}{$(e, \phi)$}
&  \multicolumn{1}{|c|}{$m_i$}
& \multicolumn{1}{|c|}{$f_{e, \phi}$}
\\ \hline \hline
$A_4+A_1$  & $1$ &  $q^{5m-21} \frac{[m-1]}{[2]} \cdot \frac{[m-3] + [m-5]}{[2]}$\\ \hline
$A_4+A_1$, $[1^2]$  & $4$ &  $q^{5m-22} \frac{[m-1]}{[2]} \cdot \frac{[m-3]+q^2[m-5]}{[2]}$\\ \hline
$D_5(a_1)$  & $1,3$ &  $(q^{m-1}-1) \cdot q^{4m-16} \frac{[m-3]^2}{[2]^2}$\\ \hline
$D_5(a_1)$, $[1^2]$  & $4$ &  $(q^{m-1}-1) \cdot q^{4m-15} \frac{[m-3][m-5]}{[2]^2} $\\ \hline
$A_4+A_2$  & $1$ &   $q^{6m-24} \frac{[m-1]}{[2]}$\\ \hline
$A_5'$  & $1,3$  &   $q^{5m-17} \frac{[m-1][m-3]}{[2]^2}$\\ \hline
$D_5(a_1)+A_1$  & $1,3$ &  $(q^{m-1}-1) \cdot q^{5m-19} \frac{[m-3]}{[2]}$\\ \hline
$A_5+A_1$  & $1$ &  $q^{6m-22} \frac{[m-1]}{[2]}$\\ \hline
$D_6(a_2)$  & $1,3$ &  $(q^{m-1}-1) \cdot q^{5m-17}  \frac{[m-3]}{[2]}$\\ \hline
%$E_{6}(a_{3})$  & $q^{4m +7 - 21}  \frac{[m-1][m-3]}{[2]}( q^{m-3} )$ \\ \hline
$E_{6}(a_{3})$  & $1,3$ & $(q^{m-1}-1) \cdot q^{5m - 17}  \frac{[m-3]}{[2]}$ \\ \hline
%$E_{6}(a_{3})$, $[1^2]$  & $q^{4m +7 - 20}  \frac{[m-1][m-3]}{[2]}( q^{m-3} )$ \\ \hline
$E_{6}(a_{3})$, $[1^2]$  & $3$ & $(q^{m-1}-1) (-q^{4m -14})  \frac{[m-3]}{[2]}$ \\ \hline
%$E_{7}(a_{5})$  & 
%$q^{4m+7-21 }(q^{m-1}-1)\frac{ (q^{2(m-3)}(1) -q^{m-3}(0) +1*0)}{1}  $  \\ \hline
$E_{7}(a_{5})$  & $1$ & 
$(q^{m-1}-1) \cdot q^{6m-20 }$  \\ \hline

%$E_{7}(a_{5})$, $[2,1]$  & 
%$q^{4m+7-21 }(q^{m-1}-1)\frac{ (q^{2(m-3)}(0) -q^{m-3}(1) +1*0)}{1}  $  \\ \hline
$E_{7}(a_{5})$, $[2,1]$  & $3$ & 
$(q^{m-1}-1) \cdot (-q^{5m-17 })$  \\ \hline
%$E_{7}(a_{5})$, sgn & 
%$q^{4m+7-21 }(q^{m-1}-1)\frac{ (q^{2(m-3)}(0) -q^{m-3}(0) +1*1)}{1}  $  \\ \hline
 $E_{7}(a_{5})$, $[1^3]$ & $-$ & 
 $(q^{m-1}-1) \cdot q^{4m-14 }$  \\ \hline
  
$D_{5}$  & $1,3$ & $q^{5m - 13} \frac{[m-1][m-3]}{[2][2]}$ \\ \hline
$A_{6}$  & $1$ & $q^{6m - 16} \frac{[m-1]}{[2]}$ \\ \hline
$D_{5}+A_1$  & $1$ & $q^{6m-16} \frac{[m-1]}{[2]}$ \\ \hline
$D_6(a_1)$  & $1,3$ &  $(q^{m-1}-1) \cdot q^{5m -13} \frac{[m-3]}{[2]}$\\ \hline
$E_{7}(a_{4})$  & $1$ & $(q^{m-1} -1) \cdot q^{6m -16}$ \\ \hline
$E_{7}(a_{4})$, $[1^2]$  & $-$ & 0 \\ \hline
%$E_{6}(a_{1})$  & $q^{5m +3 -14} \frac{[m-1]}{[2]}[  q^{m-2} ((q+1) + (q-1) )    -((q+1) - (q-1))] $ \\ \hline
$E_{6}(a_{1})$  & $1$ & $(q^{m-1} -1) \cdot q^{5m  - 11} \frac{[m-1]}{[2]} $ \\ \hline
$E_{6}(a_{1})$, $[1^2]$  & $2$ & $(q^{m-1} -1) \cdot q^{5m - 10} \frac{[m-3]}{[2]}$ \\ \hline
$D_{6}$  & $1$ & $q^{6m  - 12} \frac{[m-1]}{[2]}$\\ \hline

%$E_{7}(a_{3})$  & $q^{5m - 10} (q^{m-1} -1) [  q^{m-2} (1 )    -(0)] $ \\ \hline
$E_{7}(a_{3})$  & $1$ & $(q^{m-1} -1) \cdot q^{6m - 12}$ \\ \hline
%$E_{7}(a_{3})$, $[1^2]$  & $q^{5m - 10} (q^{m-1} -1) [  q^{m-2} (0 )    -(1)] $ \\ \hline
$E_{7}(a_{3})$, $[1^2]$  & $2$ & $(q^{m-1} -1) \cdot (-q^{5m - 10})$ \\ \hline

$E_{6}$  & $1$ & $q^{6m  - 10} \frac{[m-1]}{[2]}$\\ \hline
$E_{7}(a_{2})$  & $1$ & $(q^{m-1} -1) \cdot q^{6m - 10}$ \\ \hline
$E_{7}(a_{1})$ & $1$  & $(q^{m-1} -1) \cdot q^{6m - 8}$ \\ \hline
$E_{7}$ & $-$ & $q^{7m-7}$ \\ \hline
\hline

\end{tabular}
\end{center}
%\label{}
%\end{table}

%\begin{table}[htdp]
%\caption{}
\begin{center}
\begin{tabular}{|c|c|c|} \hline
\multicolumn{3}{|c|}{$E_{8}$} \\ \hline
\multicolumn{1}{|c|}{$(e, \phi)$}
&  \multicolumn{1}{|c|}{$m_i$}
& \multicolumn{1}{|c|}{$f_{e, \phi}$}
\\ \hline \hline

$0$  &  $1,7,11,13,17,19, 23,29$ &  
%$\frac{ [m-1] [m-7][m-11][m-13][m-17][m-19][m-23](q^{m-29} - 1)}
%$\frac{ \prod  q^{m-m_i} - 1}{ \prod  q^{m_i + 1} - 1}$ \\ \hline
$\prod_{i=1}^8 \frac{ [m-m_i]}{[m_i+ 1]}$ \\ \hline
%[2][8][12][14] [18](q^{20} - 1)(q^{24} - 1)(q^{30} - 1) } $  \\\hline
$A_1$ & $1,7,11,13,17, 19,23$  &  $q^{m-29}
%\frac{ \prod  q^{m-m_i} - 1}
\frac{ [m-1] [m-7][m-11][m-13][m-17][m-19][m-23]}
{ [2][6][8][10] [12][14][18]}$
  \\ \hline
$2A_{1}$   & $1,7,11,13, 17,19$  &  $q^{2m-46  }
\frac{ [m-1] [m-7][m-11][m-13][m-17][m-19]}
{ [2][4][6][8] [10][12] } $  \\ \hline
$3A_{1}$   & $1,7,11,13,17$  & $q^{3m-57  }
\frac{ [m-1] [m-7][m-11][m-13][m-17]}
{ [2]^2[6][8] [12]} $  \\ \hline
%$A_2$   &$q^{2m }\frac{ [m-1] [m-5][m-7][m-9]}{ [2][4][6][6][10]} 
%\cdot (q^{m-14}*[(q^5+1)(q^9+1) + (q^5-1)(q^9-1)]  -1*[(q^5+1)(q^9+1) - (q^5-1)(q^9-1)]) $   \\ \hline  
%
%$A_2$   &$q^{2m }\frac{ [m-1] [m-5][m-7][m-9]}{ [2][4][6][6][10]} 
% \cdot (q^{m-8}*[q^8+1]  -1*[q^3+q^5]) $   \\ \hline  
%
$A_2$   & $1,7,11,13,17$  & $q^{2m - 24 }\frac{ [m-1][m-7][m-11][m-13][m-17]
([m-5]+q^4 [m-23])}
%(q^{m-5} + q^{m-19} - q^4 - 1)}
{ [2][6][8][10][12][18]}$   \\ \hline  
$A_2$, $[1^2]$    & $14$  & $q^{2m - 29 }\frac{ [m-1][m-7][m-11][m-13][m-17]
([m-5] + q^{14} [m-23])}
%(q^{m-5} + q^{m-9} - q^{14} - 1)}
{ [2][6][8][10][12][18]}$   \\ \hline
$4A_{1}$   & $1,7,11,13$ & $q^{4m-  68} \frac{ [m-1] [m-7][m-11][m-13]}
{ [2][4][6] [8]} $  \\ \hline
%$A_{2} + A_{1}$ & 
%$q^{3m + 21 -  47}\frac{ [m-1] [m-5][m-7]}{ [2][2][4][6]} 
%\cdot (q^{m-8}*[(q^3+1)(q+1) + (q^3-1)(q-1)]  -1*[(q^3+1)(q+1) - (q^3-1)(q-1)]) $   \\ \hline  
$A_{2} + A_{1}$  & $1,7,11,13$ & 
$q^{3m  - 43}
\frac{ [m-1][m-7][m-11][m-13] (q^2 [m-11] +[m-17])}
%(q^{m-9} + q^{m-17} - q^2 - 1)}
{ [2][4][6]^2[10]}$   \\ \hline  
$A_{2} + A_{1}$, $[1^2]$ & $14$ & 
$q^{3m  - 46}
\frac{ [m-1][m-7][m-11][m-13]([m-11] + q^8 [m-17])} 
% (q^{m-9} + q^{m-11} - q^8 - 1)}
{ [2][4][6]^2[10]}$
\\ \hline  
$A_{2} + 2A_{1}$ & $1,7,11,13$ & $q^{4m-56} \frac{ [m-1] [m-7][m-11][m-13]}
{ [2]^2[4][6] } $  \\ \hline
$A_{3}$ & $1,7,9,11,13$  &  $q^{3m-29}\frac{ [m-1] [m-7][m-9][m-11][m-13]}
{ [2][4][6][8][10]}  $  \\ \hline
$A_{2} + 3A_{1}$ & $1,7,11$ & $q^{5m-65}\frac{ [m-1] [m-7][m-11]}{ [2]^2[6]}  $  \\ \hline

$2A_{2}$ & $1,7,11$ & $q^{4m-44} \frac{ [m-1] [m-7][m-11]([m-5]+ q^4 [m-17])}
%(q^{m-5}+q^{m-13} - q^4-1)}
{ [2][4][6][12]}  $  \\ \hline
$2A_{2}$, $[1^2]$ & $11$ & $q^{4m- 46} \frac{ [m-1] [m-7][m-11]([m-5]+ q^8 [m-17])}
%(q^{m-5}+q^{m-9} - q^8-1)}
{ [2][4][6][12]}  $  \\ \hline

$2A_{2} + A_{1}$ & $1,7,11$ & $q^{5m-57}\frac{ [m-1] [m-7][m-11]}{ [2]^2[6]}  $  \\ \hline
$A_{3}+A_1$ & $1,7,9,11$  & $q^{4m-42}\frac{ [m-1] [m-7][m-9][m-11]}
{ [2]^2[4][6]}  $  \\ \hline

$D_{4}(a_1)$ & $1,7,11$  &  %%#40
%$q^{3m+16- 33}(q^{m-1}-1)(q^{m-5}-1)]\frac{ (q^{2(m-5)}
%$ replace q by q^2 in the counting of points
%(1/6(q+1)(q^2+q+1) + 1/3(q^2-1)(q-1) + 1/2(q^3-1) = q^3)
 %-q^{m-3}(1/6*2*(q+1)(q^2+q+1) + -1/3(q^2-1)(q-1) +0(q^3-1) = q^2+q)     +1*
  %(1/6*1*(q+1)(q^2+q+1) + 1/3(q^2-1)(q-1) -1/2(q^3-1) = 1)       )}{(q^3-1)(q^2-1)}  $  \\ \hline

%$q^{3m-17 }(q^{m-1}-1)(q^{m-5}-1)\frac{ (q^{2(m-5)}(q^6)
 %-q^{m-5}(q^4+q^2)     +1*(1)       )}{(q^3-1)(q^2-1)}  $  \\ \hline
$(q^{m-1}-1) \cdot q^{3m-29 }\frac{[m-1][m-5][m-7][m-11]}{[2][6][8][12]}  $  \\ \hline
 $D_{4}(a_1)$, $[2,1]$ & $9^{}$ &  $(q^{m-1}-1) \cdot q^{3m-25 }\frac{[m-5][m-7][m-11][m-13]}{[2][4][6][12]}  $  \\ \hline
 $D_{4}(a_1)$, $[1^3]$ & $-$ &  $(q^{m-1}-1) \cdot q^{3m-17}\frac{[m-7][m-11][m-13][m-17]}{[2][6][8][12]}  $  \\ \hline

 $D_4$  & $1,5,7,11$ &  $q^{4m-28} \frac{[m-1][m-5][m-7][m-11]}{[2][6][8][12]}$\\ \hline  
 $2A_{2} + 2A_{1}$ & $1,7$ & $q^{6m-66}\frac{ [m-1][m-7]}{[2] [4]} $  \\ \hline  %80

$A_{3}+2A_1$ & $1,7,9$ & $q^{5m-51}\frac{ [m-1][m-7][m-9]}{[2]^2 [4]} $  \\ \hline  %76

$D_{4}(a_1)+A_1$   & $1,7$ &  $(q^{m-1}-1) \cdot q^{4m-40} \frac{[m-5][m-7]^2}{[2][4][6]}$ \\ \hline 
$D_{4}(a_1)+A_1, [2,1]$  & $9^{}$ &   $(q^{m-1}-1) \cdot q^{4m-38} \frac{[m-5][m-7][m-11]}{[2]^2[6]}$\\ \hline
$D_{4}(a_1)+A_1, [1^3]$   & $-$ &   $(q^{m-1}-1) \cdot q^{4m-34} \frac{[m-7][m-11][m-13]}{[2][4][6]}$ \\ \hline

$A_3+A_2$  & $1,7,9$ &  $q^{5m -45} \frac{[m-1][m-7][m-9]}{[2]^2 [4]}$\\ \hline 
$A_3+A_2$, $[1^2]$  & $-$ &  $q^{5m -46} \frac{[m-1][m-7][m-9]}{[2]^2 [4]}$ \\ \hline  
$A_4$  & $1,7,9$ &  $q^{4m -26} \frac{[m-1][m-7][m-9](q^2[m-5] + [m-11])}
%(q^{m-3} + q^{m-11}-q^2-1)}
{[2][4][6][10]}$ \\ \hline 
%$A_4$, $[1^2]$  &  $q^{4m -17} \frac{[m-1][m-5]}{[2]^2 [6]}
%(q^{m-4}(q^3 +q)- 1-q^4)$\\ \hline   %% #32,33
$A_4$, $[1^2]$  & $8$ &  $q^{4m -29} \frac{[m-1][m-7][m-9]([m-5] + q^8[m-11])}
%(q^{m-3} + q^{m-5}-q^8-1)}
{[2][4][6][10]}$ \\ \hline   
$A_3+A_2+A_1$  & $1,7$ &  $q^{6m-54}  \frac{[m-1][m-7]}{[2]^2} $\\ \hline  
$D_4+A_1$  & $1,5,7$ &  $q^{5m-39}  \frac{[m-1][m-5][m-7]}{[2][4][6]} $ \\ \hline
$D_{4}(a_1)+A_2$  & $1,7$ &  $(q^{m-1} -1) \cdot q^{5m- 43} \frac{[m-5][m-7]}{[2][6]} $  \\ \hline
$D_{4}(a_1)+A_2$, $[1^2]$  & $8$ &   $(q^{m-1} -1) \cdot q^{5m-40}\frac{[m-7][m-11]}{[2][6]} $  \\ \hline
\hline
\end{tabular}
\end{center}
%\label{}
%\end{table}

%\begin{table}
\begin{center}
\begin{tabular}{|c|c|c|} \hline
\multicolumn{3}{|c|}{$E_{8}$, part 2} \\ \hline
\multicolumn{1}{|c|}{$(e, \phi)$}
&  \multicolumn{1}{|c|}{$m_i$}
& \multicolumn{1}{|c|}{$f_{e, \phi}$}
\\ \hline \hline
$A_4+A_1$  & $1,7$ &  $q^{5m-37} \frac{[m-1][m-7] ([m-5] + q^2 [m-11])}
%(q^{m-5} + q^{m-9}-q^2-1)}
{[2]^2[6]}$ \\ \hline
$A_4+A_1$, $[1^2]$  & $8$ &  $q^{5m-38} \frac{[m-1][m-7]([m-5] + q^4 [m-11])}
%(q^{m-5}+q^{m-7}-q^4-1)}
{[2]^2 [6]}$ \\ \hline
$2A_3$  & $1,7$ &  $q^{6m-46} \frac{[m-1][m-7]}{[2][4]}$\\ \hline

$D_5(a_1)$  & $1,5,7$ &  $(q^{m-1} -1) \cdot q^{4m-28} \frac{[m-5]^2[m-7]}{[2][4][6]} $\\ \hline  %%C=58
$D_5(a_1)$, $[1^2]$  & $8$ &  $(q^{m-1} -1) \cdot q^{4m-25} \frac{[m-5][m-7][m-11]}{[2][4][6]}$\\ \hline

$A_4+2A_1$  & $1,7$ &   $q^{6m-44} \frac{[m-1][m-7]}{[2]^2}$ \\ \hline
$A_4+2A_1$, $[1^2]$  & $-$ &  $q^{6m-45} \frac{[m-1][m-7]}{[2]^2}$ \\ \hline

$A_4+A_2$  & $1,7$ &   $q^{6m-42} \frac{[m-1][m-7]}{[2]^2}$\\ \hline  %%54

$A_5$  & $1,5,7$ &  $q^{5m-29} \frac{[m-1][m-5][m-7]}{[2]^2[6]}$\\ \hline
$D_5(a_1)+A_1$  & $1,5,7$ & $(q^{m-1} -1) \cdot q^{5m-35} \frac{[m-5][m-7]}{[2]^2}$ \\ \hline
$A_4 + A_2+A_1$  & $1$ &  $q^{7m-49} \frac{[m-1]}{[2]}$\\ \hline

$D_4 + A_2$  & $1,5$ &  $q^{6m-36} \frac{[m-1][m-5]}{[2][6]}$ \\ \hline
$D_4 + A_2$, $[1^2]$ & $-$ &  $q^{6m-39} \frac{[m-1][m-5]}{[2][6]}$ \\ \hline
$E_{6}(a_{3})$  & $1,5,7$ & $(q^{m-1} -1) \cdot q^{5m - 29}  \frac{[m-5][m-7]}{[2][6]}$ \\ \hline
$E_{6}(a_{3})$, $[1^2]$  & $5$ & $(q^{m-1} -1)\cdot (-q^{4m -24})  \frac{[m-5][m-7]}{[2][6]}$ \\ \hline

$D_{5}$  & $1,5,7$  & $q^{5m - 23} \frac{[m-1][m-5][m-7]}{[2][4][6]}$ \\ \hline
$A_4+A_3$  & $1$ &   $q^{7m-45} \frac{[m-1]}{[2]}$\\ \hline

$A_5+A_1$  & $1,5$ &   $q^{6m-36} \frac{[m-1][m-5]}{[2]^2}$\\ \hline
$D_5(a_1)+A_2$  & $1,5$ & $(q^{m-1}-1) \cdot q^{6m-38} \frac{[m-5]}{[2]}$\\ \hline

$D_6(a_2)$  & $1,5$  &  $(q^{m-1}-1) \cdot q^{5m-29}  \frac{[m-3][m-5]}{[2][4]}$\\ \hline
$D_6(a_2)$, $[1^2]$  & $5$  &  $(q^{m-1}-1) \cdot q^{5m-27}  \frac{[m-5][m-7]}{[2][4]}$\\ \hline
$E_{6}(a_{3})+A_1$  & $1,5$ & $(q^{m-1}-1) \cdot q^{6m - 36}  \frac{[m-5]}{[2]}$ \\ \hline
$E_{6}(a_{3})+A_1$, $[1^2]$ & $5$  & $(q^{m-1}-1) \cdot (-q^{5m -31})  \frac{[m-5]}{[2]}$ \\ \hline
$E_{7}(a_{5})$  & $1,5$ &  $(q^{m-1}-1) \cdot q^{6m-34} \frac{[m-5]}{[2]}$  \\ \hline
$E_{7}(a_{5})$, $[2,1]$  & $5$ & 
$(q^{m-1} - 1)\cdot (-q^{5m-29 })\frac{[m-5]}{[2]}$  \\ \hline
 $E_{7}(a_{5})$, $[1^3]$ & $-$  &
 $(q^{m-1} - 1) \cdot q^{4m-24 }\frac{[m-5]}{[2]}$  \\ \hline
$D_{5}+A_1$  & $1,5$  & $q^{6m-30} \frac{[m-1][m-5]}{[2]^2}$ \\ \hline

$E_{8}(a_{7})$  &  $1$  & $(q^{m-1} -1) \cdot q^{7m - 39}$ \\ \hline
$E_{8}(a_{7}), [4,1]$  &  $5$ &  $(q^{m-1} -1) \cdot (-q^{6m - 34})$ \\ \hline
$E_{8}(a_{7}), [3,2]$  &  $-$ &  $0$ \\ \hline
$E_{8}(a_{7}), [3,1^2]$  &  $-$ &  $(q^{m-1} -1) \cdot q^{5m - 29}$ \\ \hline
$E_{8}(a_{7}), [2^2,1]$  &  $-$ &  $0$ \\ \hline
$E_{8}(a_{7}), [2,1^3]$  &  $-$ &  $(q^{m-1} -1) \cdot (-q^{4m - 24})$ \\ \hline
%%%$E_{8}(a_{7}), [1^5]$  &  $-$ & \text{not in Springer corresp.}  \\ \hline
$A_{6}$  &  $1,5$  & $q^{6m - 28} \frac{[m-1][m-5]}{[2]^2}$ \\ \hline  %% 7m in earlier version
\hline
\end{tabular}
\end{center}
%\label{}
%\end{table}

%\begin{table}
\begin{center}
\begin{tabular}{|c|c|c|} \hline
\multicolumn{3}{|c|}{$E_{8}$, part 3} \\ \hline
\multicolumn{1}{|c|}{$(e, \phi)$}
& \multicolumn{1}{|c|}{$m_i$}
& \multicolumn{1}{|c|}{$f_{e, \phi}$}
\\ \hline \hline
$D_6(a_1)$  &  $1,5$  &  $(q^{m-1} - 1) \cdot q^{5m -23} \frac{[m-3][m-5]}{[2][4]}$\\ \hline
$D_6(a_1)$, $[1^2]$  &  $5$ &  $(q^{m-1} - 1) \cdot q^{5m -21} \frac{[m-5][m-7]}{[2][4]}$\\ \hline
$A_{6}+A_1$  &  $1$ & $q^{7m - 33} \frac{[m-1]}{[2]}$ \\ \hline
$E_{7}(a_{4})$  &  $1,5$ & $(q^{m-1} -1) \cdot q^{6m  - 28} \frac{[m-5]}{[2]} $ \\ \hline
$E_{7}(a_{4})$, $[1^2]$  &  $-$ & 0 \\ \hline
$E_{6}(a_{1})$  &  $1,5$  & $(q^{m-1} -1) \cdot q^{5m  - 19} \frac{[m-1][m-5]}{[2][6]} $ \\ \hline
$E_{6}(a_{1})$, $[1^2]$  &  $4$ & $(q^{m-1} -1) \cdot q^{5m -16} \frac{[m-5][m-7]}{[2][6]}$ \\ \hline
$D_{5}+A_2$  &  $1$ & $q^{7m-31} \frac{[m-1]}{[2]}$ \\ \hline  %T_1
$D_{5}+A_2$, $[1^2]$  &  $-$ & $q^{7m-32} \frac{[m-1]}{[2]}$ \\ \hline  %T_1
$D_{6}$  &  $1,3$ & $q^{6m  -22} \frac{[m-1][m-3]}{[2][4]}$\\ \hline  %B_2
$E_{6}$  &  $1,5$ & $q^{6m  - 18} \frac{[m-1][m-5]}{[2][6]}$\\ \hline    %G_2
$D_{7}(a_{2})$  &  $1$ & $(q^{m-1} -1) \cdot q^{6m-26} \frac{[m-3]}{[2]}$ \\ \hline  %T_1
$D_{7}(a_{2})$, $[1^2]$  &  $4$ & $(q^{m-1} -1) \cdot q^{6m-25} \frac{[m-5]}{[2]}$ \\ \hline
$A_{7}$  &  $1$ & $q^{7m - 27} \frac{[m-1]}{[2]}$\\ \hline %A_1
$E_{6}(a_{1})+A_1$  &  $1$ & $(q^{m-1} -1) \cdot q^{6m-24} \frac{[m-3]}{[2]} $ \\ \hline %T_1
$E_{6}(a_{1})+A_1$, $[1^2]$ &  $4$  & $(q^{m-1} -1) \cdot q^{6m-23} \frac{[m-5]}{[2]} $ \\ \hline %T_1
$E_{7}(a_{3})$  &  $1,3$   & $(q^{m-1} - 1) \cdot q^{6m-22} \frac{[m-3]}{[2]}$\\ \hline %A_1
$E_{7}(a_{3})$, $[1^2]$ &  $4$  & $(q^{m-1} - 1) \cdot (-q^{5m-18}) \frac{[m-3]}{[2]}$\\ \hline %A_1
$E_{8}(b_{6})$  &  $1$ & $(q^{m-1} - 1) \cdot q^{7m-27}$\\ \hline  %#24, 25, 26   degs 1
$E_{8}(b_{6})$, $[2,1]$ & $-$  & $0$ \\ \hline
$E_{8}(b_{6})$, $[1^3]$ & $-$  & $0$  \\ \hline
$D_7(a_1)$ &  $1,3$ & $(q^{m-1} - 1) \cdot q^{6m - 20} \frac{[m-3]}{[2]}$   \\ \hline %#22, 23   degs 1, 3   T_1
$D_7(a_1)$, $[1^2]$ &  $-$ & $(q^{m-1} - 1) \cdot  q^{6m - 21} \frac{[m-3]}{[2]}$   \\ \hline  %T_1
$E_{6}+A_1$  &  $1$ & $q^{7m-23} \frac{[m-1]}{[2]}$ \\ \hline  %#21   degs 1  A_1 
$E_{7}(a_{2})$  &  $1,3$ & $(q^{m-1} - 1) \cdot q^{6m - 18} \frac{[m-3]}{[2]}$ \\ \hline  %#20   degs 1, 3  A_1
$E_{8}(a_{6})$  &  $1$ & $(q^{m-1} - 1) \cdot q^{7m-23}$ \\ \hline  %#17,18,19   degs 1, 3( 2-d)  
$E_{8}(a_{6})$, $[2,1]$  &  $3$ & $(q^{m-1} - 1) \cdot (-q^{6m-20})$ \\ \hline
$E_{8}(a_{6})$, $[1^3]$ & $-$  &$(q^{m-1} - 1) \cdot q^{5m-17}$ \\ \hline
$D_7$ &  $1$ & $q^{7m -19 } \frac{[m-1]}{[2]}$   \\ \hline  %#16   degs 1  A_1
%$E_{8}(b_{5})$  & $q^{6m+4-22} (q^{m-1} - 1)q^{m-3}$ \\ \hline %#13,14,15   degs 1, 3( 2-d)  
$E_{8}(b_{5})$  &  $1$ & $(q^{m-1} - 1) \cdot q^{7m-21}$ \\ \hline %#13,14,15   degs 1, 3( 2-d)  
$E_{8}(b_{5})$, $[2,1]$ &  $3$  & $(q^{m-1} - 1) \cdot (-q^{6m-18})$ \\ \hline
$E_{8}(b_{5})$, $[1^3]$  & $-$ & $ (q^{m-1} - 1) \cdot q^{5m-15}$ \\ \hline
$E_{7}(a_{1})$  &  $1,3$ & $(q^{m-1} - 1) \cdot q^{6m-14} \frac{[m-3]}{[2]}$ \\ \hline %#12   degs 1,3  A_1
$E_{8}(a_{5})$ &  $1$  & $(q^{m-1} -1) \cdot q^{7m -19}$ \\ \hline   %#10,11   degs 1
$E_{8}(a_{5})$, $[1^2]$ & $-$ & $0$\\ \hline
$E_{8}(b_{4})$  &  $1$ & $(q^{m-1} -1) q^{7m -17}$ \\ \hline   %#8,9   degs 1
$E_{8}(b_{4})$, $[1^2]$ & $-$  & $0$\\ \hline
%$E_{7}$ & $q^{7m -13} \frac{q^{m-1} - 1}{[2]}$ \\ \hline   %#7   degs 1  A_1
%$E_{8}(a_{4})$  & $q^{7m -15} (q^{m-1} -1)$   \\ \hline    % degs 1, 2 (sgn)   
%$E_{8}(a_{4})$, sgn  & $-q^{6m -13} (q^{m-1} -1)$ \\ \hline
%$E_{8}(a_{3})$  & $q^{7m -13} (q^{m-1} -1)$ \\ \hline   % degs 1, 2 (sgn)
%$E_{8}(a_{3}), sgn$  & $-q^{6m -11} (q^{m-1} -1)$ \\ \hline
%$E_{8}(a_{2})$  & $q^{7m - 11} (q^{m-1} -1)$ \\ \hline  % degs 1
%$E_{8}(a_{1})$  & $q^{7m - 9} (q^{m-1} -1)$ \\ \hline   % degs 1
%$E_{8}$ & $q^{8m-8}$ \\ \hline
\hline

\end{tabular}
\end{center}
%\label{}
%\end{table}

%\begin{table}
\begin{center}
\begin{tabular}{|c|c|c|} \hline
\multicolumn{3}{|c|}{$E_{8}$, part 4} \\ \hline
\multicolumn{1}{|c|}{$(e, \phi)$}
& \multicolumn{1}{|c|}{$m_i$}
& \multicolumn{1}{|c|}{$f_{e, \phi}$}
\\ \hline \hline
$E_{7}$ &  $1$ & $q^{7m -13} \frac{[m-1]}{[2]}$ \\ \hline   %#7   degs 1  A_1
$E_{8}(a_{4})$  &  $1$ & $(q^{m-1} -1) \cdot q^{7m -15}$   \\ \hline    % degs 1, 2 (sgn)   
$E_{8}(a_{4})$, $[1^2]$ &  $2$  & $ (q^{m-1} -1) \cdot (-q^{6m -13})$ \\ \hline
$E_{8}(a_{3})$  &  $1$ & $(q^{m-1} -1) \cdot q^{7m -13}$ \\ \hline   % degs 1, 2 (sgn)
$E_{8}(a_{3})$, $[1^2]$ & $2$  & $(q^{m-1} -1) \cdot (-q^{6m -11})$ \\ \hline
$E_{8}(a_{2})$  &  $1$ & $(q^{m-1} -1) \cdot q^{7m - 11}$ \\ \hline  % degs 1
$E_{8}(a_{1})$  &  $1$ & $ (q^{m-1} -1) \cdot q^{7m - 9}$ \\ \hline   % degs 1
$E_{8}$ & $-$ & $q^{8m-8}$ \\ \hline
\hline

\end{tabular}
\end{center}

%%%%%%%%%%%%%%%%%%%%%%%%%%%%%%%%%%%%%%%%%%%%%%%%%%%%%%%%%%%%%%%
\section{Proof of Theorem~\ref{divisibility-and-nonnegativity-theorem_all}}
\label{proof-of-divisibility-and-nonnegativity-section}
%%%%%%%%%%%%%%%%%%%%%%%%%%%%%%%%%%%%%%%%%%%%%%%%%%%%%%%%%%%%%%%

Before recalling here the statement of the theorem, and giving
its proof, let us review some of the terminology.
We denote $R = \rank(Z_G(e))$, while
$H^*(\BBB_e)$ denotes the cohomology of the Springer fiber 
for $e \in \0$, regarded as a $W$-representation.
Lastly, the ill-behaved nilpotent
orbits from \eqref{ill-behaved-orbits} are
$$
F_4(a_3), \,
E_6(a_3), \,
E_6(a_3)+A_1, \, 
E_7(a_5), \,
E_7(a_3), \,
E_8(a_7), \,
E_8(a_6), \,
E_8(b_5), \, 
E_8(a_4), \,
E_8(a_3).
$$
\vskip.1in
\noindent
{\bf Theorem~\ref{divisibility-and-nonnegativity-theorem_all}.}
{\it 
Let $e$ be a nilpotent element {\bf not} among the 
ill-behaved orbits from \eqref{ill-behaved-orbits},
and assume that $f_{e,\phi}$ is not identically zero.
Then there exists $L, c \in \NN$, independent of $\phi$, such that
$$f_{e,\phi}(m ;q) = \prod^{L}_{j=1} (q^{m+1-2j} -1) \cdot q^{cm} \cdot g_{\phi}(m; q),$$
where $g_{\phi}(m;q)$ is the sum of at most two products of the form $q^{-z} \displaystyle \prod^{R}_{i=1} \tfrac{[m-a_i]_q}{[b_i]_q}$ for some $a_i, b_i, z \in \NN$.
Moreover,
\begin{enumerate}
\item[(i)] For each very good $m$, the polynomial $q^{cm} \cdot g_{\phi}(m; q)$ 
lies in $\NN[q]$.
\item[(ii)] The rank $r$ of $\ggg$ equals $L + c + R$.
\item[(iii)] The multiplicity of $V$ in the $W$-representation $H^*(\BBB_e)$ 
is $r-c$. 
\item[(iv)] If $e$ is principal-in-a-Levi, then $L=0$. 
In particular, $f_{e,\phi}(m;q) \in \NN[q]$ for each very good $m$. 
\item[(v)] If $e$ is not principal-in-a-Levi, then $L \geq 1$.  In the exceptional types it always happens that $L=1$.
\end{enumerate}

Even when $e$ is one of the ill-behaved orbits from
\eqref{ill-behaved-orbits}, at least
for the case when $\phi=1$, the polynomials $f_{e,1}(m;q)$ 
are always nonzero,
and still have properties (i),(ii),(iv),(v) listed above.
}

\vskip.1in
Let us embark on the proof.
The fact that $f_{e,\phi}$ takes the form asserted in the theorem follows 
from inspection of the formulas for the $f_{e, \phi}$.   
The formula in part (ii) is a consequence of \eqref{formula:f}
and the fact that 
$$\sum_x  \frac{\phi(x)} {|Z_{G^F}(e_x)|} \neq 0,$$ 
whenever $\phi=1$ or $e$ does not belong to one of the orbits in \eqref{ill-behaved-orbits}.
This also explains why $f_{e,1}$ is always nonzero.
The formula in part (iii) is a consequence of \eqref{formula:f}
and the fact that 
$$\sum_x  \frac{\left( \wedge^{d_{\kappa}} \pi_{\kappa} \right) (x) \cdot \phi(x)} {|Z_{G^F}(e_x)|} \neq 0,$$ 
whenever $e$ does not belong to one of the orbits in \eqref{ill-behaved-orbits}, so that 
$c =  r - (\kappa -1 + d_\kappa)$.

For (i), (iv), (v), it remains to show that $L = 0$ if and only if $e$ is principal-in-a-Levi and that $g_{\phi}(m;q)$ has the desired positivity property.  We do this case-by-case.

\subsection{Type $A$}
Since every nilpotent orbit in type $A_n$ is principal-in-a-Levi,
we need to show that for every $\lambda \in \Par(n)$ that 
when $\gcd(m,n)=1$ one has
$$
\frac{1}{[m]_q}\qbin{m}{\mu(\lambda)}{q} \in \NN[q].
$$
%lying in $\NN[q]$.  
According to \cite[Corollary 10.4]{RSW},
this follows if all the $\mu_j(\lambda)$'s together
with $m$ have trivial greatest common divisor.
But this is true since a common divisor of
all the $\mu_j(\lambda)$'s would also be a divisor of 
$
n=|\lambda|=\sum_j j \mu_j(\lambda),
$
and we assumed that $\gcd(m,n)=1$.

\subsection{Types $B, C$}
%As noted in Example~\ref{principal-in-a-Levi-example},
For $\lambda$ in $\Par_B(2n+1)$ or $\Par_C(2n)$, 
the orbit $\0_{\lambda}$ is principal-in-a-Levi if and only if $\hat{L}(\lambda) = 0$.
Thus whenever $\0_\lambda$ is not principal-in-a-Levi,
so that $\hat{L}(\lambda) > 0$,
%the formula for the $q$-Kreweras numbers in Theorem~\ref{q-Kreweras-formula-theorem}(types $B_n,C_n$)
the formula for $f_{e,\phi}$ in Proposition \ref{BC_most_general}
contains as a factor the 
product $\prod_{i=1}^{\hat{L}(\lambda)} (q^{m-2i+1}-1)$.
%,which has a factor of $q^{m-1}-1$.
On the other hand, if $\hat{L}(\lambda) = 0$, this product
is empty, and $f_{e,\phi} \in \NN[q]$
because it is a power of $q$ times a $q$-multinomial.
For the same reason in all cases, aside from this product, the remaining factor lies in $\NN[q]$.

\subsection{Type $D$}
%As noted in Example~\ref{principal-in-a-Levi-example},
For $\lambda \in \Par_D(2n)$,
the orbit $\0_\lambda$ is principal-in-a-Levi if and only if 
$L(\lambda)=0$ or $L(\lambda) =2$ with $\mu_1$ odd.
Note that $|L(\lambda)|$ is always even since $\lambda$ is partition of $2n$. 
We examine separately the three conditions on $\lambda$ in 
Proposition~\ref{D_most_general}.
%Theorem~\ref{q-Kreweras-formula-theorem}(type $D_n$).
\begin{itemize}
\item If $\mu_1$ is odd, then $L(\lambda) \geq 2$,
or equivalently ${\hat{L}(\lambda}) \geq 1$.
Thus in this case, $\0_\lambda$ is principal-in-a-Levi if and only if $\hat{L}(\lambda) = 1$.
Thus the product $\prod_{i=1}^{\hat{L}(\lambda)-1} (q^{m-2i+1}-1)$ is non-empty 
exactly when $\0_\lambda$ is not principal-in-a-Levi.  
The remaining factors in $f_{e,\phi}$ all lie in $\NN[q]$.
%contains a factor of $q^{m-1}-1$ when $\hat{L}(\lambda) > 1$,
%and otherwise $f_{e,\phi} \in \NN[q]$ because it is a power of $q$ times a multinomial.
\item When $L(\lambda) \geq 2$ and $\mu_1$ is even, 
then $\0_\lambda$ is never principal-in-a-Levi.
Since $\hat{L}(\lambda) \geq 1$,
the product $\prod_{i=1}^{\hat{L}(\lambda)} (q^{m-2i+1}-1)$
is always non-empty.
The other terms in $f_{e, \phi}$ lie in $\NN[q]$.
\item When $\hat{L}(\lambda)=0$,
$\0_\lambda$ is always principal-in-a-Levi, and 
 $f_{e,\phi} \in \NN[q]$ because it a sum
of a $q$-multinomial and a product of two $q$-multinomials, shifted by powers of $q$.
\end{itemize}

\begin{remark}
In types $A, B, C,$ as well as in the first case of type $D$,
those $\0_\lambda$ which are principal-in-a-Levi not only have
$\Krew(\Phi,\0_\lambda,m;q)$ in $\NN[q]$, but also have
their coefficient sequence symmetric-- this follows in type $A$
from the same result  \cite[Corollary 10.4]{RSW} quoted earlier,
and follows in the other types because $q$-multinomials have this
property.  

However, this is {\it not} in general true for the third
case in type $D$, even though they are always
principal-in-a-Levi.  For example, when $\lambda=(3,3,1,1)$ in $\Par_D(8)$ 
one has 
$$
%\begin{aligned}
\Krew(D_4,\0_{(3,3,1,1)},m;q)
%&=q^{\exp(\lambda,m)}
%     \left( q^{\frac{\ell(\lambda)}{2}-\tau(\lambda)}
%     \qbin{\hat{m}}{\hat{\mu}(\lambda)}{q^2}  + 
%            q^{\frac{\ell(\lambda)}{2}-\mu_1(\lambda)}
%               \qbin{\hat{m}}{\hat{\mu}_{\geq 2}(\lambda)}{q^2} 
%                \qbin{\hat{m}+1-(\hat{\ell}(\lambda)-\hat{\mu}_1(\lambda))}
%                 {\hat{\mu}_1(\lambda)} {q^2}
%      \right)\\
=
  q^{14} \left( \qbin{\hat{m}}{1,1}{q^2}  + 
  \qbin{\hat{m}}{1}{q^2}^2 \right)
%\end{aligned}
$$
which equals 
$
2 q^{14} + 4q^{16} + 6q^{18} + 7q^{20} + 5q^{22} + 3q^{24} + q^{26}
$ 
when $m=9$.
\end{remark}

\subsection{Exceptional Types}

In the exceptional types many $f_{e,\phi}(m;q)$ can be related to $\Cat(W',m; q)$
for some Weyl group $W'$, which has the desired positivity property.  

Most of the remaining cases 
can be handled by writing  
\begin{equation}  \label{basic_form}
\prod^{R}_{i=1} \frac{[m-a_i]_q}{[b_i]_q}
\end{equation}
as a product of polynomials in $q$ with positive coefficients 
as in the paper of Krattenthaler-M\"uller \cite{KrattenthalerMuller2}.  This is accomplished by restricting $m$ to a fixed congruence class modulo the least common multiple of the $b_i$'s (with $m$ also relatively prime to $h$).   
We wrote a program in Sage \cite{sage}, posted on the second author's webpage, that accomplishes this task, except for a handful of cases, making use of \cite[Corollary 6]{KrattenthalerMuller2}, which states that 
\begin{equation} \label{KM}
\frac{[\gamma]_q [ab]_q} { [a]_q [b]_q }
\end{equation}
is a polynomial in $q$ with positive coefficients when $\gcd(a,b)=1$ and $\gamma \geq (a-1)(b-1)$. 

\begin{example}
Let $e$ be of type $A_1$ in $E_8$.  When $m \equiv 17$ modulo $2520$, we find that 
$$
\frac{ [m-1] [m-7][m-11][m-13][m-17][m-19][m-23]}
{ [2][6][8][10] [12][14][18]}
$$
is equal to 
$$\left [ \frac{m- 17 }{504} \right ]_ {q^{504}} 
\left( \frac{ [ \frac{m- 11}{6}]_ {q^6} [ 84 ]_ {q^6}  }{ [ 3 ]_ {q^6} [ 28 ]_{q^6}    }  \right)
\left( \frac{ [ \frac{m- 19}{2} ]_ {q^2} [ 84 ]_{q^2} }{  [ 7 ]_{q^2}  [ 12 ]_{q^2}  }  \right)
\left( \frac{ [  \frac{m- 13}{4} ]_ {q^4} [ 6 ]_ {q^4}  }{ [ 3 ]_{q^4}  [ 2 ]_{q^4} } \right)
\left[  \frac{m- 7}{10} \right ]_ {q^{10}}
\left[   \frac{m- 23}{6} \right ]_ {q^6}
\left[    \frac{m- 1}{2} \right ]_ {q^2}.$$
Each term is a polynomial with positive coefficients, using \eqref{KM} for the expressions in parentheses. 
%There are multiple ways to write and this one isn't even the most elementary one, since m-1 is divisible by 8
\end{example}

The remaining cases are a few of those where $g_{\phi}(m;q)$ is a sum of two terms of the form in \eqref{basic_form}.   For some congruence classes, each expression of the form  \eqref{basic_form} alone will not even  be polynomial, let alone positive.  We give an example that illustrates how these cases are handled.

\begin{example}
Let $e$ be of type $A_2$ in $E_8$.   The expression for $f_{e,1}(m;q)$, up to a power of $q$, is
$$
\frac{ [m-1] [m-7][m-11][m-13][m-17]\left( [m-5] +q^4[m-23] \right)}
{ [2][6][8][10] [12][18]}.
$$
As long as $\gcd(m, 30)=1$ and $m \not\equiv 29$ modulo $30$, the program returns $f_{e,1}$ as a sum of polynomials
with positive coefficients, possibly after rewriting $[m-5] +q^4[m-23]$ as $[m-19]+q^4[m-9]$.
Otherwise, neither summand as in \eqref{basic_form} is polynomial and this also holds even if we rewrite
$[m-5] +q^4[m-23]$ as $[m-19]+q^4[m-9]$.  In such cases we need to deal with the full expression $[m-5] +q^4[m-23]$. 
For example, when $m \equiv 29$ modulo $360$, we can write 
$ \left[ m-1 \right] \left( \frac{[m-5] +q^4[m-23]}{[6][10]} \right)$ as 
$$\left[ \frac{m-1}{2}  \right]_{q^2} \left( (q^{10}+q^{24}) 
\cdot \left[ \frac{m- 29}{30} \right]_ {q^{30}} \cdot 
%\left( \frac{[ 15 ]_{q^2} }{  [ 3 ]_{q^2}  [ 5]_{q^2}  }  \right) 
\frac{[ 15 ]_{q^2} }{  [ 3 ]_{q^2}  [ 5]_{q^2}  }  
%+  \left( \frac{[12]_{q^2}+q^4 [3]_{q^2} }{ [3]_{q^2} [5]_{q^2} } \right) \right).$$
+   \frac{[12]_{q^2}+q^4 [3]_{q^2} }{ [3]_{q^2} [5]_{q^2} }  \right).$$
Then $ \frac{[12]_{q^2}+q^4 [3]_{q^2} }{ [3]_{q^2} [5]_{q^2}} = q^{10} - q^{8} + q^4 - q^2+1$, so the product of this polynomial
with $\left[ \frac{m-1}{2}  \right]_{q^2}$ has positive coefficients as in the proof of Corollary 6 in \cite{KrattenthalerMuller2}. The remaining terms in $f_{e,1}$
are
$$
\left[    \frac{m- 7}{2} \right ]_ {q^2}
\left[ \frac{m- 11}{18}  \right ]_ {q^{18}} 
\left[ \frac{m- 13}{8}  \right ]_ {q^{8}} 
\left[ \frac{m- 17}{12}  \right ]_ {q^{12}},
$$
and so $f_{e,1}(m;q)$ has positive coefficients when $m \equiv 29$ modulo $360$.   
\end{example}

All cases in the exceptional groups can be handled by reducing to the Catalan case or the case of one of these two examples.  It would be nice to have a uniform proof, or at least one that makes use of the fact that $f_{e,\phi}(m;q)$ is polynomial for all very good $m$.

This completes the proof of 
Theorem~\ref{divisibility-and-nonnegativity-theorem_all}.

%%%%%%%%%%%%%%%%%%%%%%%%%%%%%%%%%%%%%%%%%%%%%%%%%%%%%%%%%%%%%%%

\section{$q$-Narayana formulas}
\label{proof-of-Narayana-formulas-section}

In this section we prove, in types $A,B,C$, that 
the $q$-Kreweras numbers,
%as given by the formulas in Theorem~\ref{q-Kreweras-formula-theorem},
when summed over nilpotent orbits $\0$ with a fixed
value of the statistic $d(\0)$ as in \eqref{q-Narayana-definition},
give the $q$-Narayana formulas in  Theorem~\ref{q-Narayana-formula-theorem}.

In type $A_{n-1}$ we have $d(\0_\lambda) = \ell(\lambda)-1$ from the formula 
for $\Krew(A_{n-1},\0_{\lambda},m;q)$ since $r = n-1$.
Thus we want to show that for $k$ in the range $0 \leq k \leq n-1$ that 
\begin{equation}
\label{desired-q-Narayana-summation-in-type-A}
\sum_{\substack{\lambda \in \Par(n):\\ \ell(\lambda)=k+1}} \Krew(A_{n-1},\0_{\lambda},m;q)
=
\frac{q^{(n-1-k)(m-1-k)}}{[k+1]_q} \qbin{n-1}{k}{q} \qbin{m-1}{k}{q}.
\end{equation}
In types $B_n$ and $C_n$, 
%in light of \eqref{type-BC-kappa-formula} and the identity 
we have $d(\0) = \hat{\ell}(\lambda)$, so we wish to show
for $k$ in the range $0 \leq k \leq n$ that
\begin{equation}
\label{desired-q-Narayana-summation-in-types-BC}
%\sum_{\substack{\lambda \in \Par_B(2n+1):\\ \sBC=k}} \Krew(B_n,e_{\lambda},m;q)
%=\sum_{\substack{\lambda \in \Par_C(2n):\\ \sBC=k}} \Krew(C_n,e_{\lambda},m;q)
\sum_{\substack{\lambda \in \Par_B(2n+1):\\ \hat{\ell}(\lambda) = k}} \Krew(B_n, \0_{\lambda},m;q)
=\sum_{\substack{\lambda \in \Par_C(2n):\\ \hat{\ell}(\lambda) = k}} \Krew(C_n, \0_{\lambda},m;q)
=
(q^{2})^{(n-k)(\hat{m}-k)} \qbin{n}{k}{q^2} \qbin{\hat{m}}{k}{q^2}.
\end{equation}

We now give a proof that relies on counting the number of nilpotent elements (over a finite field) of certain prescribed rank.
In a sequel paper we give some alternative proofs.

%There are essentially three different kinds of 
%proofs of this available for us in type $A$, and two of them
%also work in types $B, C$.  We discuss these in the next three
%sections.
%\subsection{Proof 1:  Via formulas counting nilpotents by rank}

\subsection{Type A}

The sum on the left in \eqref{desired-q-Narayana-summation-in-type-A}
is over nilpotent orbits $\0_\lambda$ with $\ell(\lambda)=k+1$.
In the formula \eqref{formula:typeA} for $f_{e,1}$, all the terms depend only on $\ell(\lambda)$
except for $|Z_G^F(e)|$, where $e=e_\lambda \in \0_\lambda$.  Now $\ell(\lambda)=k+1$ means that 
$e_\lambda$ is of rank $n-k-1$ when viewed as an $n \times n$-matrix.  
It follows that the number of nilpotent $n \times n$-matrices of rank $n-k-1$ over $\FF_q$ is given by
$$
%\begin{aligned}
\sum_{\substack{\lambda \in \Par(n):\\ \ell(\lambda)=k+1}}
 \frac{|G^F|}{|Z_{G^F}(e_{\lambda})|},
%\end{aligned}
$$
and thus 
by \eqref{formula:typeA}
the sum in \eqref{desired-q-Narayana-summation-in-type-A} becomes
\begin{equation}
\label{type-A-nilpotents-by-rank-equation}
%\begin{aligned}
%\sum_{\substack{\lambda \in \Par(n):\\ \ell(\lambda)=k}} \Krew(A_{n-1},e_{\lambda},m;q) = % &\\
% & \qquad \qquad 
q^{m(n- k-1) + \binom{k+1}{2}}\frac{(q-1)^{k+1} [m-1]!_q}{[m-k-1]!_q} 
%\cdot \frac{\#\{\text{nilpotent $n \times n$ matrices of rank }n-k  \}}{|G^F|}
\cdot \frac{\#\{\text{nilpotent $n \times n$ matrices of rank }n-k-1  \}}{|G^F|}.
%\end{aligned}
\end{equation}
The number of nilpotent matrices of rank $n-k-1$ equals
(see \cite{Crabb, Lus1}) 
$$q^{\binom{n-k-1}{2}} \frac{(q-1)^{n-k-1}[n]!_q}{[k+1]!_q} \qbin{n-1}{k}{q}$$
and $|G^F| = q^{\binom{n}{2}} (q-1)^n [n]!_q$.  Substituting these into  
\eqref{type-A-nilpotents-by-rank-equation} gives
\eqref{desired-q-Narayana-summation-in-type-A}.

%$$q^{m(n-k) + \binom{k}{2} + \binom{n-k}{2} - \binom{n}{2}}  = q^{mn - mk +k^2/2 - k/2 + (n-k)^2/2 -(n-k)/2 - n^2/2 + n/2} = q^{mn - mk +k^2-nk}$$

%\medskip

\subsection{Types B, C}
The sums in \eqref{desired-q-Narayana-summation-in-types-BC}
are over orbits $\0_\lambda$ 
where $\hat{\ell}(\lambda)$ is fixed.  
As in type $A_{n-1}$, the formula for $f_{e,1}$ in \eqref{BC_almost} depends only on 
$\hat{\ell}(\lambda)$ except for $|Z_G^F(e)|$, where $e=e_\lambda \in \0_\lambda$.  
Now  $\ell(\lambda)$ is the dimension of the kernel of $e$ in the standard representation of $\ggg$.
Thus the rank of $e_\lambda$ is $2n+1- \ell(\lambda)$ in type $B_n$ and $2n-\ell(\lambda)$ in type $C_n$.
%In type $B_n$, $\ell(\lambda)$ is always odd (i.e., there are no elements of odd rank).
Therefore the condition that $\hat{\ell}(\lambda) = k$ means 
that the rank of $e_\lambda$ is either $2n - 2k$
or $2n-2k-1$.  Now in type $B_n$, $\ell(\lambda)$ is always odd and so in particular there are no elements of odd rank.
Using this interpretation and \eqref{BC_almost}, 
the sums in \eqref{desired-q-Narayana-summation-in-types-BC} become

\begin{equation}
\label{type-BC-nilpotents-by-rank-equations}
\begin{array}{rcl}
\displaystyle\sum_{\substack{\lambda \in \Par_B(2n+1):\\ \hat{\ell}(\lambda) = k}} \Krew(B_n,e_{\lambda},m;q) 
&=& q^{m(n-k)+k^2} \displaystyle \prod_{j=1}^k (q^{m-(2j-1)}-1) \cdot
 \frac{
   \# \left\{ \begin{matrix}
              \text{nilpotent elts in $\ggg$ of rank } \\
               2(n\!-\!k)
              \end{matrix}  \right\}
  }{|G^F|}  \\
\displaystyle\sum_{\substack{\lambda \in \Par_C(2n):\\  
                 \hat{\ell}(\lambda) = k}} \Krew(C_n,e_{\lambda},m;q) 
&=& q^{m(n-k)+k^2} \displaystyle\prod_{j=1}^k (q^{m-(2j-1)}-1) \cdot 
  \frac{
    \#\left\{ \begin{matrix}
             \text{nilpotent elts in $\ggg$ of rank } \\
             2(n\!-\!k) \text{ or } 2(n\!-\!k)\!-\!1  
              \end{matrix} \right\}}{|G^F|} 
\end{array}
\end{equation}

\begin{lemma}
The number of nilpotent  elements of type $B$ of rank $2s$ is equal to the
number of nilpotent  elements of type $C$ of rank $2s$ or $2s-1$.   Both are
equal to 
$$q^{s^2-s} \frac{\eta(2n)}{\eta(2n-2s)} \qbin{n}{s}{q^2}.$$
\end{lemma} 

\begin{proof}
We use the formulas in \cite[Theorems 3.1, 3.2]{Lus1}.
The stated formula is exactly \cite[Theorems 3.1]{Lus1} for the number of nilpotent elements of rank $2s$ in type $B_n$.
The number of rank $2s$ and rank $2s-1$ elements in type $C_n$ are, respectively,
%$$q^{s^2+s}  \frac{\eta(n)\eta(n-1)}{\eta(s)\eta(n-s)\eta(n-s-1)} \text{ and } q^{s^2-s} \frac{\eta(n)\eta(n-1)}{\eta(s-1)\eta(n-s)\eta(n-s)}.$$
$$q^{s^2+s}  \qbin{n}{s}{q^2} \frac{\eta(2n-2)) (q^{2n-2s} - 1)}{\eta(2n-2s)} 
\text{ and } q^{s^2-s} \qbin{n}{s}{q^2} \frac{\eta(2n-2))(q^{2s}-1)}{\eta(2n-2s)}.$$
Adding these together gives
\begin{eqnarray*}
%q^{s^2-s}  \qbin{n}{s}{q^2} \frac{\eta(2(n-1))}{ \eta(2(n-s-1))}
%\left( q^{2s}+ \frac{q^{2s}-1}{q^{2(n-s)}-1} \right) = & \\
%q^{s^2-s}  \qbin{n}{s}{q^2} \frac{\eta(2(n-1))}{ \eta(2(n-s-1))}
%\left( \frac{ q^{2s}(q^{2(n-s)}-1)}{  q^{2(n-s)}-1    } + \frac{q^{2s}-1}{q^{2(n-s)}-1} \right) = 
%q^{s^2-s}  \qbin{n}{s}{q^2} \frac{\eta(2n)}{ \eta(2(n-s))}.
q^{s^2-s}  \qbin{n}{s}{q^2} \frac{\eta(2n-2)}{ \eta(2n-2s)}
\left( q^{2s}(q^{2n-2s)}-1) + q^{2s}-1 \right) = 
q^{s^2-s}  \qbin{n}{s}{q^2} \frac{\eta(2n)}{ \eta(2n-2s))}.
\end{eqnarray*}
\end{proof}

From the lemma, it is immediate that the two expressions 
in \eqref{type-BC-nilpotents-by-rank-equations}
are equal.  To evaluate them, we write
$$\displaystyle \prod_{j=1}^k (q^{m - (2j-1)}-1) =  \frac{\eta(m)}{\eta(m - 2k)}.$$ 
With $s=n-k$,
the expressions in \eqref{type-BC-nilpotents-by-rank-equations}
evaluate to 
$$q^{m(n-k)+k^2}  \frac{\eta(m)}{\eta(m - 2k)} \cdot 
q^{(n-k)^2-(n-k)}\frac{\eta(2n)}{\eta(2k)} \qbin{n}{k}{q^2} \cdot \frac{1}{|G^F|}.$$
Since $|G^F| = q^{n^2} \eta(2n)$ in both types, this becomes
$$q^{m(n-k)+k^2 + (n-k)^2-(n-k) - n^2}  \qbin{\hat{m}}{k}{q^2} \qbin{n}{k}{q^2}$$ and
the exponent of $q$ simplifies to $2(n-k)(\hat{m}-k)$ as desired, where as before $\hat{m} = \frac{m-1}{2}$.

\begin{remark}
From the description of special nilpotent pieces in \cite{Lus2},
special nilpotent pieces in type $B$ and $C$ 
consist of some nilpotent orbits whose elements are of rank $2s$ or $2s-1$ for some fixed $s$.
It follows that the set of nilpotent elements of rank $2s$ and $2s-1$ is a union of special nilpotent pieces.
Moreover, a special nilpotent piece in type $B$ corresponds to a special nilpotent piece in type $C$ for the same value of $s$. Therefore the equality between the number of elements of rank $2s$ and $2s-1$ in type $B$ and in type $C$ also follows from Lusztig's work that corresponding special pieces  in $B$ and $C$ have the same cardinality \cite[\S 6.9]{Lus2}.
\end{remark}

%%%%%%%%%%%%%%%%%%%%%%%%%%%%%%%%%%%%%%%%%%%%%%%%%%%%%%%%%%%%%%%
\section{Proof of Theorem~\ref{CSP-theorem}}
\label{proof-of-CSP-theorem-section}
%%%%%%%%%%%%%%%%%%%%%%%%%%%%%%%%%%%%%%%%%%%%%%%%%%%%%%%%%%%%%%%

Recall the statement of the theorem:

\vskip.1in
\noindent
{\bf Theorem~\ref{CSP-theorem}.}
{\it  
In types $A,B,C,D$, for $m \equiv 1$ mod $h$, say $m = sh+1$,
and for each $W$-orbit $[X]$ of intersection subspaces 
of reflecting hyperplanes, $\Krew(\Phi,\0_X,m;q=\omega_{d})$ counts
those $w_1 \leq \cdots \leq w_s$ in $NC^{(s)}(W)$
which are both 
\begin{enumerate}
\item[$\bullet$] fixed under the action
an element of order $d$ in the $\ZZ/sh\ZZ$-action, and
\item[$\bullet$]
 have the subspace $V^{w_1}$ lying in the $W$-orbit $[X]$.
\end{enumerate}
}

\begin{remark}
We mention here how recent work has generalized the type $A_{n-1}$ special case of
Theorem~\ref{CSP-theorem} from the special case $m=sn+1$ 
to the case of {\it all} very good $m$ in type $A_{n-1}$, that is, where $\gcd(m,n)=1$.
In \cite{ArmstrongRhoadesWilliams},  
Armstrong, Rhoades and Williams introduced for all $m > n$ with 
$\gcd(m,n)=1$ the set $NC(n,m)$ of {\it rational} 
or {\it $(n,m)$-noncrossing partitions}.  These $(n,m)$-noncrossing partitions
are a subset of $NC(m-1)$, specializing to $NC^{(s)}(W)$ when $m=sn+1$.
One might ask whether the subset $NC(n,m) \subset NC(m-1)$ 
is closed under the natural dihedral symmetry group of 
order $2(m-1)$ acting on $NC(m-1)$; this was left open in \cite{ArmstrongRhoadesWilliams},
but later resolved affirmatively in work of 
Bodnar and Rhoades \cite{BodnarRhoades}. In fact,
recent work of Bodnar\footnote{B. Rhoades, personal communication, 2016.} 
has shown how to define $NC(n,m)$ when $m < n$, again with a dihedral group of
order $2(m-1)$ acting.

 In particular,
considering the cyclic $\ZZ/(m-1)\ZZ$-action via rotations,
the Bodnar and Rhoades also proved a cyclic sieving phenomenon
\cite[Thm. 5.1]{BodnarRhoades} whose $m=sn+1$ special case is
equivalent to the type $A_{n-1}$ special
case of Theorem~\ref{CSP-theorem}.  These results
involve a $q$-Kreweras number that differs
slightly from the one in Theorem~\ref{q-Kreweras-formula-theorem},
in that it is
missing the factor of  $q^{m(n-\ell(\lambda))-c(\lambda)}$.
However, this power of $q$ makes no difference in the proof of the cyclic sieving phenomenon,
as it happens to equal $1$ whenever the $q$-Kreweras number
evaluated at an $(m-1)^{st}$ root-of-unity is nonvanishing-- see
Lemma \ref{c(lambda)-congruences-lemma} below.  

The aforementioned work of Bodnar also extends this cyclic sieving phenomenon 
to the case $m < n$.
\end{remark}

\vskip.1in
In proving Theorem~\ref{CSP-theorem},
our plan will be to first compute the
evaluations $\Krew(\Phi,\0_X,m;q=\omega_{d})$, and
then compare them with the combinatorial models for
$NC^{(s)}(W)$ in the classical types $A,B,C,D$.

\subsection{Root-of-unity evaluation lemmas}

We collect here a few well-known
observations on evaluating certain
polynomials in $q$ at a primitive $d^{th}$ root of unity $\omega_d$.
The proofs are straightforward and omitted.

\vskip.1in
\noindent
{\bf Warning:}
For the remainder of the paper, we abandon the convention
that ``$a \equiv b$'' means $a \equiv b \bmod{2}$, as we will
now need to often consider equivalence modulo $d$ for
other moduli $d \neq 2$.

\begin{lemma}
\label{root-of-unity-facts}
Let $\omega_d:=e^{{2 \pi i}{d}}$ or any other primitive $d^{th}$ root-of-unity.
\begin{enumerate}
\item[(i)] The polynomial 
$$
[m]_q:=\frac{1-q^m}{1-q}
$$
has $\omega_d$ as a root with multiplicity $1$ or $0$, 
depending on whether $d$ divides $m$ or not.
\item[(ii)]
For any positive integer $m$, the $q$-factorial 
$$
[m]!_q:= [1]_q [2]_q \cdots [m]_q 
$$ 
has $\omega_d$ as a root of multiplicity 
$\lfloor \frac{m}{d} \rfloor$.
\item[(iii)]
For $d$ dividing $N$, the product
$$
[N]_q [N-1]_q \cdots [N-k+1]
$$ 
has $\omega_d$ as a root of multiplicity $\lceil \frac{k}{d} \rceil$.
\item[(iv)] If $a,b$ are positive integers with $a \equiv b \mod d$, then
$$\lim_{q \rightarrow \omega_d} \frac{[a]_q}{[b]_q} = 
\begin{cases} 
 \frac{a}{b} & \text{ if } a \equiv b \equiv 0 \mod d \\
 1 & \text{ if } a \equiv b \not\equiv 0 \mod d.
\end{cases}
$$
\item[(v)]
For nonnegative $n,k$ expressed uniquely as
$
n=d \cdot \hat{n} + \hat{\hat{n}} 
$
and 
$
k=d \cdot \hat{k} + \hat{\hat{k}}
$
with $0 \leq \hat{\hat{k}}, \hat{\hat{n}} < d$, one has
$$
\lim_{q \rightarrow \omega_d} \qbin{n}{k}{q} = 
\binom{\hat{n}}{\hat{k}} \cdot 
\lim_{q \rightarrow \omega_d} \qbin{\hat{\hat{n}}}{\hat{\hat{k}}}{q} 
$$
\end{enumerate}
\end{lemma}

In types $B, C, D$, we will need to evaluate certain
polynomials in $q^2$ at $q=\omega_d$.
As notation, let
$$
d^-:=\frac{d}{\gcd(2,d)} = 
\left\{
\begin{aligned}
      d & \text{ for }d\text{ odd},\\
     \frac{d}{2}& \text{ for }d\text{ even}.
\end{aligned}
\right\}
\qquad \text{ and } \qquad 
d^+:=\lcm(2,d) =
\left\{
\begin{aligned}
       d & \text{ for }d\text{ even},\\
      2d& \text{ for }d\text{ odd}
\end{aligned}
\right\}
=2 d^-.
$$
Some facts that will be used frequently without mention are that, for an even integer $2N$, 
$$
d\text{ divides }2N \quad \Leftrightarrow \quad
d^+\text{ divides }2N \quad \Leftrightarrow \quad
d^-\text{ divides }N,
$$
and in this situation,
$$
\frac{2N}{d^+} = \frac{N}{d^-}.
$$

The proof of the following assertions are then straightforward.

\begin{lemma}
\label{BCD-root-of-unity-facts}
Assume throughout that {\bf $d$ divides $2N$}.
\begin{enumerate}
\item[(i)] For a sequence of nonnegative integers
$(\alpha_1,\ldots,\alpha_\ell)$, one has 
$$
\lim_{q \rightarrow \omega_d}
\qbin{N}{\alpha_1,\ldots,\alpha_\ell}{q^2}=
\begin{cases}
\binomial{\frac{2N}{d^+}}{\frac{2\alpha_1}{d^+},\ldots,\frac{2\alpha_\ell}{d^+}}
=\binomial{\frac{N}{d^-}}{\frac{\alpha_1}{d^-},\ldots,\frac{\alpha_\ell}{d^-}}
&\text{ if }d\text{ divides }2\alpha_i\text{ for each }i,\\
0&\text{ otherwise.}
\end{cases}
$$
\item[(ii)]
Given a nonnegative integer $k$, for $d=1,2$ one has
$$
\lim_{q \rightarrow \omega_d}
\qbin{N+1}{k}{q^2}=
\binomial{N+1}{k}
$$
but for $d \geq 3$ one has
$$
\lim_{q \rightarrow \omega_d}
\qbin{N+1}{k}{q^2}=
\begin{cases}
\binomial{\frac{2N}{d^+}}{\lfloor\frac{2k}{d^+}\rfloor}
=\binomial{\frac{N}{d^-}}{\lfloor\frac{k}{d^-}\rfloor}
&\text{ if }k \equiv 0,1 \bmod{d^-} \\
0&\text{ otherwise.}
\end{cases}
$$
\end{enumerate}
\end{lemma}

\subsection{The $q$-Kreweras numbers evaluated at roots of unity}
We next use Lemmas~\ref{root-of-unity-facts} and
\ref{BCD-root-of-unity-facts} together with our formulas for
$\Krew(\Phi,\0_\lambda,m,q)$ to evaluate them at $q=\omega_d$ 
whenever $\0_\lambda$ is principal-in-a-Levi, and $d$ divides 
$m-1=sh$ for some $s \geq 1$.  
We first compute only
their complex modulus, ignoring
multiplicative factors of powers of $q$ (Proposition~\ref{evaluations-prop}).
Then we check they are correct on the nose, 
not just up to modulus (Proposition~\ref{q-powers-become-plus-1}).

\begin{proposition}
\label{evaluations-prop}
Let $s$ be a positive integer, and assume that $d$ divides $m-1=sh$.
\begin{itemize}
\item[(i)]
In type $A_{n-1}$ one has $h=n$, so that $m-1=sn$.
Every $\lambda$ in $\Par(n)$ has $\0_\lambda$ principal-in-a-Levi.
Then 
$
\Krew(A_{n-1},\0_\lambda,m;q=\omega_d)
$
is nonvanishing if and only if at most one $\mu_{j_0}(\lambda)$
is not divisible by $d$, and if such a $j_0$ exists, 
then $\mu_{j_0}(\lambda) \equiv 1 \bmod{d}$.  Furthermore, in this situation
$$
\left\lVert
\Krew(A_{n-1},\0_\lambda,m;q=\omega_d)
\right\rVert
=
\left\lVert
\lim_{q \rightarrow \omega_d}
\frac{1}{[m]_q}\qbin{m}{\mu(\lambda)}{q}
\right\rVert
=
\begin{cases}
\displaystyle \frac{1}{m}\binomial{m}{\mu(\lambda)} & \text{ if }d=1,\\
\binomial{ \frac{sn}{d} }
         { \lfloor \frac{\mu(\lambda)}{d}\rfloor } & \text{ if } d \geq 2.
\end{cases}
$$
\item[(ii)] 
In type $B_n, C_n$ one has $h=2n$, so that $m-1=2sn$.
Then $\lambda$ in $\Par_B(2n+1)$ or $\Par_C(2n)$ 
has $\0_\lambda$ principal-in-a-Levi if and only if 
at most one part $j_0$ has $\mu_{j_0}(\lambda)$ odd,
that is $\hat{L}(\lambda)=0$.
Then 
$
\Krew(\Phi,\0_\lambda,m;q=\omega_d)
$
is nonvanishing if and only if $d$ divides $2\hat{\mu}_j(\lambda)$ for each $j$,
in which case
$$
\left\lVert
\Krew(\Phi,\0_\lambda,m;q=\omega_d)
\right\rVert
=
\left\lVert
\lim_{q \rightarrow \omega_d}
\qbin{sn}{\hat{\mu}(\lambda)}{q^2}
\right\rVert
=\binomial{\frac{sn}{d^-}}
          {  \frac{\hat{\mu}(\lambda)}{d^-} }
$$

\item[(iii)] 
In type $D_n$ one has $h=2(n-1)$, so that $m-1=2s(n-1)$.
Then $\lambda$ in $\Par_D(2n)$, 
has $\0_\lambda$ principal-in-a-Levi if and only if either
\begin{itemize}
\item[$\bullet$] there are two part sizes, namely $1$ and
some odd $j_0 \geq 3$, with odd multiplicity, so  $\hat{L}(\lambda)=1$, or
\item[$\bullet$] there are no parts with odd multiplicity, that is, $\hat{L}(\lambda)=0$.
\end{itemize}
In the former $\hat{L}(\lambda)=1$ case,
$
\Krew(D_n,\0_\lambda,m;q=\omega_d)
$
is nonvanishing if and only if $d$ divides $2\hat{\mu}_j(\lambda)$ for all $j$, 
in which case 
$$
\left\lVert
\Krew(D_n,\0_\lambda,m;q=\omega_d)
\right\rVert
=
\left\lVert
\lim_{q \rightarrow \omega_d}
\qbin{s(n-1)}{\hat{\mu}(\lambda)}{q^2}
\right\rVert
=
\binomial{\frac{s(n-1)}{d^-}}
         { \frac{\hat{\mu}(\lambda)}{d^-} }.
$$
In the latter $\hat{L}(\lambda)=0$ case,
$
\Krew(D_n,\0_\lambda,m;q=\omega_d)
$
is nonvanishing if and only if 
both 
\begin{itemize}
\item[(a)] $\mu_j(\lambda) \equiv 0 \bmod{d^+}$ for all $j \geq 2$, and
\item[(b)] $\mu_1(\lambda) \equiv 0 \text{ or } 2 \bmod{d^+}$,
\end{itemize}
in which case 
$$
\begin{aligned}
&\left\lVert
\Krew(D_n,\0_\lambda,m;q=\omega_d)
\right\rVert \\
&=\left\lVert
\lim_{q \rightarrow \omega_d}
\left(
  \qbin{s(n-1)}{\hat{\mu}_{\geq 2}(\lambda)}{q^2} 
\qbin{s(n-1)+1-|\hat{\mu}_{\geq 2}(\lambda)|)}
     {\hat{\mu}_1(\lambda)} {q^2} 
+
q^{\mu_1(\lambda)-\tau_1(\lambda)}  
\qbin{s(n-1)}{\hat{\mu}(\lambda)}{q^2}  
\right)
\right\rVert \\
&=
\begin{cases}
  \binomial{s(n-1)}{\hat{\mu}_{\geq 2}(\lambda)}
\binomial{s(n-1)+1-|\hat{\mu}_{\geq 2}(\lambda)|}
     {\hat{\mu}_1(\lambda)} 
+  \binomial{s(n-1)}{\hat{\mu}(\lambda)} 
    &\text{ if }d=1, \\
   \binomial{s(n-1)}{\hat{\mu}_{\geq 2}(\lambda)}
\binomial{s(n-1)+1-|\hat{\mu}_{\geq 2}(\lambda)|}
     {\hat{\mu}_1(\lambda)} 
+  (-1)^n \binomial{s(n-1)}{\hat{\mu}(\lambda)} 
    &\text{ if }d=2, \\
  \left( 1+(-1)^{\frac{2n}{d}} \right) 
%    \underbrace{ 
       \binomial{\frac{2s(n-1)}{d^+}}
                          { \frac{\mu(\lambda)}{d^+} } 
%}_{\binomial{\frac{s(n-1)}{d^-}}{ \frac{\mu(\lambda)}{d^-} }}  
    & \text{ if }d\geq 3\text{ and }\mu_1(\lambda) \equiv 0 \bmod{d^+},\\

 \binomial{\frac{2s(n-1)}{d^+}}
         { \frac{\mu_1(\lambda)-2}{d^+}, \frac{\mu_{\geq 2}(\lambda)}{d^+} }
%  =\binomial{\frac{2s(n-1)}{d^-}}
%         {  \frac{\hat{\mu}_1(\lambda)-1}{d^-}, \frac{\hat{\mu}_{\geq 2}(\lambda)}{d^-} } 
    & \text{ if }d\geq 3\text{ and }\mu_1(\lambda) \equiv 2 \bmod{d^+}. \\
\end{cases}
\end{aligned}
$$
\end{itemize}
\end{proposition}
\begin{proof}

\vskip.1in
\noindent
{\sf Type A.}
The $d=1$ case is clear, since one is setting $q=\omega=1$.
Thus without loss of generality, $d \geq 2$.
We know from Theorem~\ref{divisibility-and-nonnegativity-theorem} that 
$$
\frac{1}{[m]_q}\qbin{m}{\mu(\lambda)}{q}
=\frac{[m-1]_q [m-2]_q \cdots [m-\ell(\lambda)+1]_q}
{\prod_j [\mu_j(\lambda)]!_q}
$$
is a polynomial in $q$.  
Lemma~\ref{root-of-unity-facts}(ii,iii) tell us that it
has $\omega_d$ as a root of multiplicity 
$\lfloor \frac{\ell(\lambda)-1}{d} \rfloor$
in the numerator, and of multiplicity
$\sum_{j \geq 1} \lfloor \frac{\mu_j(\lambda)}{d} \rfloor$
in the denominator.  Hence one must always have the
inequality
\begin{equation}
\label{type-A-numerator-denominator-inequality}
\left\lfloor \frac{\ell(\lambda)-1}{d} \right\rfloor
\geq \sum_{j \geq 1} \left\lfloor \frac{\mu_j(\lambda)}{d} \right\rfloor
\end{equation}
and this must be an equality whenevever this polynomial 
is nonvanishing at $q=\omega_d$.
Writing $r_j$ for the
remainder of $\mu_j(\lambda)$ on division by $d$ 
with $0 \leq r_j \leq d-1$, equality 
in \eqref{type-A-numerator-denominator-inequality}
would force
$$
\frac{\ell-1}{d} 
\leq
\left\lceil \frac{\ell-1}{d} \right\rceil
=
\sum_{j \geq 1} \left\lfloor \frac{\mu_j(\lambda)}{d} \right\rfloor = 
\sum_{j \geq 1} \frac{\mu_j-r_j}{d} = 
\frac{\ell-\sum_{j \geq 1} r_j}{d}.
$$
Thus $\sum_j r_j \leq 1$, or in other words, at most one
of the $\mu_j(\lambda)$ is not divisible by $d$, and its remainder is $1$.  
In this situation, one can use Lemma~\ref{root-of-unity-facts}(iv) to match up
numerator and denominator factors yielding the asserted
evaluation in (i).

\vskip.1in
\noindent
{\sf Types B, C.}
This follows from Lemma~\ref{BCD-root-of-unity-facts}(i) with $N=sn$.

\vskip.1in
\noindent
{\sf Type D.}
The first case, where $\hat{L}(\lambda)=1$, 
similarly to types $B/C$, follows from 
Lemma~\ref{BCD-root-of-unity-facts}(i) with $N=s(n-1)$.  

In the second case, where $\hat{L}(\lambda)=0$, we must set $q=\omega_d$ in 
\begin{equation}
\label{type-D-case-4-without-q-power}
  \qbin{s(n-1)}{\hat{\mu}_{\geq 2}(\lambda)}{q^2} 
\qbin{s(n-1)+1-|\hat{\mu}_{\geq 2}(\lambda)|}
     {\hat{\mu}_1(\lambda)} {q^2} 
+  q^{\mu_1(\lambda)-\tau_1(\lambda)} \qbin{s(n-1)}{\hat{\mu}(\lambda)}{q^2}.
\end{equation}
Note that Lemma~\ref{BCD-root-of-unity-facts}(i) 
with $N=s(n-1)$ shows that,
whenever condition (a) above fails, both summands in 
\eqref{type-D-case-4-without-q-power} vanish-- the first summand vanishes
because its first factor vanishes.  Similarly, 
Lemma~\ref{BCD-root-of-unity-facts}(ii) with
$N=s(n-1)-|\mu_{\geq 2}(\lambda)|$ shows that whenever
condition (b) above fails, both summands 
in \eqref{type-D-case-4-without-q-power}
vanish-- the second factor in
the first summand vanishes
unless $\mu_1(\lambda) \equiv 0$ or $2 \bmod{d^+}$.

Hence without loss of generality we may assume
both conditions (a),(b) hold, and we examine what
happens when one evaluates
\eqref{type-D-case-4-without-q-power} at $q=\omega_d$
for various values of $d$. 
Note that when $d=1$ so that $q=1$, it gives the asserted
evaluation, so without loss of generality, $d \geq 2$.

\begin{lemma}
A $\lambda$ in $\Par_D(2n)$ with $\mu_j(\lambda) \equiv 0 \bmod{d^+}$ has
$2n \equiv 0 \bmod{d^+}$ and
$
\omega_d^{\mu_1(\lambda)-\tau_1(\lambda)} = (-1)^{\frac{2n}{d}}.
$
\end{lemma}
\begin{proof}
One has 
$$
2n=|\lambda|=\sum_j j \mu_j(\lambda) \equiv 0 \bmod{d^+}
$$
and
$$
\mu_1(\lambda)-\tau_1(\lambda)
= \mu_1(\lambda)-\frac{1}{2}\sum_{j \text{ odd }} \mu_j(\lambda) 
=\frac{1}{2} \left( \mu_1(\lambda)-
                \sum_{\text{ odd }j \geq 3} \mu_j(\lambda) 
            \right)
=\frac{d}{2} \left(\frac{\mu_1(\lambda)}{d}-
                 \sum_{\text{ odd }j \geq 3} \frac{\mu_j(\lambda)}{d} 
            \right) 
$$
and therefore
$$
\begin{aligned}
\omega_d^{\mu_1(\lambda)-\tau_1(\lambda)}
&= \left( \omega_{2d}^d 
      \right)^{
                 \frac{\mu_1(\lambda)}{d} - 
                  \sum_{\text{ odd }j \geq 3} \frac{\mu_j(\lambda)}{d}} 
= (-1)^{  \frac{\mu_1(\lambda)}{d} -
                 \sum_{\text{ odd }j \geq 3} \frac{\mu_j(\lambda)}{d} }\\
&=(-1)^{
                 \sum_{\text{ odd }j} \frac{\mu_j(\lambda)}{d}} 
=(-1)^{
                 \sum_{j} \frac{j \cdot \mu_j(\lambda)}{d}} 
=(-1)^{\frac{2n}{d}}.
\end{aligned}
$$
\end{proof}
\noindent
Together with Lemma~\ref{BCD-root-of-unity-facts}(ii), this gives the
asserted evaluation for $d=2$, as in that case, $(-1)^{\frac{2n}{d}}=(-1)^n.$

Now assume $d \geq 3$ and both conditions (a) and (b) hold.
If $\mu_1(\lambda)\equiv 0 \bmod{d^+}$, then 
setting $q=\omega_d$ in \eqref{type-D-case-4-without-q-power} gives,
after applying Lemma~\ref{BCD-root-of-unity-facts}(i) with 
$N=s(n-1)-|\hat{\mu}_{\geq 2}(\lambda)|$,
$$
\binomial{\frac{2s(n-1)}{d^+}}{\frac{\mu_{\geq 2}(\lambda)}{d^+}}
\binomial{ 
            \frac{2s(n-1)-|\mu_{\geq 2}(\lambda)| }{ d^+ } 
         }
         {
            \frac{\mu_1(\lambda)}{d^+}
         }
+  (-1)^{\frac{2n}{d}} 
     \binomial{ \frac{2s(n-1)}{d^+} }{ \frac{\mu(\lambda) }{d^+} }
= \left( 1 + (-1)^{\frac{2n}{d}} \right)
     \binomial{ \frac{2s(n-1)}{d^+} }{ \frac{\mu(\lambda) }{d^+} },
$$
as desired.  On the other hand, if $\mu_1(\lambda)\equiv 2 \bmod{d^+}$
then the first summand in  \eqref{type-D-case-4-without-q-power} vanishes
because its second factor is zero, and hence 
Lemma~\ref{BCD-root-of-unity-facts}(i) gives (up to sign), the stated
answer.
\end{proof}

\begin{proposition}
\label{q-powers-become-plus-1}
In types $A, B, C, D$ and for a positive integer $s$ and
a divisior $d$ of $m-1=sh$, whenever a principal-in-a-Levi
nilpotent orbit $\0_\lambda$ has
$\Krew(\Phi,\0_\lambda,m;q=\omega_d)$ nonvanishing,
it equals its (nonnegative integer) complex modulus
$\left\lVert \Krew(\Phi,\0_\lambda,m;q=\omega_d) \right\rVert$
given in Proposition~\ref{evaluations-prop}.
\end{proposition}

As preparation for the proof of this proposition, we recall that
the power of $q$ appearing as factor in front of the $q$-Kreweras
formula in Theorem~\ref{q-Kreweras-formula-theorem}(Type A)
involves the quantity
$
c(\lambda) = \sum_j \lambda_j' \lambda_{j+1}',
$
while the corresponding powers of $q$ in 
Theorem~\ref{q-Kreweras-formula-theorem}(Types B,C,D)
involve $c(\lambda)/2$.   Thus the following lemma on values 
modulo $d$, that is, in $\QQ/d\ZZ$, will be helpful.

\begin{lemma}
\label{c(lambda)-congruences-lemma}
Assume $\0_\lambda$ is principal-in-a-Levi
and $\Krew(\Phi,\0_\lambda,m;q=\omega_d) \neq 0$.
\begin{itemize}
\item[$\bullet$] In type $A_{n-1}$, 
\begin{enumerate}
\item[(i)]
one either has $\mu_j(\lambda) \equiv 0 \bmod{d}$ for all $j$, in which case,
$$
c(\lambda) \equiv 0 \bmod{d}, \text{ or }
$$
\item[(ii)]
one $j_0$ has
$\mu_j(\lambda) \equiv 0 \bmod{d}$ for all $j \neq j_0$
and $\mu_{j_0}(\lambda) \equiv 1 \bmod{d}$,
in which case
$$
c(\lambda) \equiv j_0-1 \bmod{d}.
$$
\end{enumerate}
\item[$\bullet$] In types $B_n, C_n$, 
\begin{enumerate}
\item[(i)]
one either has all $\mu_j(\lambda)$ even and divisible by $d$,
in which case
$$
\frac{c(\lambda)}{2} \equiv 0 \bmod{d}, \text{ or }
$$
\item[(ii)]
one $j_0$ has $\mu_{j_0}(\lambda)$ odd,
$\mu_j(\lambda)$ even for $j \neq j_0$,
and $2\hat{\mu}_j(\lambda) \equiv 0 \bmod{d}$ for all $j$, in which case
$$
\frac{c(\lambda)}{2} \equiv \frac{j_0-1}{2} +\hat{\ell}(\lambda) - \hat{\mu}_{j_0}(\lambda) \bmod{d}.
$$
\end{enumerate}
\item[$\bullet$] In type $D_n$, 
\begin{enumerate}
\item[(i)]
one either  has all $\mu_j(\lambda)$ are even, 
$2\hat{\mu}_j(\lambda) \equiv 0 \bmod{d}$ for $j \geq 2$, and
$2\hat{\mu}_1(\lambda) \equiv 0 \text{ or }2 \bmod{d^+}$,
in which case
$$
\frac{c(\lambda)}{2} \equiv 0 \bmod{d}, \text{ or }
$$
\item[(ii)]
one has $\mu_1(\lambda)$ and $\mu_{j_0}(\lambda)$ odd
for a unique odd part size $j_0 \geq 3$, but $\mu_j(\lambda)$ even
for all $j \neq 1,j_0$, 
and $2\hat{\mu}_j(\lambda) \equiv 0 \bmod{d}$ for all $j$,
in which case
$$
\frac{c(\lambda)}{2} \equiv \frac{j_0}{2} 
   + \hat{\mu}_1(\lambda) + \hat{\mu}_{j_0}(\lambda) \bmod{d}.
$$
\end{enumerate}
\end{itemize}
\end{lemma}

\begin{proof}[Proof of Proposition~\ref{c(lambda)-congruences-lemma}]
The arguments are all similar, so we show only the last (hardest) case:
the one in type $D$ with a unique odd $j_0 \geq 3$ for which 
$\mu_1(\lambda), \mu_{j_0}(\lambda)$ are odd,
$\mu_j(\lambda)$ is even for $j \neq 1,j_0$,
and $2\hat{\mu}_j(\lambda) \equiv 0 \bmod{d}$ for all $j$.
Note that 
$$
\mu_i(\lambda)= 
\begin{cases}
2 \hat{\mu}_i(\lambda) & \text{ if }i \neq 1,j_0,\\
1+2 \hat{\mu}_i(\lambda) & \text{ if }i = 1,j_0,\\
\end{cases}
\quad \text{ and hence } \quad
\lambda_j'= \sum_{i \geq j} \mu_i(\lambda) =
\begin{cases}
\sum_{i \geq j} 2\hat{\mu}_i(\lambda) & \text{ if }j \geq j_0+1,\\
1+\sum_{i \geq j} 2\hat{\mu}_i(\lambda) & \text{ if }2 \leq j \leq j_0,\\
2+\sum_{i \geq j} 2\hat{\mu}_i(\lambda) & \text{ if }j =1.\\
\end{cases}
$$
Therefore, using ``$\equiv$'' to denote equivalence 
modulo $d\ZZ$ in $\QQ/d\ZZ$, one has
$$
\frac{\lambda_j'\lambda_{j+1}'}{2}
=\begin{cases}
\frac{1}{2}
\left( \sum_{i \geq j} 2\hat{\mu}_i(\lambda) \right)
 \left( \sum_{i \geq j+1} 2\hat{\mu}_i(\lambda) \right)
  \equiv 0 & \text{ if }j \geq j_0+1, \\
 & \\
\frac{1}{2}
\left( 1+\sum_{i \geq j_0} 2\hat{\mu}_i(\lambda) \right)
 \left( \sum_{i \geq j_0+1} 2\hat{\mu}_i(\lambda) \right)
  \equiv \sum_{i \geq j_0+1} \hat{\mu}(\lambda) & \text{ if }j = j_0, \\
 & \\
\frac{1}{2}
\left(1+\sum_{i \geq j} 2\hat{\mu}_i(\lambda) \right)
 \left(1+ \sum_{i \geq j+1} 2\hat{\mu}_i(\lambda) \right)& \\
  \equiv \frac{1}{2}+\sum_{i \geq j+1}\hat{\mu}_i(\lambda) 
                     +\sum_{i \geq j}\hat{\mu}_i(\lambda)& \\
  \equiv \frac{1}{2}+\hat{\mu}_j(\lambda) & \text{ if }2 \leq j \leq j_0-1, \\
 & \\
\frac{1}{2}
\left(2+ \sum_{i \geq j} 2\hat{\mu}_i(\lambda) \right)
 \left(1+ \sum_{i \geq j+1} 2\hat{\mu}_i(\lambda) \right)
  \equiv 1+\sum_{i \geq 1}\hat{\mu}_i(\lambda) & \text{ if }j =1.
\end{cases}
$$
Consequently,
$$
\begin{array}{rccccccl}
c(\lambda) &=&\sum_j \frac{\lambda_j'\lambda_{j+1}'}{2} & & & & &\\
 &=& \sum_{j \geq j_0+1} \frac{\lambda_j'\lambda_{j+1}'}{2}
   &+& \frac{\lambda_{j_0}'\lambda_{j_0+1}'}{2}
   &+& \sum_{j=2}^{j_0-1} \frac{\lambda_j'\lambda_{j+1}'}{2}
   &+ \frac{\lambda_1'\lambda_{2}'}{2}\\
 &\equiv& 0 
    & +& \sum_{i \geq j_0+1} \hat{\mu}_i(\lambda) 
    & +& \sum_{j=2}^{j_0-1} \left( \frac{1}{2} + \hat{\mu}_j(\lambda) \right) 
    & + \left( 1 +  \sum_{i \geq 1}\hat{\mu}_i(\lambda) \right)\\

 &\equiv& 
   \hat{\mu}_1(\lambda)
    + \hat{\mu}_{j_0}(\lambda)
    + \frac{j_0}{2}\qedhere.& & & & & 
\end{array}
$$
\end{proof}

\begin{proof}[Proof of Proposition~\ref{q-powers-become-plus-1}]
We go through eight cases 
from Proposition~\ref{c(lambda)-congruences-lemma}. Here  
``$\equiv$'' is equivalence in $\QQ/d\ZZ$.

\vskip.1in
\noindent
{\sf Type $A_{n-1}$.}
Here one needs to show that the exponent
$$
E:=m(n-\ell(\lambda))-c(\lambda)
$$
has $E \equiv 0 $. 
In the case $A_{n-1}(i)$, since $d$ divides
$\mu_j(\lambda)$ for all $j$, it also divides
$
n=\sum_j j \mu_j(\lambda)
$
and 
$
\ell(\lambda) = \sum_j \mu_j(\lambda).
$
Also $c(\lambda)\equiv 0$ 
by Proposition~\ref{c(lambda)-congruences-lemma},
so $E \equiv 0 $, as desired.  
In the case $A_{n-1}(ii)$, one has 
$$
n-\ell(\lambda) = \sum_{j \geq 1}(j-1) \mu_j \equiv j_0-1 
$$
and Proposition~\ref{c(lambda)-congruences-lemma}
showed $c(\lambda) \equiv j_0-1 $.  But then $m \equiv 1 $ gives
$$
\begin{aligned}
E = m(n-\ell(\lambda))-c(\lambda) 
&\equiv (n-\ell(\lambda))-c(\lambda) \\
&\equiv j_0-1 - (j_0-1) \equiv 0.
\end{aligned}
$$

\vskip.1in
\noindent
{\sf Types B, C.}
Here one needs to show that the exponent
$
E=
%\exp(\lambda,m)+\epsilon
%=m(n-\hat{\ell}(\lambda))-\frac{c(\lambda)}{2}+\tau(\lambda)-\frac{L(\lambda)}{4}+\epsilon
\BCDexponent+
\sigma(\lambda)
\equiv 0 $, 
using the abbreviations
$$
\begin{aligned}
\BCDexponent
  &:=m(n-\hat{\ell}(\lambda))-\frac{c(\lambda)}{2}-\frac{L(\lambda)}{4} \\
\sigma(\lambda) 
&:=
\begin{cases}
\tau_1(\lambda)+\frac{1}{4} & \text{ in type }B,\\
\tau_0(\lambda) & \text{ in type }C\text{ if }\ell(\lambda)\text{ even},\\
\tau_0(\lambda)+\frac{1}{4}-\frac{\ell(\lambda)}{2} & \text{ in type }C\text{ if }\ell(\lambda)\text{ odd}. \\
\end{cases}
\end{aligned}
$$
Note that since $m \equiv 1$, one has the congruence  
$
E\equiv
n-\hat{\ell}(\lambda)-\frac{c(\lambda)}{2}-\frac{L(\lambda)}{4}
+\sigma(\lambda).
$

In case $B_n/C_n (i)$, 
Proposition~\ref{c(lambda)-congruences-lemma}
says that $\frac{c(\lambda)}{2} \equiv 0 $.
Furthermore $L(\lambda)=0$, and $\mu_j(\lambda)$ all even implies
both $\ell(\lambda) =\sum_j \mu_j(\lambda)$ and
$|\lambda|=\sum_j j \cdot \mu_j(\lambda)$ are even.
Thus one must be in type $C_n$, with $\sigma(\lambda)=\tau_0(\lambda)$, so 
$$
\begin{aligned}
E &\equiv n-\hat{\ell}(\lambda)+\tau_0(\lambda) \\
  &\equiv \sum_j j \hat{\mu}_j(\lambda)
           - \sum_j \hat{\mu}_j(\lambda)
           + \sum_{j\text{ even }} \hat{\mu}_j(\lambda) \\
  &= \sum_{j\text{ odd}} (j-1) \hat{\mu}_j(\lambda)
           + \sum_{j\text{ even}} j \hat{\mu}_j(\lambda). 
\end{aligned}
$$
These last two sums are both even because 
$2\hat{\mu}_j(\lambda)=\mu_j(\lambda)\equiv 0 $ for all $j$.

In case $B_n/C_n (ii)$, 
Proposition~\ref{c(lambda)-congruences-lemma}
says that 
$
\frac{c(\lambda)}{2} \equiv 
\frac{j_0-1}{2}+\hat{\ell}(\lambda)-\hat{\mu}_{j_0}(\lambda) 
.
$
Also $L(\lambda)=1$, so one has
$$
\begin{aligned}
E
&\equiv
n-\hat{\ell}(\lambda)
-\left(
\frac{j_0-1}{2}+\hat{\ell}(\lambda)-\hat{\mu}_{j_0}(\lambda) 
\right)
-\frac{1}{4} +\sigma(\lambda)\\
&\equiv
n-\frac{j_0-1}{2}+\mu_{j_0}(\lambda) 
-\frac{1}{4} +\sigma(\lambda)
\end{aligned}
$$
since $2 \hat{\ell}(\lambda) \equiv 0 $.
In addition, since $\lambda$ is either in $\Par_B(2n+1)$ or $\Par_C(2n)$,
one can rewrite $n$ as
$$
n = \left\lfloor\frac{|\lambda|}{2} \right\rfloor
= \left\lfloor\frac{j_0}{2} \right\rfloor 
    + \sum_j j \hat{\mu}_j(\lambda)
\equiv \left\lfloor\frac{j_0}{2} \right\rfloor + \sum_{j\text{ odd }} \hat{\mu}_j(\lambda).
$$
Therefore
\begin{equation}
\label{processed-BC-expression-for-q-power}
E \equiv
\left\lfloor\frac{j_0}{2} \right\rfloor 
-\frac{j_0-1}{2}
+ \sum_{j\text{ odd }} \hat{\mu}_j(\lambda)
+\hat{\mu}_{j_0}(\lambda) 
-\frac{1}{4} +\sigma(\lambda).
\end{equation}
If $j_0$ is odd, so we are in type $B_n$, then
$\left\lfloor\frac{j_0}{2} \right\rfloor 
=\frac{j_0-1}{2}$, and $\sigma(\lambda)=\tau_1(\lambda)+\frac{1}{4}$, so 
\eqref{processed-BC-expression-for-q-power} becomes
$$
E \equiv
\sum_{j\text{ odd }} \hat{\mu}_j(\lambda)
+\hat{\mu}_{j_0}(\lambda) 
+\tau_1(\lambda)
= \sum_{j \text{ odd}} 2 \hat{\mu}_j(\lambda)
\equiv 0 .
$$
If $j_0$ is even, so we are in type $C_n$, then
$\left\lfloor\frac{j_0}{2} \right\rfloor- 
\frac{j_0-1}{2}=\frac{1}{2}$, and 
$\sigma(\lambda)=\tau_0(\lambda)+\frac{1}{4}-\frac{\ell(\lambda)}{2}$, 
so \eqref{processed-BC-expression-for-q-power} becomes
$$
E \equiv
\frac{1}{2}
+ \sum_{j\text{ odd }} \hat{\mu}_j(\lambda)
+\hat{\mu}_{j_0}(\lambda) 
+\tau_0(\lambda)-\frac{\ell(\lambda)}{2}
= \frac{1-\ell(\lambda)}{2}
+ \sum_j \hat{\mu}_j(\lambda)
= 0.
$$

\vskip.1in
\noindent
{\sf Type D.}
In case $D_n(i)$, one needs to
compare the powers of $q$ in front of 
\eqref{type-D-case-4-without-q-power}
and the $q$-Kreweras formula in
Theorem~\ref{q-Kreweras-formula-theorem}(type $D_n$).
Noting that in the case where
$\mu_1(\lambda) \equiv 2 \bmod{d^+}$, the
factor of $1+(-1)^{\frac{2n}{d}}$ vanishes unless $n \equiv 0 \bmod{d}$,
one finds that here one needs to show that this exponent
$$
E:=\BCDexponent + \tau_1(\lambda) + \frac{\ell(\lambda)}{2}
-\left\{
\begin{aligned}
\tau_1(\lambda) &\quad \text{ if }\mu_1(\lambda) \equiv 0 \bmod{d^+}
                 \text{ and }n \equiv 0 \bmod{d}, \\ 
\mu_1(\lambda) &\quad \text{ if }\mu_1(\lambda) \equiv 2 \bmod{d^+}
\end{aligned}
\right\}
$$
has $E \equiv 0 \bmod{d}$.  Using 
$m \equiv 1, L(\lambda)=0, \hat{\ell}(\lambda)=\frac{\ell(\lambda)}{2}$,
and since Proposition~\ref{c(lambda)-congruences-lemma}
gives $\frac{c(\lambda)}{2} \equiv 0$,
one has 
$$
E \equiv
\begin{cases}
n &\text{ if }\mu_1(\lambda) \equiv 0 \bmod{d^+}
                 \text{ and }n \equiv 0 \bmod{d}, \\ 
n+\tau_1(\lambda) - \mu_1(\lambda) &\text{ if }\mu_1(\lambda) \equiv 2 \bmod{d^+}.
\end{cases}
$$
In the first case, the assumption of case $D_n (i)$ implies $n \equiv 0$.
In the second case, one can compute
$$
\begin{aligned}
E 
&\equiv n+\tau_1(\lambda) -\mu_1(\lambda)\\
&= \sum_j j \hat{\mu}_j(\lambda) 
  + \sum_{j\text{ odd }} \hat{\mu}_j(\lambda) 
  - \mu_1(\lambda) \\
&= \sum_{j \text{ even}} j \hat{\mu}_j(\lambda) 
  + \sum_{\text{ odd } j\geq 3} (j+1) \hat{\mu}_j(\lambda) 
  \,\, + \,\, \left( 2\hat{\mu}_1(\lambda) - \mu_1(\lambda)\right) \\
&\equiv 0 + 0 + 0 = 0.
\end{aligned}
$$

In case $D_n(ii)$, one needs to show that the exponent
$$
\begin{aligned}
E&:=\BCDexponent +\tau_1(\lambda) + m-\frac{\ell(\lambda)}{2}+1 \\
&=\left( m(n-\hat{\ell}(\lambda))
-\frac{c(\lambda)}{2}-\frac{L(\lambda)}{4} \right)
+ \tau_1(\lambda)+ m-\frac{\ell(\lambda)}{2}+1
\end{aligned}
$$
has $E \equiv 0 \bmod{d}$.
Using the facts that 
$m \equiv 1, L(\lambda)=2, \hat{\ell}(\lambda)=\frac{\ell(\lambda)}{2}$,
and since Proposition~\ref{c(lambda)-congruences-lemma} gives 
$
\frac{c(\lambda)}{2} \equiv 
\frac{j_0}{2} 
   + \hat{\mu}_1(\lambda) + \hat{\mu}_{j_0}(\lambda),
$ 
one has
\begin{equation}
\label{last-D-case-q-power-expression}
E \equiv n-\ell(\lambda)+\tau_1(\lambda)+\frac{3}{2}
-\left(
\frac{j_0}{2} 
   + \hat{\mu}_1(\lambda) + \hat{\mu}_{j_0}(\lambda)
\right).
\end{equation}
Note that since all $\mu_j(\lambda)$ are even except for $j=1,j_0$, one
has
$$
\begin{array}{rccccl}
n &=& \displaystyle \frac{|\lambda|}{2} 
  &=&\displaystyle \frac{1}{2}+ \frac{j_0}{2}+ \sum_j j \hat{\mu}_j(\lambda) 
  &\equiv \displaystyle \frac{j_0+1}{2}+ \sum_{j\text{ odd }}  \hat{\mu}_j(\lambda) \\
 & & & &  & \\
\ell(\lambda) 
  &=&  \sum_j \mu_j(\lambda)  
  &=& 2+\sum_j 2\hat{\mu}_j(\lambda)
  &\equiv 2\\
 & & & &  & \\
\tau_1(\lambda)
  &=&\displaystyle \sum_{\substack{j\text{ odd }:\\ \mu_j(\lambda)\text{ even}}} \frac{\mu_j(\lambda)}{2}
  &=&\displaystyle\sum_{\substack{j\text{ odd }:\\ j \neq 1,j_0}} \hat{\mu}_j(\lambda)
  &
\end{array}
$$
Thus one can rewrite \eqref{last-D-case-q-power-expression} as
$$
%\begin{aligned}
E \equiv 
\frac{j_0+1}{2}+ \sum_{j\text{ odd }}  \hat{\mu}_j(\lambda)
-2
+\sum_{\substack{j\text{ odd }:\\ j \neq 1,j_0}} \hat{\mu}_j(\lambda)
+\frac{3}{2}
-\left(
\frac{j_0}{2} 
   + \hat{\mu}_1(\lambda) + \hat{\mu}_{j_0}(\lambda)
\right) 
\equiv 
\sum_{\substack{j\text{ odd }:\\ j \neq 1,j_0}} 2\hat{\mu}_j(\lambda) 
 \equiv 0.\qedhere
%\end{aligned}
$$
\end{proof}

\subsection{Combinatorial models for $NC^{(s)}(w)$}

We review here for the classical types $A,B,C,D$
the combinatorial models for the elements of $NC^{(s)}(W)$, that
is, the $s$-element multichains $w_1 \leq \cdots \leq w_s$ in $NC(W)$.
We also review how to read off the $W$-orbit $[X]$ of
the subspace $X=V^{w_1}$, and the $\ZZ/sh\ZZ$-action on $NC^{(s)}(W)$.

\subsubsection{Type $A$}
\label{Type-A-combinatorial-model-section}

One can identify  $NC^{(s)}(A_{n-1})$ with the 
set of {\it $s$-divisible noncrossing
set partitions} of $\{1,2,\ldots,sn\}$, that is, those whose block sizes
are all divisible by $s$; see \cite[Chapter 3]{Armstrong1}.

Under this identification, if
$w_1 \leq \cdots \leq w_s$ corresponds to an $s$-divisible noncrossing 
partition having block sizes $s\lambda=(s\lambda_1,\ldots,s\lambda_\ell)$
for some partition $\lambda$ of $n$, then the fixed space $V^{w_1}$ will lie in
the $W$-orbit $[X]$ where $W_X$ is the parabolic subgroup
$\symm_\lambda = \symm_{\lambda_1} \times \cdots \times \symm_{\lambda_\ell}$
inside $\symm_n$.

Here $h=n$, and Armstrong's $\ZZ/sh\ZZ$-action on $NC^{(s)}(A_{n-1})$ corresponds,
under this identification with $s$-divisible partitions, 
to the $\ZZ/sn\ZZ$-action that cycles the label set $\{1,2,\ldots,sn\}$ 
within the blocks via
$$ 
\cdots \rightarrow 1 \rightarrow 2  \rightarrow \cdots \rightarrow sn-1
\rightarrow sn \rightarrow 1 \rightarrow \cdots .
$$

\subsubsection{Types $B, C$}
\label{Type-BC-combinatorial-model-section}

One can identify $NC^{(s)}(B_n)=NC^{(s)}(C_n)=$ 
with the subset of $s$-divisible noncrossing partitions of the label set $\{1,2,\ldots,2sn\}$ 
that are invariant under the involution 
$\iota$ swapping lablels $i \leftrightarrow i+sn \bmod{2sn}$;
see \cite[\S 4.5]{Armstrong1}.  For this reason, we relabel 
$\{1,2,\ldots,2sn\}$ as 
\begin{equation}
\label{type BCD-label-set}
\{+1,+2,\ldots,+sn,-1,-2,\ldots,-sn\}=\{\pm 1,\ldots,\pm n\}
\end{equation}
so that $\iota$ swaps $+i \rightarrow -i$ for each $i$.
Note that, since every block $B$ of the $s$-divisible noncrossing partition
must have $\iota(B)$ another block, the noncrossing condition implies that there
can be at most one block $B_0$ with the property $\iota(B_0)=B_0$; we call such
a block $B_0$, if it exists, a {\it zero block}, and call the pairs $\{B, \iota(B)\}$
with $\iota(B) \neq B$ the {\it nonzero blocks}.  

Under the identification, if
$w_1 \leq \cdots \leq w_s$ corresponds to an $s$-divisible noncrossing 
partition with nonzero blocks $\{B_1,\iota(B_1)\}, \ldots,\{B_\ell,\iota(B_\ell)\}$
of sizes $s\nu=(s\nu_1,\ldots,s\nu_\ell)$, then the 
partition $\nu=(\nu_1,\ldots,\nu_\ell)$ will have $|\nu| \leq n$,
and the fixed space $V^{w_1}$ will lie in
the $W$-orbit $[X]$ where $W_X$ is the parabolic subgroup
$B_{n-|\nu|} \times \symm_{\nu_1} \times \cdots \times \symm_{\nu_\ell}$
inside $B_n$.

Here $h=2n$, and the $\ZZ/sh$-action corresponds to a 
$\ZZ/2sn\ZZ$-action cycling the labels $\{\pm 1,\ldots, \pm sn\}$ via
$$ 
\cdots \rightarrow +1 \rightarrow +2  \rightarrow \cdots \rightarrow +sn
\rightarrow -1 \rightarrow -2  \rightarrow \cdots \rightarrow -sn
\rightarrow +1 \rightarrow \cdots .
$$

\subsubsection{Type $D$}
\label{Type-D-combintorial-model-section}

One can identify $NC^{(s)}(D_n)$
with certain set partitions of the same
label set of size $2sn$ as in
\eqref{type BCD-label-set}, but this time
arranged on the inner and outer
boundary of an {\it annulus}, where the $2s(n-1)$ labels
$$
+1,+2,\ldots,+s(n-1),-1,-2,\ldots,-s(n-1)
$$
appear in this order {\it clockwise} on the outer boundary, and
the remaining $2n$ labels
$$
\begin{aligned}
&+(s(n-1)+1), +s((n-1)+2),\ldots,+(sn-1),+sn,\\
&\quad -(s(n-1)+1), -s((n-1)+2),\ldots,-(sn-1),-sn,
\end{aligned}
$$
appear in this order {\it counterclockwise} on the inner boundary.
Given a block in such any such set partition, say that the block is
{\it entirely inner} (resp. {\it entirely outer}) if it only contains
labels from the inner (resp. outer) boundary;  say that the block
is {\it traversing} if it is neither entirely inner nor entirely outer.
Then $NC^{(s)}(D_n)$ is identified with those set partitions that
satisfy these conditions:
\begin{itemize}
 \item[NCD1] {\it noncrossing-ness}: the vertices within a block can all be connected by simple closed curves staying within the annulus, in such a way that curves corresponding to distinct blocks do not cross.
\item[NCD2]{\it $\iota$-stability}: if $B$ is a block, then $\iota(B)$ is also a block.
\item[NCD3]{\it zero-block closure}: if there exists a zero-block $B_0=\iota(B_0)$, then it is unique, and contains all the labels on the inner boundary.
\item[NCD4] {\it strong $s$-divisibility}: when reading elements of a block in any clockwise 
order coming from the planar embedding, their sequence of absolute values
pass through consecutive residue classes in $\ZZ/s\ZZ$.
\item[NCD5] {\it determinacy}: if there are no traversing blocks, then
the outer blocks completely determine the inner blocks in a certain fashion, whose
details are not important to us here; see \cite[\S 7]{Kim1} or \cite[\S 7]{Rhoades}.
\end{itemize}
The conditions NCD1-4 were given by Krattenthaler-M\"uller \cite[\S 7]{KrattenthalerMuller};
while condition NCD5 was inadvertently omitted, and recorded by 
J. S. Kim \cite[\S 7]{Kim1}.

Similarly to types $B_n, C_n$,
under the identification, if
$w_1 \leq \cdots \leq w_s$ corresponds to a partition
with nonzero blocks $\{B_1,\iota(B_1)\}, \ldots,\{B_\ell,\iota(B_\ell)\}$
of sizes $s\nu=(s\nu_1,\ldots,s\nu_\ell)$, then the 
partition $\nu=(\nu_1,\ldots,\nu_\ell)$ will either have
$|\nu|=n$ if there is no zero block or
$|\nu| \leq n-2$ if there is a zero block.
Then the fixed space $V^{w_1}$ will lie in
the $W$-orbit $[X]$ where $W_X$ is the parabolic subgroup
$D_{n-|\nu|} \times \symm_{\nu_1} \times \cdots \times \symm_{\nu_\ell}$
inside $D_n$.

Here $h=2(n-1)$, and Rhoades \cite[\S 7]{Rhoades}
observed that the $\ZZ/sh$-action corresponds to
the $\ZZ/2s(n-1)$-action simultaneously 
\begin{itemize}
\item cycling the outer labels  ({\it clockwise}) via
$$ 
\cdots \rightarrow -s(n-1) 
\rightarrow +1 \rightarrow +2  \rightarrow \cdots \rightarrow +s(n-1)
\rightarrow -1 \rightarrow -2  \rightarrow \cdots \rightarrow -s(n-1)
\rightarrow +1 \rightarrow \cdots 
$$
\item cycling the inner labels ({\it counterclockwise}) via
$$ 
\begin{aligned}
& \cdots \rightarrow -sn 
\rightarrow +(s(n-1)+1) \rightarrow +(s(n-1)+2) \rightarrow \cdots \rightarrow +sn  \\
&\qquad \rightarrow 
 -(s(n-1)+1) \rightarrow -(s(n-1)+2) \rightarrow \cdots \rightarrow -sn
\rightarrow +(s(n-1)+1) \rightarrow \cdots .
\end{aligned}
$$
\end{itemize}

\subsection{Putting it together}

We now assemble the proof of Theorem~\ref{CSP-theorem} in each type $A,B,C,D$.
Thus we assume that $m=sh+1$ for a positive integer $s$, and that $d$ is a divisor of $m-1=sh$.

In each case, the strategy will be to first show that if $w_1 \leq \cdots \leq w_s$ is an element in $NC^{(s)}(W)$ 
having {\it $d$-fold symmetry}, that is, fixed by
an element $c^{\frac{sh}{d}}$ having order $d$ in the cyclic group
$C=\langle c \rangle \cong \ZZ/sh\ZZ$, then the parabolic
subgroup $W_X$ fixing $X=V^{w_1}$ corresponds to a principal-in-a-Levi
nilpotent orbit $\0_\lambda$ that falls into one of the
cases from Proposition~\ref{evaluations-prop}.  
where $\Krew(\Phi,\0_\lambda,m,q=\omega_d)$ is nonvanishing.
Then we will count the number of such $d$-fold symmetric elements, generally relying on known formulas or bijections to objects with known formulas, and see that they agree with
the evaluations in Proposition~\ref{evaluations-prop}.

\subsubsection{Type A}

As in Section~\ref{Type-A-combinatorial-model-section},
we have identified the elements of $NC^{(s)}(A_{n-1})$
with the $s$-divisible noncrossing partition of $\{1,2,\ldots,sn\}$,
that is, those noncrossing partitions having block
sizes $s\lambda$ for some $\lambda$ with $|\lambda|=n$.

When $d=1$, one needs to count those which are
fixed by the identity element in $\ZZ/sn\ZZ$, 
which are all such elements.  In this case, formula 
in Proposition~\ref{evaluations-prop}(i) agrees
with Kreweras's original count for such partitions.

If $d \geq 2$ and the $s$-divisible partition 
has $d$-fold symmetry, then most of blocks will be in orbits of
size $d$, with at most one block which is itself $d$-fold symmetric--
two such blocks would cross each other.

Thus either $\lambda$ has $\mu_j(\lambda) \equiv 0 \bmod{d}$ for all $j$, 
or there exists one $j_0$ having $\mu_j(\lambda) \equiv 1 \bmod{d}$,
matching the description for when
$\Krew(\Phi,\0_\lambda,m,q=\omega_d)$ is nonvanishing from
Proposition~\ref{evaluations-prop}(i).

The proof that in this situation there are exactly
\begin{equation}
\label{Athanasiadis-formula}
\binomial{\frac{sn}{d}}{\lfloor \frac{\mu(\lambda)}{d} \rfloor}
\end{equation}
elements with $d$-fold symmetry,
as in Proposition~\ref{evaluations-prop}(i), was 
sketched in \cite[Theorem 6.2]{BessisR}, but we repeat it here
for completeness.  Such an element is completely determined by
restricting each of its blocks to its intersection with the subset 
$\{1,2,\ldots,\frac{2sn}{d}\}$;  relabel these numbers
$
\{ 1,2,\ldots,\frac{sn}{d},-1,-2,\ldots,-\frac{sn}{d} \}.
$
If there is a (unique) $d$-fold symmetric $j_0$-block, then call its 
restriction the ``zero block''.
It is easily seen that this gives a bijection to 
the type $B_{\frac{sn}{d}}$ 
noncrossing partitions considered in \cite{Reiner} having
$\hat{\mu}_j(\lambda)$ nonzero blocks of size $sj$ for each $j$.
The formula \eqref{Athanasiadis-formula}
then agrees with the count for such type $B$ noncrossing partitions given
by Athanasiadis \cite[Theorem 2.3]{Athanasiadis-block-sizes}.

\subsubsection{Types B, C}

As in Section~\ref{Type-BC-combinatorial-model-section},
we are identifying the elements of $NC^{(s)}(B_n)$ or
$NC^{(s)}(C_n)$ with the
$s$-divisible $\iota$-stable 
noncrossing partition of $\{\pm 1, \pm 2,\ldots,\pm sn\}$.
Such a noncrossing partition
has nonzero blocks $\{B_1,\iota(B_1)\},\ldots,\{B_\ell,\iota(B_\ell)\}$
of sizes $s\nu$ where $\nu=(\nu_1,\ldots,\nu_\ell)$ with $|\nu|\leq n$.
The principal-in-a-Levi nilpotent orbit $\0_\lambda$
corresponding to the same parabolic subgroup $W_X$
has $\lambda=(\nu,\nu,N-2|\nu|)$ where $N=2n+1$ in type $B_n$
and $N=2n$ in type $C_n$.

If the noncrossing partition additionally has $d$-fold symmetry, 
then the nonzero blocks will all lie in orbits of
size $d$.  That is, only the zero block (if present) can have
fewer than $d$ distinct images under the $d$-fold symmetry, and will in fact,
be itself $d$-fold symmetric.
As there will be $2\mu_j(\nu)=2\hat{\mu}_j(\lambda)$ 
nonzero blocks of size $sj$
for each $j$, this means that $2\hat{\mu}_j(\lambda) \equiv 0 \bmod{d}$
for each $j$, matching the description for when
$\Krew(\Phi,\0_\lambda,m,q=\omega_d)$ is nonvanishing from
Proposition~\ref{evaluations-prop}(ii).

Similarly to type $A$, there will be exactly
\begin{equation}
\label{Athanasiadis-formula-again}
\binomial{\frac{2sn}{d^+}}{\frac{2\mu(\lambda)}{d^+}}=
\binomial{\frac{sn}{d^-}}{\frac{\mu(\lambda)}{d^-}}
\end{equation}
elements with $d$-fold symmetry,
matching Proposition~\ref{evaluations-prop}(ii):
restricting each block to its intersection with the subset 
$\{\pm 1,\pm 2,\ldots,\pm \frac{sn}{d}\}$ gives a bijection to 
the type $B_{\frac{sn}{d}}$ 
noncrossing partitions having $\hat{\mu}_j(\lambda)$ 
nonzero blocks of size $sj$ for each $j$,
and \eqref{Athanasiadis-formula-again} again agrees with
Athanasiadis' count \cite[Theorem 2.3]{Athanasiadis-block-sizes} 
for this.

%%%%%%%%%%%%%%%%%%%%
% Not sure we should have examples like this one which
% is temporarily commented out.  
% We can decide later if we should add in a bunch of them...
\begin{comment}

\begin{example}
If $d=4, s=2, n=3$ and $\nu=(1,1)$,
and Figure~\ref{symmetric-ncs-figure} shows one of the
three $2$-divisible partitions of type $B_3$ having $4$-fold
symmetry.  Here $d^+=4$ and $2sn=12$, and $\nu$ has only one nonzero
multiplicity, namely $\mu_1(\nu)=2$, so that
Proposition~\ref{evaluations-prop}(ii)
also predicts that there should be
$\binomial{ \frac{12}{4}  }{ \frac{2 \cdot 2}{4} }
=\binomial{ 3 }{ 1 }=3$
such partitions.

\begin{figure}
\epsfxsize=50mm
\epsfbox{symmetric-ncs.eps}
\caption{The case $d=4, s=2, n=3$ and $\nu=(1,1)$ of the type $B_n$ 
noncrossing $s$-divisible partitions with $d$-fold symmetry.  There are 3 such
partitions: the one shown, along with its rotations by $30$ and by $60$ degrees.}
\label{symmetric-ncs-figure}
\end{figure}

\end{example}

\end{comment}
%%%%%%%%%%%%%%%%%

\subsubsection{Type D}

As in Section~\ref{Type-D-combintorial-model-section},
we are identifying the elements of $NC^{(s)}(D_n)$ 
with certain ``annular''
noncrossing partitions.
It is convenient to consider separately the two cases
where a zero block is present or absent; these will correspond to the
two cases in Proposition~\ref{evaluations-prop}(iii)
where $\hat{L}(\lambda)=1$ and $\hat{L}(\lambda)=0$, respectively.

\vskip.1in
\noindent
{\sf The case with a zero block present.}

Due to the {\it zero-block closure condition $NCD3$},
removing the $2s$ elements on the inner boundary
of the annulus from the zero block gives a bijection between
the subset of elements of $NC^{(s)}(D_n)$ having a zero block, and
the whole set $NC^{(s)}(B_{n-1})$.  Furthermore, this bijection is
equivariant with respect to the $\ZZ/s(n-1)\ZZ$-action on these sets.

Just as in the type $B/C$ case, let
us assume that the nonzero blocks
$\{B_1,\iota(B_1)\},\ldots,\{B_\ell,\iota(B_\ell)\}$
have sizes $s\nu$ where $\nu=(\nu_1,\ldots,\nu_\ell)$ with $|\nu|\leq n-2$.
Then the principal-in-a-Levi nilpotent orbit $\0_\lambda$
corresponding to the same parabolic subgroup $W_X$
has $\lambda=(\nu,\nu,1,j_0)$ where $j_0=2n-1-2|\nu|$.
Again as in the type $B/C$ case, 
$d$-fold symmetry implies that the nonzero blocks all lie in orbits of
size $d$.  Hence there will be $2\mu_j(\nu)=2\hat{\mu}_j(\lambda)$ 
nonzero blocks of size $sj$ for each $j$, 
so $2\hat{\mu}_j(\lambda) \equiv 0 \bmod{d}$
for each $j$.  Also the number of such $d$-fold symmetric elements
should be the same as the formula
in Proposition~\ref{evaluations-prop}(ii), replacing
$n$ by $n-1$.  This exactly matches the conditions
and the formula in the $\hat{L}(\lambda)=1$ case of
Proposition~\ref{evaluations-prop}(iii).

\vskip.1in
\noindent
{\sf The case with no zero block.}

Again assume that the nonzero blocks
$\{B_1,\iota(B_1)\},\ldots,\{B_\ell,\iota(B_\ell)\}$
have sizes $s\nu$ where $\nu=(\nu_1,\ldots,\nu_\ell)$ with $|\nu|= n$.
Then the principal-in-a-Levi nilpotent orbit $\0_\lambda$
corresponding to the same parabolic subgroup $W_X$
has $\lambda=(\nu,\nu)$.

We are claiming that this case will match with the
conditions and formulas appearing in
the $\hat{L}(\lambda)=0$ cases of
Proposition~\ref{evaluations-prop}(iii).
Thus one should expect the analysis to 
break into further subcases based on whether $d=1,2$ or at least $3$.

\vskip.1in
\noindent
{\sf The subcase with no zero block and $d=1$.}

Here one wishes to count {\it all} of the elements of $NC^{(s)}(D_n)$ whose
annular noncrossing partition has no zero block and
nonzero block sizes $s\nu$.  This is given by a
formula of Krattenthaler and M\"uller \cite[Corollary 16]{KrattenthalerMuller},
\cite[Theorem 1.2]{Kim2}.  Using the formulation in
\cite[Theorem 1.2]{Kim2}, and the
notational correspondences $\ell=1$, $b=\mu$, $k=s$, so that $s_1=n-b, s_2=b$,
one obtains a formula equivalent to
the $\hat{L}(\lambda)=0$ case of
Proposition~\ref{evaluations-prop}(iii) with $d=1$.

\vskip.1in
\noindent
{\sf The subcase with no zero block and $d=2$.}

Here one wishes to count the elements of $NC^{(s)}(D_n)$ whose
annular noncrossing partition has no zero block and
nonzero block sizes $s\nu$, and with the additional property
that they are fixed by the
element $c^{s(n-1)}$ of order $2$ inside the cyclic group
$C \cong \ZZ/2s(n-1)\ZZ$.  Note that this
element has different effects on the outer and inner boundary
labels:  on the outer boundary it rotates $180$ degrees,
acting as the map $\iota: i \leftrightarrow -i$, while on the inner boundary
it iterates the $180$-degree rotation $n-1$ times, acting as $\iota^{n-1}$.

Thus whenever the 
annular noncrossing partition has two (nonzero) inner blocks of
size $s$, it is always fixed by $c^{s(n-1)}$: all blocks
will be either entirely inner or outer, and hence
the {\it $\iota$-stability condition $NCD2$} implies the stability
of the blocks under any power of $\iota$.

If the annular noncrossing partition does not have two inner blocks
of size $s$, then all of the labels on its inner boundaries lie
in traversing blocks.  We claim that it is then fixed by
$c^{s(n-1)}$ if and only if $n$ is even:  on the outer
labels one is acting by $\iota$, and
on the inner labels one is acting by $\iota^{n-1}$,
so one needs $\iota^{n-1}=\iota$ to be fixed, that is,
$n$ even.

Thus for $n$ even, all of these elements are fixed by
$c^{s(n-1)}$, and therefore should have the same formula as in
Proposition~\ref{evaluations-prop}(iii) with $d=1$.  Indeed
this agrees with the formula in Proposition~\ref{evaluations-prop}(iii) 
when $\hat{L}(\lambda)=0$ and $d=2$ and $n$ even.

Meanwhile for $n$ odd, the elements fixed by
$c^{s(n-1)}$ are those with two (nonzero) inner blocks of size $s$.
Removing these inner blocks gives an easy bijection to the elements
of $NC^{(s)}(B_{n-1})$ with all nonzero blocks, of sizes
$s\lambda-(s,s)$, where $s\lambda-(s,s)$ is obtained from
$s\lambda$ by removing two copies of the part $s$ (corresponding to
these inner blocks).  Therefore this should have the same formula
as Proposition~\ref{evaluations-prop}(ii) but replacing $\lambda$ with
$\lambda-(1,1)$, namely
$$
\binomial{s(n-1)}{\hat{\mu}_1(\lambda)-1, \hat{\mu}_{\geq 2}(\lambda)}.
$$
Happily, one has an easily checked identity
\begin{equation}
\label{easily-checked-identity}
\binomial{s(n-1)}{\hat{\mu}_1(\lambda)-1, \hat{\mu}_{\geq 2}(\lambda)}
=
\binomial{s(n-1)}{\hat{\mu}_{\geq 2}(\lambda)}
\binomial{s(n-1)+1-|\hat{\mu}_{\geq 2}(\lambda)|}{\hat{\mu}_1(\lambda)}
-\binomial{s(n-1)}{\hat{\mu}(\lambda)},
\end{equation}
whose right side agrees with the value given in 
Proposition~\ref{evaluations-prop}(iii) 
when $\hat{L}(\lambda)=0$ and $d=2$ and $n$ is odd.

\vskip.1in
\noindent
{\sf The subcase with no zero block and $d \geq 3$.}

Note that there is a dichotomy in the conditions and formulas in 
Proposition~\ref{evaluations-prop}(iii) 
when $\hat{L}(\lambda)=0$ and $d \geq 3$.
This will correspond to the following dichotomy in the 
annular noncrossing partitions that model $NC^{(s)}(D_n)$ and have
no zero block.

\begin{proposition}
\label{NCD-traversing-subtlety}
In annular noncrossing partitions modelling
 $NC^{(s)}(D_n)$ with no zero block, either 
\begin{itemize}
\item[(A)] every label on the inner boundary lies in a traversing block, or 
\item[(B)] there are no traversing
blocks, only entirely outer blocks, and two entirely 
inner blocks each of size $s$.
\end{itemize}
\end{proposition}
\begin{proof}
Assume case (B) fails, that is, some inner boundary label $j$
lies in a traversing block $B$.
Then $-j$ lies in $\iota(B)$, which will be a {\it different} 
traversing block, since we are assuming that there
is no zero block.  But then the
the strong $s$-divisibility condition $NCD4$ now
prevents any of the inner boundary labels from lying in an entirely
inner block:  $NCD4$ implies that such a traversing block 
would contain elements having absolute values that achieve
every residue class in $\ZZ/s\ZZ$, which is impossible
since $j$ and $-j$ are the only inner boundary labels
whose absolute values achieve their residue class.
\end{proof}

An immediate corollary is that if an element of $NC^{(s)}(D_n)$
corresponds to an annular noncrossing partition with no zero blocks and
has $d$-fold symmetry, then either it is in case (A) of
Proposition~\ref{NCD-traversing-subtlety} and its blocks all come in orbits of
size $d$ (as they are all entirely outer or traversing),
or it is in case (B), so that its entirely outer blocks come in orbits of size $d$, leaving only the two entirely inner blocks (each of size $s$).
This means that its block sizes $s\lambda$ will satisfy 
$$
\begin{aligned}
\mu_j(\lambda)(=2\hat{\mu}_j(\lambda)) 
&\equiv 0 \bmod{d^+}\text{ for each }j \geq 2,\\
\mu_1(\lambda)(=2\hat{\mu}_1(\lambda)) 
&\equiv
\begin{cases} 
0 \bmod{d^+}&\text{ in case (A),}\\
2 \bmod{d^+}&\text{ in case (B).}
\end{cases}
\end{aligned}
$$
Note that this matches the dichotomy of conditions in
Proposition~\ref{evaluations-prop}(iii) 
when $\hat{L}(\lambda)=0$.  It remains only to check
that the formulas there match the number of $d$-fold
symmetric elements for $d \geq 3$ in each case.

In case (B), we claim that the two entirely inner blocks of size $s$
are always stable under the element $c^{\frac{2s(n-1)}{d}}$ of order $d$
inside the cyclic group $C=\langle c \rangle$.  This is because
one has
$$
2n= \sum_j j \mu_j(\lambda) \equiv 2 \bmod{d^+}
$$
so that $d^+$ divides $2(n-1)$, and hence $s$ divides $\frac{2s(n-1)}{d}$.
Thus these two inner blocks are always $d$-fold symmetric, and they are completely determined by the entirely outer blocks (according to the {\it determinacy condition} $NCD5$).  Therefore 
the number of $d$-fold symmetric elements in case (B) is
the same as the number of $d$-fold symmetric elements of type $B_{n-1}$ having
block sizes $s\lambda-(s,s)$, that is, the formula from
Proposition~\ref{evaluations-prop}(ii),
but replacing $\lambda$ with $\lambda-(1,1)$:
$$
\binom{\frac{2s(n-1)}{d^+}}
      {\frac{\mu_1(\lambda)-2}{d^+},\frac{\mu_{\geq 2}(\lambda)}{d^+}}.
$$
This matches the desired formula in
Proposition~\ref{evaluations-prop}(iii), with 
$\hat{L}(\lambda)=0, d \geq 3$ and
$\hat{\mu}_1(\lambda) \equiv 0 \bmod{d^+}$.

In case (A), we need a further structural observation.

\begin{lemma}
\label{counterclockwise-becomes-clockwise-lemma}
For $d \geq 3$, in case (A),
the annular noncrossing partition 
is fixed by $c^{\frac{2s(n-1)}{d}}$ 
if and only if it has both $d$-fold {\it rotational symmetry} and in
addition, $d$ divides $n$.
\end{lemma}
\begin{proof}
Since in case (A) one has $d$ dividing $2n=\sum_j \mu_j(\lambda)$,
and $d$ also dividing $2s(n-1)$, one concludes that $d$ divides $2s$.
Thus $d$ divides the number of labels on both the inner boundary and the outer boundary.  

Given any block $B$ of a $d$-fold symmetric 
annular noncrossing partition in case (A),
decompose it uniquely as $B=B_i \sqcup B_o$ with inner boundary labels $B_i$ 
and outer boundary labels $B_o$.   
Then under the action of $g:=c^{\frac{2s(n-1)}{d}}$ 
one must have $d$ disjoint images
$B_o,g(B_o),\ldots,g^{d-1}(B_o)$ reading {\it clockwise}, and
also $d$ disjoint images
$B_i,g(B_i),\ldots,g^{d-1}(B_i)$ reading {\it counterclockwise}.  However,
since each of the sets $g^j(B)=g^j(B_i) \sqcup g^j(B_o)$ is another
block, the {\it noncrossing condition $NCD1$} (and $d \geq 3$) 
implies that the counterclockwise ordering 
$B_i,g(B_i),\ldots,g^{d-1}(B_i)$ must actually also be {\it clockwise}.  In
other words, the partition has $d$-fold rotational symmetry.

Furthermore, since the inner boundary has $2s$ elements
and rotating it $\frac{2s(n-1)}{d}$ steps counterclockwise is the
same as rotating it $\frac{2s}{d}$ steps clockwise, one concludes
that 
$$
\frac{-2s(n-1)}{d} \equiv \frac{+2s}{d} \bmod{2s}, \qquad \text{ that is},
\frac{2sn}{d} \equiv 0 \bmod{2s}, \qquad \text{ or }d\text{ divides }n.
$$

Conversely, when $d$ divides $n$ and the partition 
has $d$-fold rotational symmetry, one can reverse the above arguments to see that it is fixed by $c^{\frac{2s(n-1)}{d}}$.
\end{proof}

We can now complete the comparison of the number of
$d$-fold symmetric elements in case (A) of 
Proposition~\ref{NCD-traversing-subtlety}
with the formula in Proposition~\ref{evaluations-prop}(iii) 
at $\hat{L}(\lambda)=0$ and $d \geq 3$
\begin{equation}
\label{last-sub-sub-case-formula}
\left( 1+(-1)^{\frac{2n}{d}} \right)
\binomial{\frac{2s(n-1)}{d^+}}{\frac{\mu(\lambda)}{d^+}}
=\begin{cases}
2\binomial{\frac{2s(n-1)}{d^+}}{\frac{\mu(\lambda)}{d^+}}& \text{ if }d\text{ divides }n,\\
0 & \text{ if }d\text{ does not divide }n.
\end{cases}.
\end{equation}
As noted at the start of the
proof of Lemma~\ref{counterclockwise-becomes-clockwise-lemma},
the assumptions of case (A) imply that
$d$ divides $2n$.  

If $d$ does not divide $n$,
Lemma~\ref{counterclockwise-becomes-clockwise-lemma}
implies that there are no $d$-fold symmetric elements,
matching the value of $0$ on the right 
in \eqref{last-sub-sub-case-formula}.

So assume that $d$ does divide $n$.
By Lemma~\ref{counterclockwise-becomes-clockwise-lemma} we must count 
the annular noncrossing partitions modelling elements in $NC^{(s)}(D_n)$
which lie in case (A) of Proposition~\ref{NCD-traversing-subtlety},
having block sizes $s\lambda$, and which are additionally
$d$-fold rotationally symmetric.  The rotational symmetry means that such
 a partition is completely determined by restricting its blocks to the 
$\frac{2s(n-1)}{d^-}$ 
outer labels $\{\pm 1,\pm 2,\ldots, \pm \frac{s(n-1)}{d^-} \}$; the 
$\frac{2s}{d^-}$ inner labels that accompany these blocks
are determined by the {\it strong $s$-divisibility condition $NCD4$}.
Since $d$ divides $m-1=2s(n-1)$,
it also divides $2s$.  Thus, setting $\tilde{s}:=\frac{s}{d^-}$,
the well-defined reduction map $\ZZ/s\ZZ \rightarrow \ZZ/\tilde{s}\ZZ$
shows that these blocks will also satisfy the
{\it strong $\tilde{s}$-divisibility condition $NCD4$}
and determine a unique element of 
$NC^{\tilde{s}}(D_n)$.  Its blocks still have sizes of the
form $sj$, of course, and the number of blocks of size $sj$ will
be $\frac{\mu_{sj}(s\lambda)}{d^-}=\frac{\mu_{j}(\lambda)}{d^-}$.
This means that this element of $NC^{\tilde{s}}(D_n)$
has block sizes 
$\tilde{s} \tilde{\lambda}$ where $\tilde{\lambda}$ only has parts
of the form $d^- \cdot j$.  Note that $d \geq 3$ forces $d^- \geq 2$, so that
$\mu_1(\tilde{\lambda})=0$ 
and $\mu_{\geq 2}(\tilde{\lambda})=\mu(\tilde{\lambda})$.  Also,
$
\mu_{d^- \cdot j}(\tilde{\lambda})=\frac{\mu_{sj}(s\lambda)}{d^-}=\frac{\mu_{j}(\lambda)}{d^-}.
$
Hence the number of such elements is counted by
the formula of Krattenthaler and M\"uller already mentioned, 
that is, the $\hat{L}(\lambda)=0$ case of
Proposition~\ref{evaluations-prop}(iii) with $d=1$, replacing $s$ with $\tilde{s}$:
$$
\begin{aligned}
&
\binomial{\tilde{s}(n-1)}{\hat{\mu}_{\geq 2}(\lambda)}
\binomial{\tilde{s}(n-1)+1-|\hat{\mu}_{\geq 2}(\lambda)|}
     {\hat{\mu}_1(\lambda)} 
+  \binomial{\tilde{s}(n-1)}{\hat{\mu}(\lambda)} \\
&=
\binomial{\tilde{s}(n-1)}{\frac{\hat{\mu}(\lambda)}{d^-}}
\binomial{\tilde{s}(n-1)+1-|\hat{\mu}_{\geq 2}(\lambda)|}
     {0} 
+  \binomial{\tilde{s}(n-1)}{\frac{\hat{\mu}(\lambda)}{d^-}} 
=2 \binomial{\tilde{s}(n-1)}{\frac{\hat{\mu}(\lambda)}{d^-}} 
=2 \binomial{\frac{2s(n-1)}{d^+}}{ \frac{\mu(\lambda)}{d^+} },
\end{aligned}
$$
which matches \eqref{last-sub-sub-case-formula} when $d$ divides $n$.

The completes the proof of Theorem~\ref{CSP-theorem}.

%%%%%%%%%%%%%%%%%%%%%%%%%%%

\section*{Acknowledgments}
%%%%%%%%%%%%%%%%%%%%%%%%%%%
The first author thanks Jang Soo Kim for helpful conversations.

\end{document}